\documentclass[12pt]{amsart}
\usepackage{tikz}
\usepackage{amssymb}
\usepackage{amscd}
\allowdisplaybreaks
\newcommand\SsquaredActionEven{S_{i}^{2} S_{j} S_{i}^{-2} & \placement{if $A_{ij}$ is even} {}\cdot{} S_{j}^{-1}}
\newcommand\SsquaredActionOdd{S_{i}^{2} S_{j} S_{i}^{-2} & \placement{if $A_{ij}$ is odd} {}\cdot{} S_{j}}
\newcommand\AdditiveRlnsInRootGroups{X_{i}(t) X_{i}(u) {}\cdot{} X_{i}(t+u)^{-1}}
\newcommand\SiSquaredActionOnXj{S_{i}^{2} X_{j}(t) S_{i}^{-2} & {}\cdot{} \Bigl(X_{j}\bigl((-1)^{A_{i j}} t\bigr)\Bigr)^{-1}}
\newcommand\SiXjRelationWhenMisTwo{[S_{i},X_{j}(t)] & \placement{if $m_{ij}={2}$} }
\newcommand\SiXjRelationWhenMisThree{S_{j} S_{i} X_{j}(t) & \placement{if $m_{ij}={3}$} {}\cdot{} \bigl(X_{i}(t) S_{j} S_{i}\bigr)^{-1}}
\newcommand\SiXjRelationWhenMisFour{[S_{i} S_{j} S_{i},X_{j}(t)] & \placement{if $m_{ij}={4}$} }
\newcommand\SiXjRelationWhenMisSix{[S_{i} S_{j} S_{i} S_{j} S_{i},X_{j}(t)] & \placement{if $m_{ij}={6}$} }
\newcommand\ChevRlnWhenMisTwo{[X_{i}(t),X_{j}(u)] &}
\newcommand\ChevRlnWhenMisThreeCloseRoots{[X_{i}(t),S_{i} X_{j}(u) S_{i}^{-1}] &}
\newcommand\ChevRlnWhenMisThreeDistantRoots{[X_{i}(t),X_{j}(u)]& {}\cdot{} S_{i} X_{j}(-t u) S_{i}^{-1}}
\newcommand\ChevRlnWhenMisFourAdjacentShortAndLong{[S_{s} X_{l}(t) S_{s}^{-1},S_{l} X_{s}(u) S_{l}^{-1}]}
\newcommand\ChevRlnWhenMisFourOrthogonalLong{[X_{l}(t),S_{s} X_{l}(u) S_{s}^{-1}]}
\newcommand\ChevRlnWhenMisFourOrthogonalShort{[X_{s}(t),S_{l} X_{s}(u) S_{l}^{-1}]& {}\cdot{} S_{s} X_{l}(2 t u) S_{s}^{-1}}
\newcommand\ChevRlnWhenMisFourDistantShortAndLong{[X_{s}(t),X_{l}(u)]& {}\cdot{} S_{s} X_{l}(-t^{2} u) S_{s}^{-1} {}\cdot{} S_{l} X_{s}(t u) S_{l}^{-1}}
\newcommand\ChevRlnWhenMisSixNearbyLong{[X_{l}(t),S_{l} S_{s} X_{l}(u) S_{s}^{-1} S_{l}^{-1}]}
\newcommand\ChevRlnWhenMisSixAdjacentShortAndLong{[S_{s} S_{l} X_{s}(t) S_{l}^{-1} S_{s}^{-1},S_{l} S_{s} X_{l}(u) S_{s}^{-1} S_{l}^{-1}]}
\newcommand\ChevRlnWhenMisSixOrthogonalShortAndLong{[S_{s} X_{l}(t) S_{s}^{-1},S_{l} X_{s}(u) S_{l}^{-1}]}
\newcommand\ChevRlnWhenMisSixDistantLong{[X_{l}(t),S_{s} X_{l}(u) S_{s}^{-1}] & {}\cdot{} S_{l} S_{s} X_{l}(-t u) S_{s}^{-1} S_{l}^{-1}}
\newcommand\ChevRlnWhenMisSixNearbyShort{[X_{s}(t),S_{s} S_{l} X_{s}(u) S_{l}^{-1} S_{s}^{-1}]& {}\cdot{} S_{s} X_{l}(-3 t u) S_{s}^{-1}}
\newcommand\ChevRlnWhenMisSixDistantShort{[X_{s}(t),S_{l} X_{s}(u) S_{l}^{-1}] & {}\cdot{} S_{l} S_{s} X_{l}(3 t u^{2}) S_{s}^{-1} S_{l}^{-1} {}\cdot{} \\ {}\cdot{} S_{s} X_{l}(3 t^{2} u) & S_{s}^{-1} {}\cdot{} S_{s} S_{l} X_{s}(2 t u) S_{l}^{-1} S_{s}^{-1}}
\newcommand\ChevRlnWhenMisSixDistantShortAndLong{[X_{s}(t),X_{l}(u)] & {}\cdot{} S_{l} S_{s} X_{l}(t^{3} u^{2}) S_{s}^{-1} S_{l}^{-1} {}\cdot{} \\ {}\cdot{} S_{s} X_{l}(-t^{3} u) S_{s}^{-1} {}\cdot{} S_{l} X_{s}(t u) & S_{l}^{-1} {}\cdot{} S_{s} S_{l} X_{s}(-t^{2} u) S_{l}^{-1} S_{s}^{-1}}
\newcommand\TorusActionOne{\htilde_{i}(r) X_{j}(t) \htilde_{i}(r)^{-1} & {}\cdot{} X_{j}\bigl(r^{A_{ij}} t\bigr)^{-1}}
\newcommand\TorusActionTwo{\htilde_{i}(r) \, S_{j} X_{j}(t) S_{j}^{-1} \, \htilde_{i}(r)^{-1} & {}\cdot{} S_{j} X_{j}\bigl(r^{-A_{ij}} t\bigr)^{-1} S_{j}^{-1}}
\newcommand\DefnOfhtilde{\htilde_{i}(r) & :=\stilde_{i}(r) \stilde_{i}(-1)}
\newcommand\DefnOfstilde{\stilde_{i}(r) & :=X_{i}(r) S_{i} X_{i}(1/r) S_{i}^{-1} X_{i}(r)}
\newcommand\WhatCollapse{S_{i} & {}\cdot{} \stilde_{i}(1)^{-1}}
\newcommand\MultInTorus{\htilde_{i}(u) \htilde_{i}(v) {}\cdot{} \htilde_{i}(u v)^{-1}}

\def\mathrlap#1{\mathchoice
{\rlap{$\displaystyle #1$}}%
{\rlap{$\textstyle #1$}}%
{\rlap{$\scriptstyle #1$}}%
{\rlap{$\scriptscriptstyle #1$}}}

\def\T{\mathcal{T}}               % algebraic torus
\def\calP{\mathcal{P}}            % relators to kill to get KMG
\def\rk{\mathop{\rm rk}\nolimits} % rank

\newcommand{\curlyU}{\mathcal{U}} % Z form of universal enveloping algebra
\newcommand{\Runits}{R^*}
\newcommand{\Q}{\mathbb{Q}}     % rational numbers
\newcommand{\Z}{\mathbb{Z}}     % integers
\newcommand{\N}{\mathbb{N}}     % natural numbers
\newcommand{\R}{\mathbb{R}}     % real numbers
\newcommand{\F}{\mathbb{F}}     % finite field
\newcommand{\C}{\mathbb{C}}     % complex numbers
\newcommand{\SL}{{\rm SL}}      % Lie group
\newcommand{\PGL}{{\rm PGL}}      % Lie group
\newcommand{\Sp}{{\rm Sp}}      % Lie group
\newcommand{\sltwo}{\mathfrak{sl}_2}
\newcommand{\slthree}{\mathfrak{sl}_3}
\newcommand{\Add}{\mathfrak{Add}} % additive group scheme
\newcommand{\St}{\mathfrak{St}} % Steinberg group
\newcommand{\StTits}{\mathfrak{St}^{\rm Tits}} % Steinberg group
\newcommand{\PSt}{\mathfrak{PSt}} % pre-Steinberg group
\newcommand{\PStTits}{\mathfrak{PSt}^{\rm Tits}} % Tits' version of pre-Steinberg group

\def\commutes{\!\rightleftarrows}
\newcommand{\g}{\mathfrak{g}}   % KMA
\newcommand{\barh}{\bar{h}}
\newcommand{\G}{\mathfrak{G}}   % KMG or various approximations to it
\newcommand{\Wstar}{W^*}        % "adjoint form" of the Z-pts of torus normalizer
\newcommand{\sstar}{s^*}        % \sstar_i are generators for \Wstar
\newcommand{\wstar}{w^*}        % element of \Wstar
\newcommand{\U}{\mathfrak{U}}   % \U_{\alpha} is additive group scheme at root a
\newcommand{\x}{\mathfrak{x}}   % \x_e is map from additive group to \U_{\alpha}
\newcommand{\odddiagram}{\Delta^{\rm odd}} % "odd" Dynkin diagram
\newcommand{\ZI}{\Z^I}          % free abelian group gend by simple roots
\newcommand{\ZIvee}{\Z^{I\vee}}  % free abelian group gend by corresponding coroots
\newcommand{\pstar}{p^*}        % generators for root stabilizer in \Wstar
\newcommand{\rstar}{r^*}        % generators for root stabilizer in \Wstar
\newcommand{\What}{\widehat{W}} % extension of \Wstar used to construct KMG
\newcommand{\what}{\hat{w}} % element  of \What

\newcommand{\stilde}{\tilde{s}}
\newcommand{\htilde}{\tilde{h}}  

\newcommand{\w}{\omega}
\newcommand{\E}{\mathcal{E}}

\newcommand{\reverse}{{\rm reverse}}
\newcommand{\tensor}{\otimes}
\newcommand{\iso}{\cong}
\newcommand{\freeproduct}{\mathop{*}}
\newcommand{\semidirect}{\rtimes}
\newcommand{\sset}{\subseteq}
\newcommand{\ad}{\mathop{\rm ad}\nolimits}
\newcommand{\Ad}{\mathop{\rm Ad}\nolimits}
\newcommand{\Aut}{\mathop{\rm Aut}\nolimits}

\newcommand{\Hom}{\mathop{\rm Hom}\nolimits}
\newcommand{\gend}[1]{\langle#1\rangle}
\newcommand{\biggend}[1]{\bigl\langle#1\bigr\rangle}
\newcommand{\pairing}[2]{\langle#1,#2\rangle}
\newcommand{\presentation}[2]{\langle#1\mid#2\rangle}
\newcommand{\set}[2]{\{#1\mathrel{|}#2\}}

\newtheorem{theorem}{Theorem}[section]
\newtheorem{lemma}[theorem]{Lemma}
\newtheorem{proposition}[theorem]{Proposition}
\newtheorem{corollary}[theorem]{Corollary}
\theoremstyle{remark}
\newtheorem{remark}[theorem]{Remark}
\newtheorem{example}[theorem]{Example}

\numberwithin{equation}{section}
\numberwithin{table}{section}

\begin{document}

\title{Steinberg groups as amalgams}
\author{Daniel Allcock}
\thanks{Supported by NSF grant DMS-1101566}
\address{Department of Mathematics\\University of Texas, Austin}
\email{allcock@math.utexas.edu}
\urladdr{http://www.math.utexas.edu/\textasciitilde allcock}
\subjclass[2000]{%
Primary: 19C99% Steinberg groups and K2, "none of the above but in this section"
; Secondary: 20G44%     Kac-Moody groups
, 14L15%   Group Schemes
}
\date{March 29, 2016}
%\date{February 20, 2015}
%\date{September 23, 2014}
%\date{August 8, 2014}
%\date{May 17, 2014}
%\date{April 23, 2013}

\begin{abstract}
For any root system and any commutative ring we give a relatively
simple presentation of a group related to its Steinberg group $\St$.
This includes the case of infinite root systems used in Kac-Moody
theory, for which the Steinberg group was defined by Tits and
Morita-Rehmann.  In most cases our group equals $\St$, giving a
presentation with many advantages over the usual presentation of
$\St$.  This equality holds for all spherical root systems, all
irreducible affine root systems of rank${}>2$, and all $3$-spherical
root systems.  When the coefficient ring satisfies a minor condition,
the last condition can be relaxed to $2$-sphericity.

Our presentation is defined in terms of the Dynkin diagram rather than
the full root system.  It is concrete, with no implicit coefficients
or signs.  
It makes manifest the
exceptional diagram automorphisms in characteristics $2$ and~$3$, and
their generalizations to Kac-Moody groups.
And it is a Curtis-Tits style presentation: it is the direct
limit of the  groups coming from $1$- and $2$-node
subdiagrams of the Dynkin diagram.  Over non-fields this description
as a direct limit is new and surprising.
Our main application is that many Steinberg and Kac-Moody groups over
finitely-generated rings are finitely presented.
\end{abstract}
\maketitle

\section{Introduction}
\label{sec-introduction}

\noindent
In this paper we give a presentation for a Steinberg-like group, over
any commutative ring, for any root system, finite or not.  For many
root systems, including all finite ones, it is the same as the
Steinberg group~$\St$.  This is the case of interest, for then it
gives a new presentation of~$\St$ and associated Chevalley and Kac-Moody groups.
Our presentation 
\begin{enumerate}
\item
\label{virtue-simple-roots-only}
is defined in terms of the Dynkin diagram rather
than the set of all (real) roots (sections \ref{sec-examples} and~\ref{sec-pre-Steinberg-group});
\item
\label{virtue-concreteness}
is concrete, with no coefficients or signs left implicit;
\item 
\label{virtue-Curtis-Tits}
generalizes the Curtis-Tits
presentation of Chevalley groups to rings other than fields (corollary~\ref{cor-pre-Steinberg-as-direct-limit});
\item
\label{virtue-finite-presentation}
is rewritable as a finite presentation when $R$ is finitely generated as an abelian group
(theorem~\ref{thm-finite-presentation-of-pre-Steinberg-groups});
\item
\label{virtue-finite-presentability-1}
is often rewritable as a finite presentation when $R$ is
merely finitely generated as a ring (theorem~\ref{thm-finite-presentation-of-pre-Steinberg-groups});
\item
\label{virtue-finite-presentability-2}
allows one to prove that many Kac-Moody groups are finitely presented
(theorem~\ref{thm-finite-presentation-of-Kac-Moody-groups}); and
\item
\label{virtue-manifest-symmetry}
makes manifest the exceptional diagram automorphisms that lead to
the Suzuki and Ree groups, and allows one to construct similar
automorphisms of Kac-Moody groups in characteristic $2$ or~$3$
(section~\ref{sec-diagram-automorphisms}).
\end{enumerate}

More precisely, given any generalized Cartan matrix~$A$, in
section~\ref{sec-pre-Steinberg-group} we give two definitions of a
new group functor.  We call it the pre-Steinberg group $\PSt_A$
because it has a natural map to $\St_A$.  This will be obvious from
the first definition, which mimics Tits' definition \cite{Tits} of the
Steinberg group $\St_A$, as refined by Morita-Rehmann
\cite{Morita-Rehmann}.  The difference is that we leave out most of the relations.
If the root system is finite then both
$\PSt_A$ and $\St_A$ coincide with Steinberg's original group functor,
so they coincide with each other too.  Our perspective is that
$\PSt_A(R)$ is interesting if and only if $\PSt_A(R)\to\St_A(R)$ is an
isomorphism, when our second definition of $\PSt_A$ provides a new and
useful presentation of $\St_A$.

We will discuss this second definition after listing some cases in
which $\PSt_A(R)\iso\St_A(R)$.  As just mentioned, case
\ref{item-isomorphism-for-finite-dimensional-type} in the next theorem
is obvious once $\PSt$ is defined.  Cases
\ref{item-isomorphism-for-3-spherical}--\ref{item-isomorphism-for-2-spherical}
are proven in section~\ref{sec-PSt-to-St}.  By considering the list of
affine Dynkin diagrams, one sees that these cases imply case
\ref{item-isomorphism-for-irreducible-affine} except in rank~$3$ when
$R$ has a forbidden $\F_2$ or $\F_3$ quotient.  Proving
\ref{item-isomorphism-for-irreducible-affine} requires removing this
restriction on $R$, for which we refer to
\cite{Allcock-affine-Kac-Moody-groups}.

\begin{theorem}[Coincidence of Steinberg and pre-Steinberg groups]
\label{thm-examples-of-PSt-to-St-being-an-isomorphism}
Suppose $R$ is a commutative ring and $A$ is a generalized Cartan
matrix.  Then the natural map $\PSt_A(R)\to\St_A(R)$ is an isomorphism
in any of the following cases:
\begin{enumerate}
\item
\label{item-isomorphism-for-finite-dimensional-type}
if $A$ is spherical; or
\item
\label{item-isomorphism-for-irreducible-affine}
if $A$ is irreducible affine of rank${}>2$; or
\item
\label{item-isomorphism-for-3-spherical}
if $A$ is $3$-spherical; or
\item
\label{item-isomorphism-for-2-spherical}
if $A$ is $2$-spherical and (if $A$ has a multiple bond) $R$ has no
quotient $\F_2$ and (if $A$ has a triple bond) $R$ has no quotient $\F_3$.
\end{enumerate}
\end{theorem}

\noindent
(Language: we pass between Cartan matrices and Dynkin diagrams
whenever convenient.  The rank $\rk A$ of $A$ means the number of
nodes of the Dynkin diagram.  $A$ is called spherical if its Weyl
group is finite; this is equivalent to every component of the Dynkin
diagram being one of the classical $ABCDEFG$ diagrams.  $A$ is called
$k$-spherical if every subdiagram with${}\leq k$ nodes is spherical.)

\medskip
Our second ``definition'' of $\PSt_A(R)$ is the following theorem,
giving a presentation for it.  It is a restatement of
theorem~\ref{thm-G-4-is-PSt}, whose proof occupies
sections~\ref{sec-pre-Steinberg-group}--\ref{sec-step-3}.  The proof
relies on an understanding of root stabilizers under a certain
extension of the Weyl group, which appears to be a new ingredient in
Lie theory.  To give the flavor of the result, the full presentation
appears in table~\ref{tab-simply-laced-presentation} if $A$ is
simply-laced without $A_1$ components.  In this case we have
$\PSt_A(R)=\St_A(R)$ by the previous theorem, so we get a new
presentation for $\St_A(R)$.

\def\FOOalign#1#2{#1=\mathrlap{#2}\kern120pt}
\def\FOOtag#1{\rlap{\rm #1}\kern90pt}
\begin{table}
\begin{align*}
\left.
\begin{aligned}
\FOOalign{X_i(t)X_i(u)}{X_i(t+u)}\\
\FOOalign{[S_i^2,X_i(t)]}{1}\\
\FOOalign{S_i}{X_i(1) S_i X_i(1) S_i^{-1} X_i(1)}
\end{aligned}
\right\}\FOOtag{all $i$}
\\
\noalign{\medskip}
\left.
\begin{aligned}
\FOOalign{S_i S_j}{S_j S_i}\\
\FOOalign{[S_i,X_j(t)]}{1}\\
\FOOalign{[X_i(t),X_j(u)]}{1}
\end{aligned}
\right\}
\FOOtag{all unjoined $i\neq j$}
\\
\noalign{\medskip}
\left.
\begin{aligned}
\FOOalign{S_i S_j S_i}{S_j S_i S_j}\\
\FOOalign{S_i^2 S_j S_i^{-2}}{S_j^{-1}}\\
\FOOalign{X_i(t) S_j S_i}{S_j S_i X_j(t)}\\
\FOOalign{S_i^2 X_j(t) S_i^{-2}}{X_j(t)^{-1}}\\
\FOOalign{[X_i(t),S_i X_j(u) S_i^{-1}]}{1}\\
\FOOalign{[X_i(t),X_j(u)]}{S_i X_j(t u) S_i^{-1}}
\end{aligned}
\right\}
\FOOtag{all joined $i\neq j$}
\end{align*}
\caption{Our defining relations for the Steinberg group $\St_A(R)$,
  when $A$ is any simply-laced generalized Cartan matrix, without $A_1$ components, and $R$ is any
  commutative ring.  The generators are $X_i(t)$ and $S_i$
  where $i$ varies over the nodes of the Dynkin diagram and $t$
  over~$R$.}
\label{tab-simply-laced-presentation}
\end{table}

\begin{theorem}[Presentation of pre-Steinberg groups]
\label{thm-definition-of-PSt-as-presentation}
For any commutative ring~$R$ and any generalized Cartan matrix $A$, 
$\PSt_A(R)$ has a presentation with generators $S_i$ and $X_i(t)$, where
$i$ varies over the simple roots and $t$ varies over $R$, and relators
\eqref{eq-def-of-What-Artin-relators}--\eqref{eq-collapse-to-PSt}.
\end{theorem}

\noindent
Table~\ref{tab-simply-laced-presentation} 
shows that the presentation is less intimidating than a
list of $26$ relations would suggest.   See section~\ref{sec-examples} for the
$B_2$ and $G_2$ cases.  Each 
relator \eqref{eq-def-of-What-Artin-relators}--\eqref{eq-collapse-to-PSt}
involves at most two distinct subscripts.
This proves the following.

\begin{corollary}[Curtis-Tits presentation for pre-Steinberg groups]
\label{cor-pre-Steinberg-as-direct-limit}
Let $A$ be a generalized Cartan matrix and $R$ a commutative ring.
Consider the groups $\PSt_B(R)$ and the obvious maps between them, as
$B$ varies over the $1\times1$ and $2\times2$ submatrices of $A$
coming from singletons and pairs of nodes of the Dynkin diagram.  
The direct limit of this family of groups equals $\PSt_A(R)$.  \qed
\end{corollary}

In any of the cases in
theorem~\ref{thm-examples-of-PSt-to-St-being-an-isomorphism}, we may
replace $\PSt_A$ by $\St_A$ everywhere in
corollary~\ref{cor-pre-Steinberg-as-direct-limit}, yielding a
Curtis-Tits style presentation for $\St_A$.  This is the source of our
title {\it Steinberg groups as amalgams}.  We learned after writing
this paper that Dennis and Stein \cite[Theorem
  B]{Dennis-Stein-injective-stability} announced corollary~\ref{cor-pre-Steinberg-as-direct-limit} for
finite root systems.  They did not publish a proof, and from their
announcement it appears that their approach was not
via our theorem~\ref{thm-definition-of-PSt-as-presentation}.

\medskip
In the $A_1$, $A_2$, $B_2$ and $G_2$ cases we write out our
presentation of $\PSt_A(R)=\St_A(R)$ explicitly in
section~\ref{sec-examples}.  We do this to make our results as
accessible as possible, and to show in
section~\ref{sec-diagram-automorphisms} that our presentation makes
manifest the exceptional
diagram automorphisms in characteristics $2$ and~$3$.
Namely, the arrow-reversing diagram automorphism of the $B_2$ or $G_2$
Dynkin diagram yields a self-homomorphism of the corresponding
Steinberg group if the coefficient ring~$R$ has characteristic $2$
or~$3$ respectively.   If $R$ is a perfect field then this self-homomorphism is the
famous outer automorphism that leads to the Suzuki and (small) Ree
groups.

Because of the direct limit property
(corollary~\ref{cor-pre-Steinberg-as-direct-limit}), one obtains the
corresponding self-homomorphisms of $F_4$ in characteristic~$2$ with
no more work.  That is, the defining relations for $\St_{F_4}$ are
those for $\St_{B_2}$, two copies of $\St_{A_2}$ and three copies of
$\St_{A_1^2}=\St_{A_1}\times\St_{A_1}$.  The diagram automorphism
transforms the $B_2$ relations as in the previous paragraph and
sends the other relations into each other.  The same argument
applies to many Kac-Moody groups.  By work of H\'ee, this leads to
Kac-Moody-like analogues of the Suzuki and Ree groups, discussed
briefly in section~\ref{sec-diagram-automorphisms}.

\medskip
An application of the theory we have described is that Steinberg
groups and Kac--Moody groups are finitely presented under quite weak
hypotheses on their Dynkin diagrams and coefficient rings.
We state
the Steinberg group result in terms of $\PSt_A(R)$, keeping in mind
that the interesting case is when $\PSt_A(R)$ coincides with
$\St_A(R)$.  See section~\ref{sec-finite-presentations} for the proof.

\begin{theorem}[Finite presentation of pre-Steinberg groups]
\label{thm-finite-presentation-of-pre-Steinberg-groups}
Let $R$ be a commutative ring and $A$ a generalized Cartan matrix.
Then $\PSt_A(R)$ is finitely presented in any of the following cases:
\begin{enumerate}
\item
\label{item-finitely-generated-abelian-group-makes-St-finitely-presented}
if $R$ is finitely generated as an abelian group; or
\item
\label{item-P-S-t-is-f-p-rank-2}
if $A$ is $2$-spherical without $A_1$ components, and $R$ is finitely
generated as a module over a subring generated by finitely many units;
or
\item
\label{item-P-S-t-is-f-p-rank-3}
if $R$ is finitely generated as a ring, and any two nodes of $A$ lie in an irreducible spherical diagram of
rank${}\geq3$.
\end{enumerate}
\end{theorem}

Many authors have studied the finite presentation of Steinberg groups
and related groups.  Our theorem~\ref{thm-finite-presentation-of-pre-Steinberg-groups} is inspired by work of
Splitthoff \cite{Splitthoff}.  See \cite{Kiralis}, \cite{Zhang} and \cite{Li} for
some additional results.

The Kac--Moody group version of theorem~\ref{thm-finite-presentation-of-pre-Steinberg-groups} concerns the group functors $\G_D$
constructed by Tits in \cite{Tits} (he wrote $\widetilde{\G}_D$).  They
were his motivation for generalizing the Steinberg groups beyond the
case of spherical Dynkin diagrams.  He defined the ``simply
connected'' Kac--Moody groups as certain quotients of Steinberg
groups, and arbitrary Kac--Moody groups are only slightly more
general.  Specifying a Kac--Moody group requires specifying a root
datum $D$, which is slightly more refined information than $D$'s
associated generalized Cartan matrix~$A$.  But the choice of $D$
doesn't affect any of our results.  

Our final theorem shows that a great many Kac--Moody groups over
rings are finitely presented.  This is surprising because one thinks
of Kac--Moody groups over (say) $\R$ as infinite-dimensional Lie groups, so
the same groups over (say) $\Z$ should be some sort of discrete subgroups.
There is no obvious reason why a discrete subgroup of an
infinite-dimensional Lie group should be finitely presented.
See section~\ref{sec-finite-presentations}
for the definition of the Kac--Moody groups, and the proof of the
following theorem.

\begin{theorem}[Finite presentation of Kac--Moody groups]
\label{thm-finite-presentation-of-Kac-Moody-groups}
Suppose $A$ is a generalized Cartan matrix and $R$ is a commutative
ring whose group of units $\Runits$ is finitely generated.  Let $D$ be any
root datum with generalized Cartan matrix $A$.  Then Tits'
Kac--Moody group $\G_D(R)$ is finitely presented if $\St_A(R)$ is.  

In particular, this holds if one of
\ref{item-isomorphism-for-finite-dimensional-type}--\ref{item-isomorphism-for-2-spherical}
from
theorem~\ref{thm-examples-of-PSt-to-St-being-an-isomorphism}
holds and one of
\ref{item-finitely-generated-abelian-group-makes-St-finitely-presented}--\ref{item-P-S-t-is-f-p-rank-3}
from
theorem~\ref{thm-finite-presentation-of-pre-Steinberg-groups}
holds.
\end{theorem}

\medskip
The paper is organized as follows.  Sections \ref{sec-examples}
and~\ref{sec-diagram-automorphisms} are expository and not essential
for later sections.  Section~\ref{sec-examples} is really a
continuation of the introduction, writing down the essential cases of
our presentation of $\PSt_A(R)$.  These can be understood
independently of the rest of the paper.
Section~\ref{sec-diagram-automorphisms} treats the exceptional diagram
automorphisms: their existence is hardly even an exercise.

Sections \ref{sec-Kac-Moody-algebra}--\ref{sec-Steinberg-group} give
necessary background.  Section~\ref{sec-Kac-Moody-algebra} gives a
little background on the Kac-Moody algebra $\g_A$.
Section~\ref{sec-W-star} is mostly a review of results of Tits about a
certain extension $\Wstar\sset\Aut(\g_A)$ of the Weyl group $W$.  But
we also use a more recent result of Brink \cite{Brink} on Coxeter
groups to describe generators for root stabilizers in $\Wstar$, and how
they act on the corresponding root spaces
(theorem~\ref{thm-root-stabilizers-in-W-star}).
Section~\ref{sec-Steinberg-group} reviews Tits' definition of $\St_A$
and its refinement by Morita-Rehmann.

Sections \ref{sec-pre-Steinberg-group}--\ref{sec-step-3} are the
technical heart of the paper, establishing theorem~\ref{thm-definition-of-PSt-as-presentation}.  In
section~\ref{sec-pre-Steinberg-group} we define $\PSt_A$ and then
establish a presentation for it.  
We
do this by defining a group functor $\G_4$ by a presentation and
proving $\PSt_A\iso\G_4$.  As the notation suggests, this is the last
in a chain of group functors $\G_1,\dots,\G_4$ that give successively
better approximations to $\PSt_A$.  Lemma~\ref{lem-description-of-G-1} and theorems \ref{thm-G-2-is-free-product-semidirect-What},
\ref{thm-G3-is-spherical-Tits-Steinberg-group-semidirect-What}
and~\ref{thm-G-4-is-PSt} give ``intrinsic'' descriptions of $\G_1$, $\G_2$,
$\G_3$ and $\G_4$, the last one being the same as
theorem~\ref{thm-definition-of-PSt-as-presentation} above.  See  section~\ref{sec-examples} for a quick overview of the meanings of these
intermediate groups.  The proof for $\G_1$ is trivial, the proofs
for $\G_2$ and $\G_3$ occupy sections
\ref{sec-proof-of-free-by-semidirect-product-theorem}
and \ref{sec-step-3}, and the proof for $\G_4$ appears in
section~\ref{sec-pre-Steinberg-group}.

Section~\ref{sec-adjoint-representation} reviews work of R\'emy
\cite{Remy} on the adjoint representation of a Kac--Moody group,
regarded as a representation of the corresponding Steinberg group.
The definition of $\St$ is as the direct limit of a family of unipotent
groups, and we use the adjoint representation to show that the natural
maps from these groups to $\St$ are embeddings.  This is necessary for the proof of
theorem~\ref{thm-examples-of-PSt-to-St-being-an-isomorphism}
in section~\ref{sec-PSt-to-St}.  Finally, in
section~\ref{sec-finite-presentations} we discuss finite
presentability of pre-Steinberg groups and Kac--Moody groups.
In particular we prove
theorems \ref{thm-finite-presentation-of-pre-Steinberg-groups}
and~\ref{thm-finite-presentation-of-Kac-Moody-groups}.  The result for
pre-Steinberg groups relies heavily on work of Splitthoff.

The author is very grateful to the Japan Society for the Promotion of
Science and to Kyoto University, for their support and hospitality,
and to Lisa Carbone, for getting him interested in Kac-Moody groups
over~$\Z$.

\section{Examples}
\label{sec-examples}

\noindent
In this section we give our presentation of $\PSt_A(R)=\St_A(R)$ when
$R$ is a commutative ring and $A=A_1$, $A_2$, $B_2$ or $G_2$.  It is
mostly a writing-out of the general construction in
section~\ref{sec-pre-Steinberg-group}.  Because of the direct limit
property of the pre-Steinberg group
(corollary~\ref{cor-pre-Steinberg-as-direct-limit}), understanding these
cases, together with $\PSt_{A_1^2}=\PSt_{A_1}\times\PSt_{A_1}$, is
enough to present $\PSt_A$ whenever $A$ is $2$-spherical.  As usual,
we are mainly interested in the presentation when $\PSt$ and $\St$
coincide.  This happens  in any of the
cases of theorem~\ref{thm-examples-of-PSt-to-St-being-an-isomorphism}.

For generators we take formal symbols $S$, $S'$, $X(t)$ and $X'(t)$
with~$t$ varying over~$R$.  The primed generators should be omitted in
the $A_1$ case.   We divide the relations into batches $0$
through~$4$, with several intermediate groups having useful descriptions.
At the end of the section we give an overview of these descriptions. 
For now we make only brief remarks.
The batch~$0$ relations make the $S$'s generate something like the
Weyl group. The batch~$1$
relations make the $X(t)$'s additive in~$t$.  The batch~$2$
relations describe the interaction between the $S$'s and the $X(t)$'s.
These are the essentially new component of our approach to Steinberg groups.
The batch~$3$ relations are Chevalley relations, describing
commutators of conjugates of the $X(t)$'s by various words in
the $S$'s.  Finally, the batch~$4$ relations are Steinberg's $A_1$-specific
relations, and relations identifying the $S$'s
with the generators of the
``Weyl group'' inside the Steinberg group.

In the presentations 
we write $x\commutes y$ to indicate that $x$ and $y$ commute.
The notation ``(\& primed)'' next to a relation means to also impose
the relation got from it by the typographical substitution
$S\leftrightarrow S'$ and
$X(t)\leftrightarrow X'(t)$.

\def\andprimed{\rlap{\kern120pt(\& primed)}}
\def\andprimedperiod{\rlap{\kern120pt(\& primed).}}
\def\batchAndLHS#1#2#3{\llap{\hbox to #1 pt{Batch #2:\hfil}}\llap{$#3$}}

\begin{example}[$A_1$]
We take generators $S$ and $X(t)$ with $t$ varying over~$R$.
There are no batch~$0$ or batch~$3$ relations.
\begin{align}
\batchAndLHS{150}{1}{X(t) X(u)}
&{}= X(t+u)
\\
\batchAndLHS{150}{2}{S^2}
&{}\commutes X(t)
\\
\batchAndLHS{150}{4}{S}
&{}=\stilde(1)
\\
\label{eq-action-of-h-on-X-A1-example}
\htilde(r)\cdot X(t)\cdot\htilde(r)^{-1}
&{}=
X(r^2t)
\\
\htilde(r)\cdot S X(t) S^{-1}\cdot\htilde(r)^{-1}
&{}=S X(r^{-2}t)S^{-1}
\end{align}
for all $t,u\in R$ and all~$r$ in the unit group $\Runits$ of~$R$, 
where $\stilde(r):=X(r)\cdot S X(1/r) S^{-1}\cdot X(r)$ and
$\htilde(r):=\stilde(r)\stilde(-1)$.
This is essentially Steinberg's original
presentation (the group $G'$ on p.~78 of \cite{Steinberg}), with a slightly different generating set.
\end{example}

\begin{example}[$A_2$]
We take generators $S$, $S'$, $X(t)$ and $X'(t)$ with $t$ varying
over~$R$.
\begin{align}
\batchAndLHS{120}{0}{S S' S}
&{}=S' S S'
\\
S^2\cdot S'\cdot S^{-2}&{}=\andprimed S'^{-1}
\\
\batchAndLHS{120}{1}{X(t)X(u)}
&{}=\andprimed X(t+u)
\\
\batchAndLHS{120}{2}{S^2}
&{}\commutes\andprimed X(t)
\\
S^2\cdot X'(t)\cdot S^{-2}&{}=\andprimed X'(-t)
\\
S S' X(t)&{}=\andprimed X'(t) S S'
\\
\batchAndLHS{120}{3}{[X(t),X'(u)]}
&{}=\andprimed S X'(t u) S^{-1}
\\
X(t)&\commutes\andprimed S X'(u)S^{-1}
\\
\batchAndLHS{120}{4}{S}
&{}=\andprimedperiod X(1)S X(1)S^{-1} X(1)
\end{align}
As before, $t$ and $u$ vary over $R$.  The diagram
automorphism is given by
$S\leftrightarrow S'$ and $X(t)\leftrightarrow X'(t)$.
\end{example}

\begin{example}[$B_2$]
We take generators $S$, $S'$, $X(t)$ and $X'(t)$ with $t$ varying over $R$. 
Unprimed letters correspond to the short simple root and primed
letters to the long one.  
\begin{align}
\batchAndLHS{120}{0}{S S' S S'}
&{}=S' S S' S
\\
S^2&{}\commutes S'
\\
S'^2\cdot S\cdot S'^{-2}&{}= S^{-1}
\\
\batchAndLHS{120}{1}{X(t) X(u)}
&{}=\andprimed X(t+u)
\\
\batchAndLHS{120}{2}{S^2}
&{}\commutes\andprimed X(t)
\\
S^2&{}\commutes X'(t)
\\
S'^2\cdot X(t)\cdot S'^{-2}&{}=X(-t)
\\
S S' S&{}\commutes\andprimed X'(t)
\\
\batchAndLHS{120}{3}{S X'(t) S^{-1}
}&{}\commutes S' X(u) S'^{-1}
\\
\label{eq-orthogonal-long-roots-for-B2-in-example-section}
X'(t)&{}\commutes S X'(u) S^{-1}
\\
\label{eq-orthogonal-short-roots-for-B2-in-example-section}
[X(t),S' X(u) S'^{-1}]&{}= S X'(-2t u) S^{-1}
\\
[X(t),X'(u)]&{}=
S' X(-t u)S'^{-1}
\cdot
S X'(t^2u) S^{-1}
\\
\label{eq-collapse-relation-for-B2-in-example-section}
\batchAndLHS{120}{4}{S}
&{}=\andprimedperiod X(1)S X(1)S^{-1} X(1)
\end{align}
\end{example}

\def\andprimed{\rlap{\kern120pt(\& primed)}}
\def\andprimedperiod{\rlap{\kern120pt(\& primed).}}

\begin{example}[$G_2$]
We take generators $S$, $S'$, $X(t)$ and $X'(t)$ as in the $B_2$ case.  
\begin{align}
\batchAndLHS{120}{0}{S S' S S' S S'}
&{}=S' S S' S S' S
\\
S^2\cdot S'\cdot S^{-2}&{}=\andprimed S'^{-1}
\\
\batchAndLHS{120}{1}{X(t) X(u)}
&{}=\andprimed X(t+u)
\\
\batchAndLHS{120}{2}{S^2}
&{}\commutes\andprimed X(t)
\\
S^2\cdot X'(t)\cdot S^{-2}&{}=\andprimed X'(-t)
\\
S S' S S' S&{}\commutes\andprimed X'(t)
\\
\label{eq-adjacent-long-roots-of-G2-in-examples-section}
\batchAndLHS{120}{3}{X'(t)}
&{}\commutes S'S X'(u) S^{-1}S'^{-1}
\\
S S' X(t)S'^{-1}S^{-1}&{}\commutes S'S X'(u)S^{-1}S'^{-1}
\\
S X'(t) S^{-1}&{}\commutes S' X(u)S'^{-1}
\\
\label{eq-distant-long-roots-of-G2-in-examples-section}
[X'(t),S X'(u) S^{-1}]&
{}=S'S X'(t u)S^{-1}S'^{-1}
\\
\label{eq-adjacent-short-roots-of-G2-in-examples-section}
\llap{$[X(t),S S' X(u)S'^{-1}S^{-1}]$}
&{}=S X'(3t u) S^{-1}
\\
\label{eq-distant-short-roots-of-G2-in-examples-section}
[X(t),S' X(u)S'^{-1}]
&{}=S S' X(-2t u)S'^{-1}S^{-1}
\\
&\phantom{{}={}}\cdot
S X'(-3t^2u) S^{-1}\cdot
\notag\\
&\phantom{{}={}}\cdot
S'S X'(-3t u^2)S^{-1}S'^{-1}
\notag\\
% commutator of simple root groups
[X(t),X'(u)]
&{}=
S S' X(t^2u)S'^{-1}S^{-1}
\\
&\phantom{{}={}}\cdot
S' X(-t u)S'^{-1}
\notag\\
&\phantom{{}={}}\cdot
S X'(t^3u) S^{-1}
\notag\\
&\phantom{{}={}}\cdot
S'S X'(-t^3u^2) S^{-1}S'^{-1}
\notag\\
\batchAndLHS{120}{4}{S}
&{}=\andprimedperiod X(1)S X(1)S^{-1} X(1)
\end{align}
\end{example}

Now we explain the meaning of the batches.  The group with generators
$S$ and $S'$, modulo the batch~$0$ relations, is what we call $\What$
in section~\ref{sec-pre-Steinberg-group}.  It is an extension of the
Weyl group~$W$, slightly ``more extended'' than a better-known
extension of $W$ introduced by Tits \cite{Tits-Normalisateurs}.  We
write $\Wstar$ for Tits' extension and discuss it in
section~\ref{sec-W-star}.  ``More extended'' means that $\What\to W$
factors through $\Wstar$.  The kernel of $\Wstar\to W$ is an
elementary abelian $2$-group, while the kernel of $\What\to W$ can be
infinite and nilpotent of class~$2$.  These details are not needed for
a general understanding.

The group with generators $X(t)$ and $X'(t)$, modulo the batch~$1$
relations, is what we call $\G_1(R)$ in
section~\ref{sec-pre-Steinberg-group}.  It is just a free product of
copies of the additive group of $R$, one for each simple root.

The group generated by $S$, $S'$ and the $X(t)$ and $X'(t)$, modulo
the relations from batches $0$ through~$2$, is what we call $\G_2(R)$
in section~\ref{sec-pre-Steinberg-group}.  It is isomorphic to
$\bigl(\freeproduct_{\alpha\in\Phi} R\bigr)\semidirect\What$ by
theorem~\ref{thm-G-2-is-free-product-semidirect-What}, where $\Phi$ is
the set of all roots.  In fact this theorem applies to any generalized
Cartan matrix~$A$.  This is the main technical result of the paper,
and the batch~$2$ relations are the main new ingredient in our
treatment of the Steinberg groups.  Furthermore,
theorem~\ref{thm-G-2-is-free-product-semidirect-What} generalizes to
groups with a root group datum in the sense of
\cite{Tits-twin-buildings}\cite{Caprace-Remy}; see remark~\ref{remark-groups-with-a-root-group-datum}.
This should lead to generalizations of our results with such groups in
place of Kac-Moody groups.

The batch~$3$ relations are a few of the Chevalley relations, written
in a manner due to Demazure; see section~\ref{sec-pre-Steinberg-group}
for discussion and references.  No batch~$3$ relations are present in
the $A_1$ case.  In the $A_2$, $B_2$ and $G_2$ cases, adjoining them
yields $\St(R)\semidirect\What$, by
theorem~\ref{thm-G3-is-spherical-Tits-Steinberg-group-semidirect-What}.
For any generalized Cartan matrix $A$, the corresponding presentation is
called $\G_3(R)$ in section~\ref{sec-pre-Steinberg-group}, and
theorem~\ref{thm-G3-is-spherical-Tits-Steinberg-group-semidirect-What}
asserts that it is isomorphic to $\PStTits(R)\semidirect\What$.
Here $\PStTits$ is the ``pre-'' version of Tits' version of the
Steinberg group. See section~\ref{sec-pre-Steinberg-group} for more
details.

Adjoining the batch~$4$ relations yields the group called $\G_4(R)$ in
section~\ref{sec-pre-Steinberg-group}.  In all four examples this
coincides with $\St_A(R)$.  This result is really the concatenation of
theorem~\ref{thm-G-4-is-PSt}, that $\G_4$ equals $\PSt_A$ (for any~$A$), with the
isomorphism $\PSt_A=\St_A$ when $A$ is spherical.

\section{Diagram automorphisms}
\label{sec-diagram-automorphisms}

\noindent
In this section we specialize our presentations of $\St_{B_2}(R)$ and
$\St_{G_2}(R)$ when the ground ring $R$ has characteristic $2$ or~$3$
respectively.  The exceptional diagram automorphisms are then visible.
These results are not needed later in the paper.

\medskip
We begin with the $B_2$ case, so assume $2=0$ in $R$.  Then
$X(t)=X(-t)$ for all $t$.  In particular, the right side of \eqref{eq-collapse-relation-for-B2-in-example-section}
is its own inverse, so $S$ and $S'$ have order~$2$.  The relations
involving $S^2$ or $S'^2$ are therefore trivial and may be omitted.
Also, the right side of \eqref{eq-orthogonal-short-roots-for-B2-in-example-section} is the
identity, so that \eqref{eq-orthogonal-short-roots-for-B2-in-example-section} is the primed version of \eqref{eq-orthogonal-long-roots-for-B2-in-example-section}.
In summary, the defining relations for $\St$ are now the following,
with $t$ and $u$ varying over $R$.
\def\andprimed{\rlap{\kern150pt(\& primed)}}
\def\andprimedperiod{\rlap{\kern150pt(\& primed).}}
\begin{align}
\label{eq-B2-Artin-relations-automorphisms-section}
S S' S S'&{}=S' S S' S
\\
X(t) X(u) &{}=\andprimed X(t+u)
\\
S S' S&{}\commutes\andprimed X'(t)
\\
\label{eq-B2-close-short-and-long-roots-automorphisms-section}
S X'(t) S^{-1}&{}\commutes S' X(u) S'^{-1}
\\
X'(t)&{}\commutes\andprimed S X'(u) S^{-1}
\\
\label{eq-B2-simple-roots-automorphisms-section}
[X(t),X'(u)]&{}=
S' X(-t u)S'^{-1}
\cdot
S X'(t^2u) S^{-1}
\\
\label{eq-B2-collapse-relation-automorphisms-section}
S&{}=\andprimedperiod X(1)S X(1)S^{-1} X(1)
\end{align}

\begin{theorem}
\label{thm-diagram-automorphisms-in-characteristic-2}
Suppose $R$ is a ring of characteristic~$2$.  Then the map
$S\leftrightarrow S'$, $X'(t)\mapsto X(t)\mapsto X'(t^2)$ extends to
an endomorphism $\phi$ of $\St_{B_2}(R)$.  If $R$ is a perfect field
then $\phi$ is an automorphism.
\end{theorem}

\begin{proof}
One must check that each relation
\eqref{eq-B2-Artin-relations-automorphisms-section}--\eqref{eq-B2-collapse-relation-automorphisms-section}
remains true after the substitution $S\leftrightarrow S'$,
$X'(t)\mapsto X(t)\mapsto X'(t^2)$.  It is easy to check that every
relation maps to its primed form (except that some $t$'s and $u$'s are
replaced by their squares).  The relations \eqref{eq-B2-Artin-relations-automorphisms-section}, \eqref{eq-B2-close-short-and-long-roots-automorphisms-section} and
\eqref{eq-B2-simple-roots-automorphisms-section} are their own primed forms.  Only 
\eqref{eq-B2-simple-roots-automorphisms-section} deserves any comment: we must check the identity
$$
[X'(t^2),X(u)]
=
S X'(t^2u^2) S^{-1}\cdot S' X(t^2u)S'^{-1}
$$
in $\St$.  The left side equals $[X(u),X'(t^2)]^{-1}$.  
The identity follows by expanding the commutator using
\eqref{eq-B2-simple-roots-automorphisms-section}.

Now suppose $R$ is a perfect field.  By a similar argument, one can check that there is an
endomorphism $\psi$ of $\St$ that fixes  $S$ and $S'$, and for
each $t\in R$ sends
$X(t)$ to $X(\sqrt{t})$ and $X'(t)$ to $X'(\sqrt{t})$.  (Because $R$
is a perfect field of characteristic~$2$, square roots
exist and are unique, and $t\mapsto\sqrt{t}$ is a field automorphism.)
Since $\psi\circ\phi\circ\phi$ sends each generator to itself, $\phi$
and $\psi$ must
be isomorphisms.
\end{proof}

Now we consider the $G_2$ case, so suppose $3=0$ in~$R$.  The
main simplifications of section~\ref{sec-examples}'s presentation of $\St$
are that the right side of \eqref{eq-adjacent-short-roots-of-G2-in-examples-section} is the identity, so \eqref{eq-adjacent-short-roots-of-G2-in-examples-section} is
the primed version of \eqref{eq-adjacent-long-roots-of-G2-in-examples-section}, and that the last two terms on the
right of \eqref{eq-distant-short-roots-of-G2-in-examples-section} are trivial, so that \eqref{eq-distant-short-roots-of-G2-in-examples-section} is the primed version
of \eqref{eq-distant-long-roots-of-G2-in-examples-section}.  So the
relations simplify to
\def\andprimed{\rlap{\kern120pt(\& primed)}}
\def\andprimedperiod{\rlap{\kern120pt(\& primed).}}
\begin{align}
S S' S S' S S' ={}&S' S S' S S' S
\\
\label{eq-S-squared-action-on-S-prime-G2-automorphisms-section}
S^2\cdot S'\cdot S^{-2}={}&\andprimed S'^{-1}
\\
X(t) X(u) ={}&\andprimed X(t+u)
\\
\label{eq-S-squared-action-on-X-G2-automorphisms-section}
S^2\commutes{}&\andprimed X(t)
\\
\label{eq-S-squared-action-on-X-prime-G2-automorphisms-section}
S^2\cdot X'(t)\cdot S^{-2}={}&\andprimed X'(-t)
\\
\label{eq-word-in-S-commutes-with-X-prime-G2-automorphisms-section}
S S' S S' S\commutes{}&\andprimed X'(t)
\\
\label{eq-adjacent-long-roots-G2-automorphisms-section}
X'(t)\commutes{}&\andprimed S'S X'(u) S^{-1}S'^{-1}
\\
S S' X(t)S'^{-1}S^{-1}\commutes{}& S'S X'(u)S^{-1}S'^{-1}
\\
S X'(t) S^{-1}\commutes{}& S' X(u)S'^{-1}
\\
[X'(t),S X'(u) S^{-1}]={}&\andprimed
S'S X'(t u)S^{-1}S'^{-1}
\\
\label{eq-simple-roots-G2-automorphisms-section}
[X(t),X'(u)]
={}&
S S' X(t^2u)S'^{-1}S^{-1}
\\
&\cdot
S' X(-t u)S'^{-1}
\notag\\
&\cdot
S X'(t^3u) S^{-1}
\notag\\
&\cdot
S'S X'(-t^3u^2) S^{-1}S'^{-1}
\notag\\
S={}&\andprimedperiod X(1)S X(1)S^{-1} X(1)
\end{align}
The following theorem is proven just like the previous one.

\begin{theorem}
\label{thm-diagram-automorphisms-in-characteristic-3}
Suppose $R$ is a ring of characteristic~$3$.  Then the map
$S\leftrightarrow S'$, $X'(t)\mapsto X(t)\mapsto X'(t^3)$ extends to
an endomorphism $\phi$ of $\St_{G_2}(R)$.  If $R$ is a perfect field
then $\phi$ is an automorphism.
\qed
\end{theorem}

The exceptional diagram automorphisms lead to the famous Suzuki and
Ree groups.  If $R$ is the finite field $\F_{q}$ where $q=2^{\rm odd}$, then the Frobenius automorphism of $R$ (namely squaring) is
the square of a field automorphism $\xi$.  Writing $\xi$ also for
the induced automorphism of $\St_{B_2}(R)$, the Suzuki group is
defined as the subgroup where $\xi$ agrees with $\phi$.  The same
construction with $F_4$ in place of $B_2$ yields the large Ree groups,
and in characteristic~$3$ with $G_2$ yields the small Ree groups.
These groups are ``like'' groups of Lie type in that they admit root
group data in the sense of \cite{Tits-twin-buildings} or \cite{Caprace-Remy}, but they are not algebraic groups.

H\'ee generalized this.  He showed in \cite{Hee} that when a group with a
root group datum  admits two automorphisms that
permute the simple roots' root groups, and satisfy some other natural
conditions, then the subgroup where they coincide also admits a root
group datum.  Furthermore, the Weyl group for the subgroup may be
computed in a simple way from the Weyl group for the containing group.
For example, over $\F_q$ with $q=2^{\rm odd}$, the Kac-Moody group
\begin{equation*}
\begin{tikzpicture}[line width=1pt, scale=1]
%edges
\draw (-2,0) -- (-1.5,0);
\draw[dashed] (-1.5,0) -- (-.5,0);
\draw (-.5,0)--(2,0);
\draw (2,.04) -- (3,.04);
\draw (2,-.04) -- (3,-.04);
\draw (3,0)--(5.5,0);
\draw[dashed] (5.5,0)--(6.5,0);
\draw (6.5,0) -- (7,0);
%nodes
\fill (-2,0) circle (.1);
\fill (0,0) circle (.1);
\fill (1,0) circle (.1);
\fill (2,0) circle (.1);
\fill (3,0) circle (.1);
\fill (4,0) circle (.1);
\fill (5,0) circle (.1);
\fill (7,0) circle (.1);
%arrow
\draw (2.42,.16) -- (2.58,0) -- (2.42,-.16);
\end{tikzpicture}
\end{equation*}
contains a Kac-Moody-like analogue of the  Suzuki groups.  By H\'ee's
theorem,  its Weyl group is 
$$
\begin{tikzpicture}[line width=1pt, scale=1]
%edges
\draw (-2,0) -- (.5,0);
\draw[dashed] (.5,0) -- (1.5,0);
\draw (1.5,0)--(2,0);
%nodes
\fill (-2,0) circle (.1);
\fill (-1,0) circle (.1);
\fill (0,0) circle (.1);
\fill (2,0) circle (.1);
% edge label
\node at (-1.5,.2) {8};
\end{tikzpicture}
$$ In \cite{Hee-announcement}, H\'ee constructs diagram automorphisms
in a different way than we do, and discusses the case ``$G_4$'' in some
detail.

\section{The Kac-Moody algebra}
\label{sec-Kac-Moody-algebra}

\newcommand{\mylabel}[2]{\leavevmode\llap{#1\ }{#2}\par}%
\setbox0\hbox{$W{=}W(M)$\ }% widest entry

\noindent
In this section we begin the technical part of the paper, by recalling
the Kac-Moody algebra and some notation from \cite{Tits}.  All group
actions are on the left.  We will use the following general notation.

\begingroup
\leftskip=\wd0
\parindent=0pt
\mylabel{$\pairing{\,}{}$}{a bilinear pairing}
\mylabel{$\gend{\ldots}$}{a group generated by the elements enclosed}
\mylabel{$\presentation{\ldots}{\ldots}$}{a group presentation}
\mylabel{$[x,y]$}{$x y x^{-1}y^{-1}$ if $x$ and $y$ are group elements}
\mylabel{$\freeproduct$}{free product of groups (possibly with amalgamation)}
\endgroup

\smallskip
\noindent
The Steinberg group is built from a
generalized Cartan matrix $A$:

\begingroup
\leftskip=\wd0
\parindent=0pt
\mylabel{$I$}{an index set (the nodes of Dynkin diagram)}
\mylabel{$i,j$}{will always indicate elements of $I$}
\mylabel{$A{=}(A_{i j})$}{a generalized Cartan matrix: an
  integer matrix
satisfying $A_{i i}=2$, $A_{i j}\leq0$ if $i\neq j$, and $A_{i j}=0\iff A_{j i}=0$}
\mylabel{$m_{i j}$}{numerical edge labels of the Dynkin diagram:
$m_{i j}=2$, $3$, $4$, $6$ or $\infty$
according to whether $A_{i j}A_{j i}=0$,
$1$, $2$, $3$ or${}\geq4$, except that
$m_{i i}=1$.}
\mylabel{$W$}{the Coxeter group
$\presentation{s_{i\in I}}{\hbox{$(s_is_j)^{m_{i j}}=1$ if $m_{i j}\neq\infty$}}$}
\mylabel{$\ZI$}{the free abelian group with basis $\alpha_{i\in I}$ (the simple roots).  $W$
acts 
on $\ZI$ by $s_i(\alpha_j)=\alpha_j-A_{i j}\alpha_i$.  This action is faithful by the
theory of the Tits cone \cite[V\S4.4]{Bourbaki}.}
\mylabel{$\Phi$}{the set of (real) roots: all $w\alpha_i$ with $w\in W$ and $i\in I$}
\endgroup

\medskip\noindent The Kac-Moody algebra $\g=\g_A$ associated to $A$ means the
complex Lie algebra with generators $e_{i\in I}$, $f_{i\in I}$,
$\barh_{i\in I}$ and defining relations
\begin{gather*}
[\barh_i,e_j]=A_{i j}e_j,
\quad
[\barh_i,f_i]=-A_{i j}f_j,
\quad
[\barh_i,\barh_j]=0,
\quad
[e_i,f_i]=-\barh_i,
\\
\hbox{for $i\neq j$:}
\quad
[e_i,f_j]=0,
\quad
(\ad e_i)^{1-A_{i j}}(e_j)=(\ad f_i)^{1-A_{i j}}(f_j)=0.
\end{gather*}
% Tits omits the definition of $\ad x$. We follow the ``action on the
% left'' convention which is the one used by Kac.  (Kac doesn't define
% it either, but one can verify that this is his convention by
% comparing his exercise~1.13 with his \S3.4.)
(Note: $(\ad x)(y)$ means $[x,y]$.  Also, Tits' generators differ from
Kac' generators \cite{Kac-IDLA} by a sign on $f_i$.)  For any $i$ the linear
span of $e_i$, $f_i$ and $\barh_i$ is isomorphic to $\sltwo\C$, via
\begin{equation}
\label{eq-standard-basis-for-sl2}
e_i=\begin{pmatrix}0&1\\0&0\end{pmatrix}
\qquad
f_i=\begin{pmatrix}0&0\\-1&0\end{pmatrix}
\qquad
\barh_i=\begin{pmatrix}1&0\\0&-1\end{pmatrix}.
\end{equation}
We equip $\g$ with a grading by $\ZI$, with $\barh_i\in \g_0$, $e_i\in \g_{\alpha_i}$ and
$f_i\in \g_{-\alpha_i}$.  For $\alpha\in\Z^I$ we refer to $\g_\alpha$ as its root space, and abbreviate $\g_{\alpha_i}$ to 
$\g_i$.   We follow Tits \cite{Tits} in
saying ``root'' 
for ``real root'' (meaning an element of~$\Phi$).
Imaginary roots play no role in this paper.

\section{The extension $\Wstar\sset\Aut\g$ of the Weyl group}
\label{sec-W-star}

\noindent
The Weyl group $W$ does not necessarily act on $\g$, but a certain
extension of it called $\Wstar$ does.  In this section we review its basic properties.  The
results through theorem~\ref{thm-Artin-relations-in-W-star} are due to
Tits.  The last result is new: it describes the root stabilizers in
$\Wstar$.  The proof relies on Brink's study of reflection centralizers
in Coxeter groups \cite{Brink}, in the form given in \cite{Allcock-centralizers}.

It is standard \cite[lemma~3.5]{Kac-IDLA} that $\ad e_i$ and $\ad f_i$ are
locally nilpotent on $\g$, so their exponentials are automorphisms  of $\g$.
Furthermore, 
\begin{equation}
\label{eq-definition-of-s-star}
\begin{split}
(\exp\ad e_i)(\exp\ad f_i)&(\exp\ad e_i)
\\{}=
&(\exp\ad f_i)(\exp\ad e_i)(\exp\ad f_i).
\end{split}
\end{equation}
We write $\sstar_i$ for this element of $\Aut \g$ and $W^*$ for
$\gend{\sstar_{i\in I}}\sset\Aut \g$.  One shows
\cite[lemma~3.8]{Kac-IDLA} that $\sstar_i(\g_\alpha)=\g_{s_i(\alpha)}$ for all
$\alpha\in\ZI$.  This defines a $\Wstar$-action on $\ZI$, with $\sstar_i$
acting as $s_i$.  Since $W$ acts faithfully on $\ZI$ this yields a
homomorphism $\Wstar\to W$.  Using $\Wstar$, the general theory
\cite[prop.~5.1]{Kac-IDLA} shows that $\g_\alpha$ is $1$-dimensional for any
$\alpha\in\Phi$.

Let $\ZIvee$ be the free abelian group with basis the formal symbols
$\alpha_{i\in I}^\vee$ and define a bilinear pairing $\ZIvee\times\ZI\to\Z$ by
$\pairing{\alpha_i^\vee}{\alpha_j}=A_{i j}$.  
%This is set up so that the maps
%$\ZI\to\Lambda$ and $\ZIvee\to\Lambdavee$ sending $a_i$ to $\alpha_i$ and
%$\alpha_i^\vee$ to $h_i$ respect the pairing.  
We define an action
of $W$ on $\ZIvee$ by $s_i(\alpha_j^\vee)=\alpha_j^\vee-A_{j i}\alpha_i^\vee$.  One
can check that this action satisfies
$\pairing{w\alpha^\vee}{w \beta}=\pairing{\alpha^\vee}{\beta}$.  
%(Note:  $\ZIvee\tensor\R$ is not the space $\Hom(\ZI,\R)$
%that contains the Tits cone.  When the Cartan matrix is degenerate,
%the $W$-actions on them may be inequivalent.)
There is a homomorphism $\Ad:\ZIvee\to\Aut \g$, with
$\Ad(\alpha^\vee)$ acting on $\g_\beta$ by $(-1)^{\pairing{\alpha^\vee}{\beta}}$, where
$\beta\in\ZI$.  The proof of the next lemma is easy and standard.

\begin{lemma}
\label{lem-map-from-ZIvee-to-scalars-is-W-star-equivariant}
$\Ad:\ZIvee\to\Aut \g$ is $\Wstar$-equivariant in the sense that
$\wstar\cdot\Ad(\alpha^\vee)\cdot{\wstar}^{-1}=\Ad(w \alpha^\vee)$, where
$\alpha^\vee\in\ZIvee$  and $w$ is the image in $W$ of $\wstar\in\Wstar$.
\qed
\end{lemma}

%\begin{proof}
%If $b\in\ZI$ and $y\in \g_b$ then ${\wstar}^{-1}(y)$ lies in $\g_{w^{-1}b}$, so $\Ad(a^\vee)$
%acts on ${\wstar}^{-1}(y)$ by $(-1)^{\pairing{a^\vee}{w^{-1}b}}$.  Applying
%$\wstar$ to the result yields
%$$
%(-1)^{\pairing{a^\vee}{w^{-1}b}}y
%=
%(-1)^{\pairing{wa^\vee}{b}}y
%=
%\Ad(wa^\vee)(y).
%$$
%\end{proof}

\begin{samepage}
\begin{lemma}
\label{lem-relations-on-s-stars}
The following identities hold in $\Aut\g$.
\begin{enumerate}
\item
\label{item-action-of-s-star-squared}
$\sstar_i{}^2=\Ad(\alpha_i^\vee)$.
\item
\label{item-W-conjugacy-action-on-s-star-squared}
$\sstar_i({\sstar_j})^2\sstar_i{}^{-1}=(\sstar_j)^2(\sstar_i)^{-2A_{j i}}$.
\end{enumerate}
\end{lemma}
\end{samepage}

\begin{proof}[Proof sketch]
\ref{item-action-of-s-star-squared} 
Identifying the span of $e_i,f_i,\barh_i$ with
$\sltwo\C$ as in \eqref{eq-standard-basis-for-sl2}
identifies $\sstar_i{}^2$ with
$\bigl(\begin{smallmatrix}-1&0\\0&-1\end{smallmatrix}\bigr)\in\SL_2\C$.  
  One uses the representation theory of $\SL_2\C$ to see how this acts
  on $\g$'s weight spaces.

\ref{item-W-conjugacy-action-on-s-star-squared} uses
\ref{item-action-of-s-star-squared} to identify ${\sstar_j}^2$ with
$\Ad(\alpha_j^\vee)$, then
lemma~\ref{lem-map-from-ZIvee-to-scalars-is-W-star-equivariant} to identify
$\sstar_i\Ad(\alpha_j^\vee){\sstar_i}^{-1}$ with $\Ad(s_i(\alpha_j^\vee))$,
then the formula defining $s_i(\alpha_j^\vee)$, and finally
\ref{item-action-of-s-star-squared} again to convert back to ${\sstar_i}^2$ and ${\sstar_j}^2$.
\end{proof}

To understand the relations satisfied by the $\sstar_i$
it will be useful to have
a characterization of them in terms of the choice of $e_i$
(together with the grading on $\g$).  
This is part of Tits' ``trijection'' \cite[\S1.1]{Tits-on-structure-constants}.
In the notation of the following
lemma, $\sstar_i$ is $\sstar_{e_i}$ (or equally well~$\sstar_{\!f_i}$).

\begin{lemma}
\label{lem-characterization-of-s-star-in-terms-of-e}
If $\alpha\in\Phi$ and $e\in \g_\alpha-\{0\}$ then there exists a unique $f\in
\g_{-\alpha}$ such that
$$
\sstar_e:=(\exp\ad e)(\exp\ad f)(\exp\ad e)
$$ exchanges $\g_{\pm\alpha}$.  Furthermore, $s_e^*$ coincides with
$s_f^*$ and exchanges $e$ and $f$.  Finally,
if $\phi\in\Aut
\g$ permutes the $\g_{\beta\in\Phi}$ then
$\phi\sstar_e\phi^{-1}=\sstar_{\phi(e)}$.
\qed
\end{lemma}

%\begin{proof}[Proof sketch]
%Existence and uniqueness in $\Aut(\sltwo\C)$
%is easy to check.
%Then conjugation by an element of $\Wstar$ gives the result for any $e$ in any $\g_\alpha-\{0\}$.
%The last claim follows from the earlier ones, once we know that
%$\phi(\g_\alpha)= \g_\beta$ implies $\phi(\g_{-\alpha})=\g_{-\beta}$.  
%This is a consequence of the grading.
%\end{proof}

\begin{lemma}
\label{lem-s-star-action-on-e-i}
\leavevmode\hbox{}% to make enumerate environment start a new line
\begin{enumerate}
\item
\label{item-root-moving-m=3}
If $m_{i j}=3$ then $\sstar_j\sstar_i(e_j)=e_i$.  
\item
\label{item-root-moving-m=2-4-or-6}
If $m_{i j}=2$, $4$ or $6$
then $e_j$ is fixed by
$\sstar_i$,
$\sstar_i\sstar_j\sstar_i$ or 
$\sstar_i\sstar_j\sstar_i\sstar_j\sstar_i$ respectively.
\end{enumerate}
\end{lemma}

\begin{proof}
\ref{item-root-moving-m=3}
follows from direct calculation in $\slthree\C$.
%
%% Taking
%% $$
%% e_1=\begin{pmatrix}0&1&0\\0&0&0\\0&0&0\end{pmatrix}
%% \qquad
%% e_2=\begin{pmatrix}0&0&0\\0&0&1\\0&0&0\end{pmatrix}
%% \qquad
%% e_3=\begin{pmatrix}0&0&1\\0&0&0\\0&0&0\end{pmatrix}
%% $$
%% we have $(\ad e_1)(e_2)=e_3$, so $(\exp\ad e_1)(e_2)=e_2+e_3$.  Furthermore,
%% $(\ad f_1)(e_2)=0$ and $(\ad f_1)(e_3)=-e_2$, so 
%% $$
%% (\exp\ad f_1)(e_2+e_3)=e_2+e_3-e_2=e_3.
%% $$
%% Since $[e_1,e_3]=0$, $e_3$ is fixed by $\exp\ad e_1$, and we have
%% $\sstar_1(e_2)=e_3$.  Now we use the Lie algebra automorphism
%% $e_1\leftrightarrow e_2$, 
%% $f_1\leftrightarrow f_2$, 
%% $\barh_1\leftrightarrow \barh_2$, which negates $e_3=[e_1,e_2]$.  So we also
%% have $\sstar_2(e_1)=-e_3$, which is to say that
%% ${\sstar_2}^{-1}\sstar_1(e_2)=-e_1$.  Since $A_{12}=-1$,
%% ${\sstar_2}^2$ negates $e_1$, so $\sstar_2\sstar_1(e_2)=e_1$.
%% Following the proof of theorem~\ref{thm-root-stabilizers-in-W} shows that each $\cstar_\gamma$
%% fixes $e_i$.
%
In the $m_{i j}=2$ case of \ref{item-root-moving-m=2-4-or-6} we have
$(\ad e_i)(e_j)=(\ad f_i)(e_j)=0$, and $\sstar_i(e_j)=e_j$ follows
immediately.
The remaining cases involve careful
tracking of signs. We will write $(\sltwo\C)_i$ for the span of $e_i$,
$f_i$, $\barh_i$.

If $m_{i j}=4$ then $\{A_{i j},A_{j i}\}=\{-1,-2\}$ 
and $\alpha_i$ and $\alpha_j$ are simple roots for a $B_2$ root system.
Using lemma~\ref{lem-characterization-of-s-star-in-terms-of-e},
\begin{align}
\sstar_i\sstar_j\sstar_i(e_j)
&{}=
\sstar_i\sstar_{e_j}{\sstar_i}^{-1}{\sstar_i}^2(e_j)\notag\\
&{}=
\sstar_{\sstar_i(e_j)}\bigl((\Ad\alpha_i^\vee)(e_j)\bigr)\notag\\
&{}=
\label{eq-foo-1}
(-1)^{A_{i j}}\sstar_{\sstar_i(e_j)}(e_j).
\end{align}
Suppose first that $A_{i j}=-2$.  Then $\alpha_i$ is the short simple root, $\alpha_j$
the long one, and $s_i(\alpha_j)$ is a long root
orthogonal to $\alpha_j$.  We have
\begin{align*}
\sstar_{\sstar_i(e_j)}
={}&
\bigl(\exp\ad\sstar_i(e_j)\bigr)\bigl(\exp\ad\sstar_i(f_j)\bigr)\bigl(\exp\ad\sstar_i(e_j)\bigr)\\
\in{}&
\exp\ad\bigl(\sstar_i\bigl((\sltwo\C)_j\bigr)\bigr).
\end{align*}
Now, $\sstar_i\bigl((\sltwo\C)_j\bigr)$ annihilates $\g_j$ because its root string through
$\alpha_j$ has length~$1$.  So $\sstar_{\sstar_i(e_j)}$ fixes $e_j$ and
\eqref{eq-foo-1} becomes
$$
\sstar_i\sstar_j\sstar_i(e_j)=(-1)^{A_{i j}}e_j=(-1)^{-2}e_j=e_j.
$$ On the other hand, if $A_{i j}=-1$ then $\alpha_j$ and $s_i(\alpha_j)$ are
orthogonal short roots.  Now the root string through $\alpha_j$ for
$\sstar_i\bigl((\sltwo\C)_j\bigr)$ has length~$3$, so the
$\sstar_i\bigl((\sltwo\C)_j\bigr)$-module generated by $e_j$ is a copy of the
adjoint representation.  In particular,
$\sstar_{\sstar_i(e_j)}=\sstar_i\sstar_j s^{*-1}_i$ acts on $\g_j$
by the same scalar as on the Cartan subalgebra $\sstar_i(\C \barh_j)$
of $\sstar_i\bigl((\sltwo\C)_j\bigr)$.  This is the same scalar by which
$\sstar_j$ acts on $\C \barh_j$, which is~$-1$.  So
$\sstar_{\sstar_i(e_j)}$ negates $e_j$ and \eqref{eq-foo-1} reads
$$
\sstar_i\sstar_j\sstar_i(e_j)=(-1)^{A_{i j}}(-e_j)=(-1)^{-1}(-e_j)=e_j.
$$

Now suppose $m_{i j}=6$, so that $\{A_{i j},A_{j i}\}=\{-1,-3\}$,
$\alpha_i$ and $\alpha_j$ are simple roots for a $G_2$ root system, and
$s_is_j(\alpha_i)\perp \alpha_j$.
Then
\begin{align*}
\sstar_i\sstar_j\sstar_i\sstar_j\sstar_i(e_j)
&{}=
\bigl(\sstar_i\sstar_j\sstar_{e_i}{\sstar_j}^{-1}{\sstar_i}^{-1}\bigr)\sstar_i\sstar_j\sstar_j\sstar_i(e_j)\\
&{}=
\sstar_{\sstar_i\sstar_j(e_i)}\circ\bigl(\sstar_i{\sstar_j}^2{\sstar_i}^{-1}\bigr)\circ{\sstar_i}^2(e_j)\\
&{}=
\sstar_{\sstar_i\sstar_j(e_i)}\circ{\sstar_j}^2{\sstar_i}^{-2A_{j i}}\circ{\sstar_i}^2(e_j)\\
&{}=
\sstar_{\sstar_i\sstar_j(e_i)}\circ{\sstar_j}^2{\sstar_i}^{4\,\rm{or}\,8}(e_j)\\
&{}=
\sstar_{\sstar_i\sstar_j(e_i)}(e_j).
\end{align*}
The root string through $\alpha_j$ for
$\sstar_i\sstar_j\bigl((\sltwo\C)_i\bigr)$ has length~$1$,
so arguing as in the $B_2$ case shows that
$\sstar_{\sstar_i\sstar_j(e_i)}$ fixes $e_j$.
\end{proof}

\begin{theorem}[Tits {\cite[\S4.6]{Tits-Normalisateurs}}]
\label{thm-Artin-relations-in-W-star}
The $\sstar_i$ satisfy the Artin relations of~$M$.  That is, if
$m_{i j}\neq\infty$ then
$\sstar_i\sstar_j\cdots=\sstar_j\sstar_i\cdots$, where there are
$m_{i j}$ factors on each side, alternately $\sstar_i$ and $\sstar_j$.
\end{theorem}

\begin{proof}
For $m_{i j}=3$ we start with $e_j=\sstar_i\sstar_j(e_i)$ from
lemma~\ref{lem-s-star-action-on-e-i}\ref{item-root-moving-m=3}.  Using lemma~\ref{lem-characterization-of-s-star-in-terms-of-e} yields
$$
\sstar_j
=
\sstar_{e_j}
=
\sstar_{\sstar_i\sstar_j(e_i)}
=
\sstar_i\sstar_j\sstar_{e_i}{\sstar_j}^{-1}{\sstar_i}^{-1}
=
\sstar_i\sstar_j\sstar_i{\sstar_j}^{-1}{\sstar_i}^{-1}.
$$
The other cases are the same.  
\end{proof}

We will need to understand the $\Wstar$-stabilizer of a simple root
$\alpha_i$ and how it acts on $\g_i$.  The first step is to 
quote from \cite{Allcock-centralizers} a refinement of a theorem of
Brink \cite{Brink} on reflection centralizers in Coxeter
groups.  Then we will ``lift'' this result to $\Wstar$ by keeping
track of signs.  

Both theorems refer to the ``odd Dynkin diagram'' $\odddiagram$,
which means the graph with vertex set $I$ where vertices $i$ and $j$
are joined just if $m_{i j}=3$.  
For $\gamma$ an edge-path in $\odddiagram$, with $i_0,\dots,i_n$ the
vertices along it, we define
\begin{equation}
\label{eq-def-of-p-gamma}
p_\gamma:=(s_{i_{n-1}}s_{i_n})(s_{i_{n-2}}s_{i_{n-1}})\cdots
(s_{i_1}s_{i_2})(s_{i_0}s_{i_1}).
\end{equation}
(If $\gamma$ has length~$0$ then we set $p_\gamma=1$.)  For $i\in I$ we write
$\odddiagram_i$ for its component of $\odddiagram$.

\begin{theorem}[{\cite[cor.~8]{Allcock-centralizers}}]
\label{thm-root-stabilizers-in-W}
Suppose $i\in I$, 
$Z$ is a set of closed edge-paths based at $i$ that generate
$\pi_1(\odddiagram_i,i)$, and $\delta_j$ is an edge-path in $\odddiagram_i$ from
$i$ to $j$, for each vertex $j$ of $\odddiagram_i$.  For each such $j$
and each $k\in I$ with $m_{j k}$ finite and even, define
\begin{equation}
\label{eq-def-of-r-gamma-k}
r_{j k}:=p_{\delta_j}^{-1}
\cdot
\left\{
\begin{matrix}
s_k\\
s_k s_j s_k\\
s_k s_j s_k s_j s_k
\end{matrix}
\right\}
\cdot
p_{\delta_j}
\end{equation}
according to whether $m_{j k}=2$, $4$ or~$6$.  Then the $W$-stabilizer
of the simple root $\alpha_i$ is generated by the $r_{j k}$ and the $p_{z\in Z}$.
\qed
\end{theorem}

It is easy to see that
the $r_{j k}$ and $p_z$ stabilize $\alpha_i$.  In fact
this is the ``image under $\Wstar\to W$'' of the corresponding
part of the next theorem.

\begin{theorem}
\label{thm-root-stabilizers-in-W-star}
Suppose $i$, $Z$ and the $\delta_j$ are as in
theorem~\ref{thm-root-stabilizers-in-W}.  Define $\pstar_\gamma$ and
$\rstar_{j k}$ by attaching $*$'s to the $s$'s, $p$'s and $r$'s in
\eqref{eq-def-of-p-gamma} and \eqref{eq-def-of-r-gamma-k}.  Then the
$\pstar_{z\in Z}$ and $\rstar_{j k}$ fix $e_i$, and together with the
$s^{*2}_{l\in I}$ they generate the $\Wstar$-stabilizer of $\alpha_i$.
{\rm(}By
lemma~\ref{lem-relations-on-s-stars}\ref{item-action-of-s-star-squared},
${\sstar_l}^2$ acts on $e_i$ by $(-1)^{A_{l i}}${\rm)}.
\end{theorem}

\begin{proof}
The $\Wstar$-stabilizer of $\alpha_i$ is generated by $\ker(\Wstar\to W)$
and any set of elements of $\Wstar$ whose projections to $W$ generate
the $W$-stabilizer of $\alpha_i$.  
Now, the ${\sstar_i}^2$ normally generate the kernel because of the
Artin relations.  Lemma~\ref{lem-relations-on-s-stars}\ref{item-W-conjugacy-action-on-s-star-squared} shows that the subgroup they
generate is normal, hence equal to this kernel.  Since
the $\pstar$'s and $\rstar$'s project to the $p$'s and $r$'s of
theorem~\ref{thm-root-stabilizers-in-W}, our generation claim follows
from that theorem.  To see that the $\pstar_z$'s fix $e_i$, apply
lemma~\ref{lem-s-star-action-on-e-i}\ref{item-root-moving-m=3}
repeatedly.  The same argument proves $\pstar_{\delta_j}(e_i)=e_j$.  Then
using
lemma~\ref{lem-s-star-action-on-e-i}\ref{item-root-moving-m=2-4-or-6}
shows that $e_j$ is fixed by $\sstar_k$, $\sstar_k\sstar_j\sstar_k$ or
$\sstar_k\sstar_j\sstar_k\sstar_j\sstar_k$ according to whether
$m_{j k}$ is $2$, $4$ or~$6$.  Applying ${\pstar_{\delta_j}}^{-1}$ sends
$e_j$ back to $e_i$, proving $\rstar_{j k}(e_i)=e_i$.
\end{proof}

\section{The Steinberg group $\St$}
\label{sec-Steinberg-group}

\noindent
In this section we give an overview of the Steinberg group $\St_A$, as
defined by Tits \cite{Tits} and refined by Morita--Rehmann
\cite{Morita-Rehmann}.  The purpose is to be able to compare the
pre-Steinberg group $\PSt_A$ (see the next section) with $\St_A$.   For
example,  theorem~\ref{thm-examples-of-PSt-to-St-being-an-isomorphism} gives many cases in which the natural map
$\PSt_A(R)\to\St_A(R)$ is an isomorphism.

The Morita--Rehmann definition is got from Tits' definition by imposing
some additional relations.  These are also due to Tits, but he imposed
them only later in his construction, when defining Kac-Moody groups in
terms of $\St_A$.  In the few places where we need to distinguish
between the definitions, we will write $\StTits_A$ for Tits' version
and $\St_A$ for the Morita-Rehmann version.  In the rest of this
section we will regard $A$ as fixed and omit it from the subscripts.

\smallskip
$\Add$ denotes the additive group, regarded as a group scheme over
$\Z$.  That is, it is the functor assigning to each commutative ring
$R$ its underlying
abelian group.  
The Lie algebra of $\Add$ is
canonically isomorphic to $\Z$.

For each $\alpha\in\Phi$, $\g_\alpha\cap\Wstar\bigl(\{e_{i\in I}\}\bigr)$
consists of either one vector or two antipodal vectors.  This is
\cite[3.3.2]{Tits} and its following paragraph, which relies on
\cite[\S13.31]{Tits-SBNP}.  Alternately, it follows from our
theorem~\ref{thm-root-stabilizers-in-W-star}.  We write $\g_{\alpha,\Z}$
for the $\Z$-span in $\g_\alpha$ of this element or antipodal pair,
and $E_\alpha$ for the set of its generators (a set of size~$2$).  The symbol $e$ will always
indicate an element of some $E_\alpha$.
We
define $\U_\alpha$ as the group scheme over $\Z$ which is isomorphic to
$\Add$ and has Lie algebra $\g_{\alpha,\Z}$.  That is, $\U_\alpha$ is
the functor assigning to each commutative ring $R$ the abelian group
$\g_{\alpha,\Z}\tensor R\iso R$.  For $i\in I$ we abbreviate
$\U_{\pm\alpha_i}$ to $\U_{\pm i}$.

If $\alpha\in\Phi$ and $e\in E_\alpha$ then we define $\x_e$ as the isomorphism $\Add\to\U_\alpha$
whose corresponding Lie algebra isomorphism identifies $1\in\Z$ with
$e\in \g_{\alpha,\Z}$.  For fixed $R$ this amounts to
$$
\x_e(t):=e\tensor t\in\g_{\alpha,\Z}\tensor R=\U_\alpha.
$$
If $R=\R$ or $\C$ then one may think of $\x_e(t)$ as $\exp(t e)$.
For $i\in I$ we abbreviate $\x_{e_i}$ to $\x_i$ and $\x_{f_i}$ to $\x_{-i}$.

Tits calls a set of roots $\Psi\sset\Phi$ prenilpotent if
some chamber in the open Tits cone lies on the positive side of all
their mirrors and some other chamber lies on the negative side
of all of them.  
(Equivalently, some element of
$W$ sends $\Psi$ into the set of positive roots and some other element
of $W$ sends $\Psi$ into the set of negative roots.)
It follows that $\Psi$ is finite.  If $\Psi$ is also
closed under addition then it is called nilpotent.  In this case
$\g_\Psi:=\oplus_{\alpha\in\Psi}\,\g_\alpha$ is a nilpotent Lie algebra
\cite[p.\ 547]{Tits}.

\begin{lemma}[Tits {\cite[sec.~3.4]{Tits}}]
\label{lem-existence-of-unipotent-groups}
If $\Psi\sset\Phi$ is a nilpotent set of roots, then there is a unique
unipotent group scheme $\U_\Psi$ over $\Z$ with the properties
\begin{enumerate}
\item
$\U_\Psi$ contains all the $\U_{\alpha\in\Psi}$;
\item
$\U_\Psi(\C)$ has Lie algebra $\g_\Psi$;
\item
\label{item-isomorphism-of-underlying-schemes}
For any ordering on $\Psi$, the product morphism
$\prod_{\alpha\in\Psi}\U_\alpha\to\U_\Psi$ is an isomorphism of the underlying
schemes.
\qed
\end{enumerate}
\end{lemma}

Tits' version $\StTits$ of the Steinberg group functor is defined as
follows.  For each prenilpotent pair $\alpha,\beta$ of roots,
$\theta(\alpha,\beta)$ is defined as $(\N \alpha+\N \beta)\cap\Phi$
where $\N=\{0,1,2,\dots\}$.  Consider the groups
$\U_{\theta(\alpha,\beta)}$ with $\{\alpha,\beta\}$ varying over all
prenilpotent pairs.  If $\gamma\in\theta(\alpha,\beta)$
then there is a natural injection
$\U_\gamma\to\U_{\theta(\alpha,\beta)}$, yielding a diagram of
inclusions of group functors.  $\StTits$ is defined as the
direct limit of this diagram.  Every automorphism of $\g$ that
permutes the subgroups $\g_{\alpha,\Z}$ induces an automorphism of the
diagram of inclusions of group functors, hence an automorphism of
$\StTits$.  In particular, $\Wstar$ acts on~$\StTits$.

As Tits points out, a helpful but less canonical way to think about
$\StTits(R)$ is to begin with the free product
$\freeproduct_{\alpha\in\Phi}\U_\alpha(R)$ and impose relations of the form
\begin{equation}
\label{eq-Chevalley-relation-general-form}
[\x_{e_\alpha}(t),\x_{e_\beta}(u)]
=
\prod_{\gamma=m\alpha+n\beta}
\x_{e_\gamma}\bigl(C_{\alpha\beta\gamma}t^m u^n\bigr)
\end{equation}
for each prenilpotent pair $\alpha,\beta\in\Phi$.  Here
$\gamma=m\alpha+n\beta$ runs over
$\theta(\alpha,\beta)-\{\alpha,\beta\}$, so in particular $m$ and $n$
are positive integers.
Also, $e_\alpha$, $e_\beta$
and the various $e_\gamma$ lie in $E_\alpha$, $E_\beta$ and the various $E_\gamma$, and must be chosen before
the relation can be written down explicitly.  The 
$C_{\alpha\beta\gamma}$ are integers that 
depend the position of $\gamma$ relative to
$\alpha$ and $\beta$, the choices of $e_\alpha$, $e_\beta$ and
the $e_\gamma$, and the ordering of the product; cf.\ (3) of
\cite{Tits}.  Usually \eqref{eq-Chevalley-relation-general-form} is called ``the Chevalley relation of
$\alpha$ and~$\beta$''.  It is
really a family of relations parameterized by $t$ and~$u$, and
(strictly speaking) not defined without the various choices being fixed.

\smallskip
Unfortunately, Tits' version of the Steinberg group is different from
Steinberg's original group when the Dynkin diagram has $A_1$
components.  Therefore we follow 
Morita-Rehmann \cite{Morita-Rehmann} in  defining the Steinberg group
functor $\St$.  That is, we 
impose the additional relations
\eqref{eq-Morita-Rehmann-relations-abstract-form}, which correspond to
the relations
(B$'$) in
\cite{Steinberg} or \cite{Morita-Rehmann}.
These relations make the
``maximal torus'' and ``Weyl group'' act on the root groups
$\U_\alpha$ in the expected manner.
If $A$ is $2$-spherical without
$A_1$ components then
the
Morita-Rehmann relations already hold in $\StTits$ and this part of
the construction can be skipped,  by \cite[(a$_4$), p. 550]{Tits}.

The relators involve the following elements of
$\StTits$.
If $\alpha\in\Phi$ and $e\in E_\alpha$  then recall from
lemma~\ref{lem-characterization-of-s-star-in-terms-of-e} that  there
is a distinguished  $f\in E_{-\alpha}$.  As the
notation suggests, if $e=e_i$ then $f=f_i$.  For any $r\in \Runits$ we
define
\begin{align}
\label{eq-defn-of-s-tilde-e}
\stilde_e(r){}&:=\x_e(r)\x_{f}(1/r)\x_e(r)\\
\label{eq-defn-of-h-tilde-e}
\htilde_e(r){}&:=\stilde_e(r)\stilde_e(-1)
\end{align}
We abbreviate special cases in the usual way:
$\htilde_{\pm i}(r)$ for $\htilde_{e_i}(r)$ and $\htilde_{f_i}(r)$,
$\stilde_{\pm i}(r)$ for $\stilde_{e_i}(r)$ and $\stilde_{f_i}(r)$,
$\stilde_{\pm i}$ for $\stilde_{\pm i}(1)$,
and $\stilde_e$ for $\stilde_e(1)$.
It is useful to note several immediate consequences of the definitions: 
$\stilde_e(-r)=\stilde_e(r)^{-1}$, $\htilde_e(1)=1$, and
\begin{equation}
\label{eq-htilde-htilde-inverse-equals-stilde-stilde-inverse}
\stilde_e(r)\stilde_e(r')^{-1}
=
\htilde_e(r)\htilde_e(r')^{-1}
.
\end{equation}
Conceptually, the relations we will impose on $\StTits$ to get $\St$
force the conjugation maps of the various $\stilde_e(r)$'s
to be the same as certain automorphisms of
$\StTits$.  So we will describe these
automorphisms and then state the relations.

Recall from
lemma~\ref{lem-map-from-ZIvee-to-scalars-is-W-star-equivariant} and
its preceding remarks that $\ZIvee$ is the free abelian group
generated by formal symbols $\alpha_{i\in I}^\vee$. Also, the bilinear
pairing $\ZIvee\times\ZI\to\Z$ given by
$\pairing{\alpha_i^\vee}{\alpha_j}=A_{i j}$ is $W$-invariant.  We
defined a map $\Ad:\ZIvee\to\Aut\g$, which we generalize to
$\Ad:(\Runits\tensor\ZIvee)\to\Aut\bigl(\freeproduct_{\alpha\in\Phi}\,\U_\alpha\bigr)$
as follows.  For any $\alpha^\vee\in\ZIvee$, $r\in \Runits$ and
$\beta\in\Phi$, $\Ad(r\tensor\alpha^\vee)$ acts on $\U_\beta\iso R$ by
multiplication by $r^{\pairing{\alpha^\vee}{\beta}}\in \Runits$.  One
recovers the original $\Ad$ by taking $r=-1$.

The Chevalley relations have a homogeneity property, namely that
$\Ad(r\tensor\alpha^\vee)$ permutes them.  This is most visible when
they are stated in the form \eqref{eq-Chevalley-relation-general-form}.  Therefore the action $\Ad$ of
$\Runits\tensor\ZIvee$ on $\freeproduct_{\alpha\in\Phi}\,\U_\alpha$
descends to an action on $\StTits(R)$.  

It is standard that there is a
$W$-equivariant bijection $\alpha\mapsto\alpha^\vee$ from the roots
$\Phi\sset\Z^I$ to their corresponding coroots in $\ZIvee$.  As the
notation suggests, the coroots corresponding to the simple roots
$\alpha_i$ are our basis $\alpha^\vee_i$ for $\ZIvee$.  By
$W$-equivariance this determines the bijection uniquely.  For
$\alpha\in\Phi$ and $r\in \Runits$ we define
$h_\alpha(r)\in\Aut\StTits(R)$
as $\Ad(r\tensor\alpha^\vee)$.  As usual, we abbreviate $h_{\alpha_i}(r)$
to~$h_i(r)$.

We define the Steinberg group functor $\St$ as follows. Informally,
$\St(R)$ is the quotient of $\StTits(R)$ got by forcing every
$\stilde_e(r)$ to act on every $\U_\beta(R)$ by
$h_\alpha(r)\circ\sstar_e$, where $\alpha$ is the root with $e\in E_\alpha$.  Formally, it is the quotient by the
subgroup normally generated by the elements
\begin{equation}
\label{eq-Morita-Rehmann-relations-abstract-form}
\stilde_e(r)\, u\,\stilde_e(r)^{-1} \cdot\Bigl(\bigl(h_\alpha(r)\circ
s_e^*\bigr)(u)\Bigr)^{-1}
\end{equation}
as $\alpha,\beta$ vary over $\Phi$, $e$ over $E_\alpha$, $r$ over
$\Runits$, and $u$ over $\U_\beta(R)$.  This set of relators is
visibly $\Wstar$-invariant, so $\Wstar$ acts on $\St$.

\begin{remark}
Because $\stilde_e(r)=\htilde_e(r)\stilde_e$,
an equivalent way to impose the relations \eqref{eq-Morita-Rehmann-relations-abstract-form} is by quotienting
by the subgroup of $\StTits(R)$ normally generated by all
\begin{align}
\label{eq-Morita-Rehmann-relations-just-the-s-1-tildes}
\stilde_e\, u\,\stilde_e^{-1} &{}\cdot s_e^*(u)^{-1}
\\
\label{eq-Morita-Rehmann-relations-just-the-h-tildes-e}
\htilde_e(r)\, u\,\htilde_e(r)^{-1}&{}\cdot \bigl(h_\alpha(r)(u)\bigr)^{-1}.
\end{align}
\end{remark}

\begin{remark}
Our relations differ slightly from the relations (B$'$) of Morita--Rehmann
\cite{Morita-Rehmann}, because 
we 
follow Tits' convention for the presentation of
$\g$ while they follow Kac' convention 
(see section~\ref{sec-Kac-Moody-algebra}).  Our relations
also differ from Tits' relations in the definition of his
Kac--Moody group functor \cite[sec.\ 3.6]{Tits}, even taking into
account that our
$\htilde_i(r)$ corresponds to his $r^{h_i}$.  This is because R\'emy
observed \cite[8.3.3]{Remy} that Tits' relator (6), namely
$\stilde_i(r)^{-1}\cdot \stilde_i\cdot r^{h_i}$, is in error.  R\'emy
fixed it by replacing the first $r$ by $1/r$.  Our repair, by
exchanging the last two terms, is equivalent.
\end{remark}

%\begin{remark}
%We used Tits' relator (6) in \cite[sec.\ 3.6]{Tits} in our definition
%of $\htilde_i(r)$ in \eqref{eq-defn-of-h-tilde-e}, corresponding to
%his $r^{h_i}$.  R\'emy observed \cite[8.3.3]{Remy} that this relator,
%namely $\stilde_i(r)^{-1}\cdot \stilde_i\cdot r^{h_i}$, is in error.
%R\'emy resolved this by replacing it with $\stilde_i(1/r)^{-1}\cdot
%\stilde_i\cdot r^{h_i}$.  We replaced it instead with
%\eqref{eq-defn-of-h-tilde-e}.  
%In the Kac--Moody group, $\stilde_i$ inverts $r^{h_i}$, which
%makes it easy to see that our repair is equivalent to R\'emy's.  We
%also remark that just before the statement of Tits' theorem~1, $\x_-$
%should send $u$ to
%$\bigl(\begin{smallmatrix}1&0\\-u&1\end{smallmatrix}\bigr)$, not
%  $\bigl(\begin{smallmatrix}1&0\\u&1\end{smallmatrix}\bigr)$.
%    Finally, the right side of the last condition in Tits' axiom
%    (KMG5) should read $t(-\alpha_i)\cdot f_i$ rather than
%    $-t(\alpha_i)\cdot f_i$.
%\end{remark}

\begin{theorem}[Alternative defining relations for $\St$]
\label{thm-alternate-relations-for-Steinberg-group}
The kernel of the natural map
$\StTits(R)\to\St(R)$ is
the smallest normal subgroup containing the elements
\begin{align}
\label{eq-alternative-Steinberg-relations-h-action-on-simple-root-groups}
\htilde_i(r)\,\x_j(t)\,\htilde_i(r)^{-1}
&{}\cdot \x_j(r^{A_{i j}}t)^{-1}
\\
\label{eq-alternative-Steinberg-relations-h-action-on-negative-simple-root-groups}
\htilde_i(r)\ \stilde_j \x_j(t)\stilde_j^{-1}\ \htilde_i(r)^{-1}
&{}\cdot 
\bigl(\stilde_j\,\x_j(r^{-A_{i j}}t)\,\stilde_j^{-1}\bigr)^{-1} \\
\label{eq-alternative-Steinberg-relations-s-action-on-everything}
\stilde_i\, u\, \stilde_i^{-1}
&{}\cdot\sstar_i(u)^{-1}
\end{align}
for all $i,j\in I$, $r\in\Runits$, $t\in R$ and $u\in\U_\beta$ where
$\beta$ may be any root.  
Furthermore, the identities
\begin{align}
\label{eq-s-tilde-action-on-h-tildes}
\stilde_i\,\htilde_j(r)\,\stilde_i^{-1}
&{}=
\htilde_i\bigl(r^{A_{j i}}\bigr)^{-1}\,\htilde_j(r)
\\
\label{eq-commutators-of-h-tildes}
[\htilde_i(r),\htilde_j(r')]
&{}=\htilde_j\bigl(r^{A_{i j}}r'\bigr)
\,
\htilde_j\bigl(r^{A_{i j}}\bigr)^{-1}
\,
\htilde_j(r')^{-1}
%\\
%\label{eq-stilde-e-vs-stilde-f}
%\stilde_i(r)
%&{}=
%\stilde_{-i}(1/r) 
\end{align}
hold in $\St(R)$, for all
$i,j\in I$, $r,r'\in\Runits$.
\end{theorem}

\begin{remark}[Applicability to $\PSt$]
\label{remark-do-not-need-St-Tits-relations}
The proof below does not use the relations defining $\StTits$.  So 
it shows that the subgroup of
$\freeproduct_{\alpha\in\Phi}\,\U_\alpha(R)$ normally generated by the
relators \eqref{eq-Morita-Rehmann-relations-abstract-form} is the same as the one normally generated by
\eqref{eq-alternative-Steinberg-relations-h-action-on-simple-root-groups}--\eqref{eq-alternative-Steinberg-relations-s-action-on-everything}, and that \eqref{eq-s-tilde-action-on-h-tildes}--\eqref{eq-commutators-of-h-tildes} hold in the quotient.
This is useful because we will use the same relations when defining
the pre-Steinberg group $\PSt$ in the next section.
\end{remark}

\begin{proof}
We begin by showing that \eqref{eq-alternative-Steinberg-relations-h-action-on-simple-root-groups}--\eqref{eq-alternative-Steinberg-relations-s-action-on-everything} are trivial in $\St(R)$.
First, \eqref{eq-alternative-Steinberg-relations-s-action-on-everything} is
got from \eqref{eq-Morita-Rehmann-relations-abstract-form} by taking $e=e_i$ and $r=1$.
Next, recall the definition of $\htilde_i(r)$ as $\stilde_i(r)\stilde_i(-1)$
in \eqref{eq-defn-of-h-tilde-e}, and that the defining relations \eqref{eq-Morita-Rehmann-relations-abstract-form} for $\St(R)$ say
how $\stilde_i(r)$ acts on every $\U_\beta$.  So
$\htilde_i(r)$ acts on every $\U_\beta$ as 
\begin{align*}
h_i(r)\circ\sstar_i\circ h_i(-1)\circ\sstar_i
&{}=
h_i(r)\circ h_i(-1)\circ(\sstar_i)^{2}
\\
&{}=
h_i(r)\circ h_i(-1)\circ h_i(-1)
=
h_i(r).
\end{align*}
Taking $\beta=\alpha_j$ gives \eqref{eq-alternative-Steinberg-relations-h-action-on-simple-root-groups}.  For \eqref{eq-alternative-Steinberg-relations-h-action-on-negative-simple-root-groups}, take
$\beta=-\alpha_j$ and use the fact that $\stilde_j$ swaps
$\U_{\pm\alpha_j}$ (since it acts as $\sstar_j$).  This finishes the
proof that \eqref{eq-alternative-Steinberg-relations-h-action-on-simple-root-groups}--\eqref{eq-alternative-Steinberg-relations-s-action-on-everything} are trivial in $\St(R)$.

Now we write $N$ for the smallest normal subgroup of $\StTits(R)$
containing
\eqref{eq-alternative-Steinberg-relations-h-action-on-simple-root-groups}--\eqref{eq-alternative-Steinberg-relations-s-action-on-everything}
and $\equiv$ for equality modulo~$N$.  We will show that
\eqref{eq-s-tilde-action-on-h-tildes}--\eqref{eq-commutators-of-h-tildes}
hold modulo~$N$ and that the relators
\eqref{eq-Morita-Rehmann-relations-just-the-s-1-tildes}--\eqref{eq-Morita-Rehmann-relations-just-the-h-tildes-e}
are trivial modulo~$N$.  We will use relator
\eqref{eq-alternative-Steinberg-relations-s-action-on-everything}
without explicit mention: modulo~$N$, each $\stilde_i$ acts on every $\U_\beta$ as
$\sstar_i$.

First we establish \eqref{eq-s-tilde-action-on-h-tildes}--\eqref{eq-commutators-of-h-tildes}.  Starting from  the definition of
$\stilde_j(r')$, we have
$$
\stilde_j(r')
=
\x_j(r')\x_{-j}(1/r')\x_j(r')
\equiv
\x_j(r')\cdot\stilde_j\x_j(1/r')\stilde_j^{-1}\cdot\x_j(r').
$$
Now the  relators
\eqref{eq-alternative-Steinberg-relations-h-action-on-simple-root-groups}--\eqref{eq-alternative-Steinberg-relations-h-action-on-negative-simple-root-groups} give
\begin{equation}
\label{eq-h-tilde-r-acting-on-stilde-r-prime}
\htilde_i(r)\,\stilde_j(r')\,\htilde_i(r)^{-1}
\equiv
\stilde_j(r^{A_{i j}}r')
\end{equation}
Taking $r'=1$, left-multiplying by  $\htilde_i(r)^{-1}$,
right-multiplying by $\stilde_j^{-1}$, and then inverting both
sides and using \eqref{eq-htilde-htilde-inverse-equals-stilde-stilde-inverse}, gives
\begin{align}
\notag
\stilde_j\,\htilde_i(r)\,\stilde_j^{-1}
&{}\equiv
\stilde_j(1)\,\stilde_j(r^{A_{i j}})^{-1}\htilde_i(r)
\\
\label{eq-s-tilde-j-acting-on-h-tilde-i-of-r}
&{}\equiv
\htilde_j\bigl(r^{A_{i j}}\bigr)^{-1}\,\htilde_i(r).
\end{align}
Exchanging $i$ and $j$ establishes \eqref{eq-s-tilde-action-on-h-tildes}.
Also,    \eqref{eq-h-tilde-r-acting-on-stilde-r-prime},
\eqref{eq-defn-of-h-tilde-e} and \eqref{eq-htilde-htilde-inverse-equals-stilde-stilde-inverse} show that
$$
\htilde_i(r)\htilde_j(r')\htilde_i(r)^{-1}
\equiv\stilde_j\bigl(r^{A_{i j}}r'\bigr)\,\stilde_j\bigl(r^{A_{i j}}\bigr)^{-1}
=\htilde_j\bigl(r^{A_{i j}}r'\bigr)\,\htilde_j\bigl(r^{A_{i j}}\bigr)^{-1}
$$  Right-multiplication by $\htilde_j(r')^{-1}$ gives \eqref{eq-commutators-of-h-tildes}.

Now we will prove
\eqref{eq-Morita-Rehmann-relations-just-the-h-tildes-e} for all $e_i$.
That is: modulo
$N$, $\htilde_i(r)$ acts on every $\U_\beta$ by $h_i(r)$.  To prove this, write $E$ for $\cup_{\beta\in\Phi}\,E_\beta$ and
consider for any $e\in E$ the following condition:
\begin{equation}
\label{eq-to-establish-torus-action-on-root-groups}
\htilde_i(r)\x_e(t)\htilde_i(r)^{-1}
\equiv
\x_e\bigl(r^{\pairing{\alpha_i^\vee}{\beta}}t\bigr)
\ 
\hbox{for all $i\in I$, $r\in\Runits$ and $t\in R$,}
\end{equation}
where $\beta$ is the root with $e\in E_\beta$.  The set of $e\in E$
satisfying this condition is closed under negation, because
$\x_{-e}(t)=\x_e(-t)$.  For every $j\in U$, it contains $e_j\in E_{\alpha_j}$ and
$f_j\in E_{-\alpha_j}$ by relations \eqref{eq-alternative-Steinberg-relations-h-action-on-simple-root-groups}--\eqref{eq-alternative-Steinberg-relations-h-action-on-negative-simple-root-groups}.  The next
paragraph shows that it is closed under the action of $\Wstar$.
Therefore all $e\in E$ satisfy
\eqref{eq-to-establish-torus-action-on-root-groups}, establishing
\eqref{eq-Morita-Rehmann-relations-just-the-h-tildes-e} for all $e=e_i$.

Here is the calculation that if $e\in E$ satisfies
\eqref{eq-to-establish-torus-action-on-root-groups}, and $j$ is any
element of $I$, then $\sstar_j(e)$ also satisfies
\eqref{eq-to-establish-torus-action-on-root-groups}.  
We must establish it  for all $i$, so fix some $i\in I$.
We have
\begin{align*}
&\htilde_i(r)\x_{\sstar_j(e)}(t)\htilde_i(r)^{-1}
\\
&{}=
\htilde_i(r)\,\x_{{\sstar_j}^{-1}\circ h_j(-1)(e)}(t)\,\htilde_i(r)^{-1}
&\llap{by $(\sstar_j)^2=h_j(-1)$}
\\
&{}\equiv
\htilde_i(r)\,\stilde_j^{-1}\,\x_{e}\bigl((-1)^{\pairing{\alpha_j^\vee}{\beta}}t\bigr)\,\stilde_j\,\htilde_i(r)^{-1}
\\
&=
\stilde_j^{-1}\bigl(\stilde_j\htilde_i(r)\stilde_j^{-1}\bigr)
\,\x_{e}\bigl((-1)^{\pairing{\alpha_j^\vee}{\beta}}t\bigr)
\bigl(\stilde_j\htilde_i(r)^{-1}\stilde_j^{-1}\bigr)\stilde_j\\
&\equiv
\stilde_j^{-1}\bigl(\htilde_j(r^{A_{i j}})^{-1}\htilde_i(r)\bigr)
\x_{e}\bigl((-1)^{\pairing{\alpha_j^\vee}{\beta}}t\bigr)
\bigl(\htilde_i(r)^{-1}\htilde_j(r^{A_{i j}})\bigr)\stilde_j
\kern40pt
&
\llap{by \eqref{eq-s-tilde-j-acting-on-h-tilde-i-of-r}}
\\
&{}\equiv
\stilde_j^{-1}\,\x_{e}\Bigl(
(-1)^{\pairing{\alpha_j^\vee}{\beta}}
r^{\pairing{\alpha_i^\vee}{\beta}}
r^{-A_{i j}\pairing{\alpha_j^\vee}{\beta}}
t
\Bigr)\,\stilde_j
&\llap{by \eqref{eq-to-establish-torus-action-on-root-groups} for $e$}
\\
&{}\equiv
\x_{{\sstar_j}^{-1}(e)}\Bigl(
(-1)^{\pairing{\alpha_j^\vee}{\beta}}
r^{\pairing{\alpha_i^\vee}{\beta}}
r^{-A_{i j}\pairing{\alpha_j^\vee}{\beta}}
t
\Bigr)
\\
&{}=
\x_{\sstar_j\circ h_j(-1)(e)}\Bigl(
(-1)^{\pairing{\alpha_j^\vee}{\beta}}
r^{\pairing{\alpha_i^\vee}{\beta}}
r^{-A_{i j}\pairing{\alpha_j^\vee}{\beta}}
t
\Bigr)
\\
&{}=
\x_{\sstar_j(e)}\Bigl(
r^{\pairing{\alpha_i^\vee}{\beta}}
r^{-A_{i j}\pairing{\alpha_j^\vee}{\beta}}
t
\Bigr).
\end{align*}
The statement of \eqref{eq-to-establish-torus-action-on-root-groups} for $\sstar_j(e)$ has a similar form.
Deducing it  amounts to showing
$\pairing{\alpha_i^\vee-A_{i j}\alpha_j^\vee}{\beta}=\pairing{\alpha_i^\vee}{s_j(\beta)}$.
This follows from
$s_j(\beta)=\beta-\pairing{\alpha_j^\vee}{\beta}\alpha_j$,
finishing the proof of
\eqref{eq-to-establish-torus-action-on-root-groups} for all $e\in E$.

For $e$ equal to any $\pm e_i$, we were given
\eqref{eq-Morita-Rehmann-relations-just-the-s-1-tildes} and we have
proven \eqref{eq-Morita-Rehmann-relations-just-the-h-tildes-e}.  The
same results for all $e$ follow by $\Wstar$ symmetry.  More precisely,
we claim that for all $j\in I$: if
\eqref{eq-Morita-Rehmann-relations-just-the-s-1-tildes} and
\eqref{eq-Morita-Rehmann-relations-just-the-h-tildes-e} hold for some
$e\in E$ then they hold for $\sstar_j(e)$ too.  We give the details
for \eqref{eq-Morita-Rehmann-relations-just-the-h-tildes-e}, and the
argument is the same for
\eqref{eq-Morita-Rehmann-relations-just-the-s-1-tildes}.  Suppose
$r\in\Runits$ and $u\in\cup_{\beta\in\Phi}\,\U_\beta$.  Then the left
and right ``sides'' of the known relation
\eqref{eq-Morita-Rehmann-relations-just-the-h-tildes-e} for $e$ lie in
$\cup_{\beta\in\Phi}\,\U_\beta$, so conjugating the left by
$\stilde_j$ has the same result as applying $\sstar_j$ to the right.
That is,
\begin{align*}
\stilde_j\,\stilde_e u\stilde_e^{-1}\,\stilde_j^{-1}
&{}\equiv
\sstar_j\circ\sstar_e(u)
\\
(\stilde_j\stilde_e\stilde_j^{-1})
\,
(\stilde_j u\stilde_j^{-1})
\,
(\stilde_j\stilde_e^{-1}\stilde_j^{-1})
&{}\equiv
\sstar_j\circ\sstar_e\circ{\sstar_j}^{-1}\circ\sstar_j (u)
\\
\sstar_{\sstar_j(e)}\,\sstar_j(u)\,(\sstar_{\sstar_j(e)})^{-1}
&{}\equiv
\sstar_{\sstar_j(e)}\bigl(\sstar_j (u))
\end{align*}
As $u$ varies over all of $\cup_{\beta\in\Phi}\,\U_\beta$, so does 
$\sstar_j(u)$.  This verifies relation \eqref{eq-Morita-Rehmann-relations-just-the-h-tildes-e} for
$\sstar_j(e)$. 
\end{proof}

\section{The pre-Steinberg group $\PSt$}
\label{sec-pre-Steinberg-group}

\noindent
In this section we define the pre-Steinberg group functor $\PSt_A$ in
the same way as $\St_A$, but omitting some of its Chevalley relations.
So it has a natural map to $\St_A$.  Then we will write down
another group functor as a concrete presentation, and show in
theorem~\ref{thm-G-4-is-PSt} that it equals $\PSt_A$.  Since
$\PSt_A\to\St_A$ is often an isomorphism (theorem~\ref{thm-examples-of-PSt-to-St-being-an-isomorphism}),
this often gives a new presentation for $\St_A$.  As discussed in the
introduction, it is simpler and more explicit than previous
presentations, and special cases of it appear in table~\ref{tab-simply-laced-presentation} and
section~\ref{sec-examples}.  In the rest of this section we suppress
the subscript~$A$.

We call two roots $\alpha,\beta$ classically
prenilpotent if 
$(\Q\alpha+\Q\beta)\cap\Phi$ is finite and $\alpha+\beta\neq0$.  
Then they are
prenilpotent, and lie in some $A_1$, $A_1^2$, $A_2$, $B_2$ or $G_2$ root
system.  We define the pre-Steinberg group functor $\PSt$ exactly as
we did the Steinberg functor $\St$
(section~\ref{sec-Steinberg-group}), except that when imposing the
Chevalley relations we only vary $\alpha,\beta$ over the classically
prenilpotent pairs rather than all prenilpotent pairs.  We still
impose the relations \eqref{eq-Morita-Rehmann-relations-abstract-form} of Morita-Rehmann, or equivalently
\eqref{eq-Morita-Rehmann-relations-just-the-s-1-tildes}--\eqref{eq-Morita-Rehmann-relations-just-the-h-tildes-e} or
\eqref{eq-alternative-Steinberg-relations-h-action-on-simple-root-groups}--\eqref{eq-alternative-Steinberg-relations-s-action-on-everything}.  (See remark~\ref{remark-do-not-need-St-Tits-relations} for why theorem~\ref{thm-alternate-relations-for-Steinberg-group} applies
with $\PSt$ in place of $\St$.)  Just as for $\St$, $\Wstar$ acts on
$\PSt$ because it permutes the defining relators.

There is an obvious natural map $\PSt\to\St$, got by imposing the
remaining Chevalley relations, coming from prenilpotent pairs that are not
classically prenilpotent.  If $\Phi$ is finite then every prenilpotent
pair is classically prenilpotent, so $\PSt\to\St$ is an
isomorphism.  

\medskip
The rest of this section is devoted to writing down a presentation for
$\PSt$.  We start by defining an analogue $\What$ of the Weyl group.
It is the quotient of the free
group on formal symbols $S_{i\in I}$ by the subgroup normally
generated by the words
\begingroup \newcommand{\placement}[1]{\rlap{\kern110pt#1}}
\begin{align}
\label{eq-def-of-What-Artin-relators}
\bigl(S_i S_j\cdots{}\bigr)
&\placement{if $m_{i j}\neq\infty$}
{}\cdot
\bigl(S_j S_i\cdots{}\bigr)^{-1}
\\
\label{eq-def-of-What-commutators-of-squares-even-case}
\SsquaredActionEven
\\
\label{eq-def-of-What-commutators-of-squares-odd-case}
\SsquaredActionOdd
\end{align}
where $i,j$ vary over $I$, and \eqref{eq-def-of-What-Artin-relators} has $m_{i j}$ terms inside each
pair of parentheses, alternating between $S_i$ and $S_j$.  These are
called the Artin relators,  for
example $S_i S_j S_i\cdot(S_j S_i S_j)^{-1}$ if $m_{i j}=3$.    

\begin{remark}
We chose these defining relations so that $\What$ would have four
properties.  First, it maps naturally to $\Wstar$, so that it acts on
$\g$ and $\freeproduct_{\alpha\in\Phi}\U_{\alpha}$.  Second, the
kernel of $\What\to W$ is generated (not just normally) by
the~$S_i^2$.  This plays a key role in the proof of
theorem~\ref{thm-G-2-is-free-product-semidirect-What} below.  Third,
each relation involves just two subscripts, which is needed for the
Curtis-Tits property of $\PSt$
(corollary~\ref{cor-pre-Steinberg-as-direct-limit}).  And fourth,
the   $\stilde_i\in\St$,  defined in \eqref{eq-definition-of-s-tilde}, satisfy
the same relations.  (Formally: $S_i\to\stilde_i$ extends to a
homomorphism $\What\to\St$.)  The first two properties are established
in the next lemma, the third is obvious, and the fourth is part of
theorem~\ref{thm-G-4-is-PSt}.
\end{remark}

\begin{lemma}[Basic properties of $\What$]
\label{lem-basic-properties-of-What}
\leavevmode\hbox{}
\begin{enumerate}
\item
\label{item-What-surjects-to-Wstar}
$S_i\mapsto\sstar_i$ defines a surjection $\What\to\Wstar$.  
\item
\label{item-conjugates-of-squares-in-What}
$S_j S_i^2 S_j^{-1}=S_i^2$ resp.\ $S_j^2 S_i^2$ if
  $A_{i j}$ is even resp.\ odd.
\item
\label{item-squares-generate-kernel-of-What-to-W}
The $S_i^2$ generate
the kernel of the composition $\What\to\Wstar\to W$.
\end{enumerate}
\end{lemma}

\begin{proof}
We saw in theorem~\ref{thm-Artin-relations-in-W-star} that the $\sstar_i$
satisfy the Artin
relations.  Rewriting lemma~\ref{lem-relations-on-s-stars}\ref{item-W-conjugacy-action-on-s-star-squared}'s relation in $\Wstar$ 
with $i$ and $j$ reversed gives
$$
\sstar_j(\sstar_i)^2\sstar_j{}^{-1}=(\sstar_i)^2(\sstar_j)^{-2A_{i j}}.
$$
Multiplying on the left by $\sstar_j{}^{-1}$ and on the right by
$(\sstar_i)^{-2}$, then inverting, gives
$$
(\sstar_i)^2\sstar_j(\sstar_i)^{-2}
=
(\sstar_i)^2
(\sstar_j)^{2A_{i j}}
(\sstar_i)^{-2}
\sstar_j
=
(\sstar_j)^{1+2A_{i j}}
$$ In the second step we used the fact that $\sstar_i{}^2$ and
$\sstar_j{}^2$ commute.  Using $\sstar_j{}^4=1$, the right side is
$\sstar_j$ if $A_{i j}$ is even and $\sstar_j{}^{-1}$ if $A_{i j}$ is
odd.  This shows that $S_i\mapsto\sstar_i$ sends the relators
\eqref{eq-def-of-What-commutators-of-squares-even-case}--\eqref{eq-def-of-What-commutators-of-squares-odd-case}
to the trivial element of $\Wstar$, proving
\ref{item-What-surjects-to-Wstar}.

One can manipulate \eqref{eq-def-of-What-commutators-of-squares-even-case}--\eqref{eq-def-of-What-commutators-of-squares-odd-case} in a similar way, yielding 
\ref{item-conjugates-of-squares-in-What}.  It follows immediately
that the subgroup generated by the $S_i^2$ is normal.  Because of the
Artin relations, this is
the kernel of $\What\to W$.  So we have proven \ref{item-squares-generate-kernel-of-What-to-W}.
\end{proof}

\begin{remark}
Though we don't need them, the following relations  in $\What$
show that $\What$ is ``not much larger'' than $\Wstar$.
First \eqref{eq-def-of-What-commutators-of-squares-even-case}--\eqref{eq-def-of-What-commutators-of-squares-odd-case} imply the centrality of every $S_i^4$.
Second, if some $A_{i j}$
is odd then \eqref{eq-def-of-What-commutators-of-squares-odd-case} shows that $S_j^{\pm4}$
are conjugate; since both are central they must be equal, so
$S_j^8=1$.  
Third, the relation obtained at the end of the proof implies
$[S_j^2,S_i^2]=1$ or $S_j^4$, according to whether $A_{i j}$ is even or odd.
In particular, these commutators are central.  Finally, we can use
this twice:
$$
\left\{
\begin{matrix}
1&\hbox{if $A_{i j}$ is even}\\
S_j^4&\hbox{if $A_{i j}$ is \rlap{odd}\phantom{even}}
\end{matrix}
\right\}
=[S_j^2,S_i^2]
=[S_i^2,S_j^2]^{-1}
=
\left\{
\begin{matrix}
1&\hbox{if $A_{j i}$ is even}\\
S_i^{-4}&\hbox{if $A_{j i}$ is \rlap{odd}\phantom{even}}
\end{matrix}
\right\}
$$ In particular, if both $A_{i j}$ and
$A_{j i}$ are odd then $S_i^4$ and $S_j^4$ are equal.
If $A_{i j}$ is even while $A_{j i}$ is odd then we get $S_i^4=1$.
\end{remark}

\medskip
Now we begin our presentation in earnest.  Ultimately, $\PSt(R)$ will
have generators $S_i$ and $X_i(t)$, with $i$ varying over~$I$ and $t$
varying over~$R$, and relators \eqref{eq-def-of-What-Artin-relators}--\eqref{eq-collapse-to-PSt}.  

We first define a group functor $\G_1$ by declaring that
$\G_1(R)$ is the quotient of the free group on the formal symbols
$X_i(t)$, by the subgroup normally generated by the relators
\begin{equation}
\label{eq-additive-relators-in-root-groups}
\AdditiveRlnsInRootGroups
\end{equation}
for all $i\in I$ and $t,u\in R$.  The following description of $\G_1$
is obvious.

\begin{lemma}
\label{lem-description-of-G-1}
$\G_1\iso\freeproduct_{i\in
  I}\U_i$, via the correspondence $X_i(t)\leftrightarrow\x_i(t)$.
\qed
\end{lemma}

Next we define a group functor $\G_2$ as a certain quotient of the free
product $\G_1\freeproduct\What$.  Namely, $\G_2(R)$
is the quotient of $\G_1(R)\freeproduct\What$ by the subgroup normally
generated by the following relators, with $i$ and $j$ varying over $I$
and $t$ over $R$.
\begin{align}
\label{eq-S-i-squared-action-on-X-j}
\SiSquaredActionOnXj
\\
\label{eq-S-i-X-j-relator-when-m-is-2}
\SiXjRelationWhenMisTwo
\\
\label{eq-S-i-X-j-relator-when-m-is-3}
\SiXjRelationWhenMisThree
\\
\label{eq-S-i-X-j-relator-when-m-is-4}
\SiXjRelationWhenMisFour
\\
\label{eq-S-i-X-j-relator-when-m-is-6}
\SiXjRelationWhenMisSix
\end{align}
\endgroup The next theorem is the key step in our development; see
section~\ref{sec-proof-of-free-by-semidirect-product-theorem} for the
proof.  Although it is not at all obvious, we have presented
$\bigl(\freeproduct_{\alpha\in\Phi}\,\U_\alpha\bigr)\semidirect\What$. Therefore we
``have'' the root groups $\U_\alpha$ for all $\alpha$, not just
simple~$\alpha$.  This sets us up for imposing the Chevalley relations
in the next step.

\begin{theorem}
\label{thm-G-2-is-free-product-semidirect-What}
$\G_2$ is the semidirect product of
$\freeproduct_{\alpha\in\Phi}\,\U_{\alpha}$ by $\What$, 
where $\What$ acts on the
free product via its homomorphism to $\Wstar$ and
$\Wstar$'s action on $\freeproduct_{\alpha\in\Phi}\U_{\alpha}$ is induced by its action on
$\oplus_{\alpha\in\Phi}\,\g_{\alpha,\Z}$. 
\end{theorem}

\begin{remark}[Groups with a root group datum]
\label{remark-groups-with-a-root-group-datum}
A Kac-Moody group over a field is an example of a group $G$ with a ``root group
datum''.   This means: a generating set of subgroups
$\U_{\alpha}$ parameterized by the roots $\alpha$ of a root system, permuted by
(some extension $\widetilde{W}$ of) the Weyl group $W$ of that root
system, and satisfying some additional hypotheses.  See
\cite{Tits-twin-buildings} or \cite{Caprace-Remy} for details.
Examples include the
Suzuki and Ree groups and isotropic forms of algebraic groups (or
Kac--Moody groups) over fields.  In many of these cases, some of
the root groups are non-commutative.  
The heart of the proof of theorem~\ref{thm-G-2-is-free-product-semidirect-What} is our understanding of root stabilizers in
$\Wstar$ (theorem~\ref{thm-root-stabilizers-in-W-star}), which would still apply in this more general
setting.  
So there should be an analogous presentation of 
$\bigl(\freeproduct_{\alpha\in\Phi}\,\U_{\alpha}\bigr)\semidirect\widetilde{W}$.
The main change would be to 
replace \eqref{eq-additive-relators-in-root-groups} by defining relations for
$\U_i$, and interpret the parameter $t$ of $X_i(t)$ as
varying over some fixed copy of $\U_i$, rather than over $R$.
Since  $G$ is a quotient of $\bigl(\freeproduct_{\alpha\in\Phi}\,\U_{\alpha}\bigr)\semidirect\widetilde{W}$, analogues of the rest of this section presumably yield 
a presentation of $G$. 
\end{remark}

Next we adjoin Chevalley relations corresponding to finite edges in
the Dynkin diagram.  That is, we define $\G_3(R)$ as the quotient of
$\G_2(R)$ by the subgroup normally generated by the relators
\eqref{eq-Chevalley-relator-m=2}--\eqref{eq-Chevalley-relators-m=6-distant-short-and-long}
below, for all $t,u\in R$.  
These are particular cases of the standard
Chevalley relators, written in a form due to Demazure (see  remark~\ref{remark-Demazure-form-of-relations}
below).  

\smallskip
\noindent
%----------------------------------A1+A1----------------------------
When $i,j\in I$ with $m_{i j}=2$, 
\begin{align}
\label{eq-Chevalley-relator-m=2}
\ChevRlnWhenMisTwo
\\
%-----------------------------------A2-----------------------------
\noalign{\noindent When $i,j\in I$ with $m_{i j}=3$,}
\label{eq-Chevalley-relators-m=3-close-roots}
\ChevRlnWhenMisThreeCloseRoots
\\
\label{eq-Chevalley-relators-m=3-distant-roots}
\ChevRlnWhenMisThreeDistantRoots
\\
%---------------------------------B2-----------------------------
\noalign{\noindent When $s,l\in I$, $m_{s l}=4$ and $s$ is the shorter root of the
$B_2$,}
\label{eq-Chevalley-relators-m=4-adjacent-short-and-long}
\ChevRlnWhenMisFourAdjacentShortAndLong
\\
\label{eq-Chevalley-relators-m=4-orthogonal-long}
\ChevRlnWhenMisFourOrthogonalLong
\\
\label{eq-Chevalley-relators-m=4-orthogonal-short}
\ChevRlnWhenMisFourOrthogonalShort
\\
\label{eq-Chevalley-relators-m=4-distant-short-and-long}
\ChevRlnWhenMisFourDistantShortAndLong
\\
%------------------------------G2-----------------------------
\noalign{\noindent When $s,l\in I$, $m_{s l}=6$ and $s$ is the shorter root of
the $G_2$,}
\label{eq-Chevalley-relators-m=6-nearby-long}
\ChevRlnWhenMisSixNearbyLong
\\
\label{eq-Chevalley-relators-m=6-adjacent-short-and-long}
% this is some cheesy hand-formatting of a left-hand-side that is just
% too too big.
\rlap{$\ChevRlnWhenMisSixAdjacentShortAndLong$}\kern128pt
\\
\label{eq-Chevalley-relators-m=6-orthogonal-short-and-long}
\ChevRlnWhenMisSixOrthogonalShortAndLong
\\
\label{eq-Chevalley-relators-m=6-distant-long}
\ChevRlnWhenMisSixDistantLong
\\
\label{eq-Chevalley-relators-m=6-nearby-short}
\ChevRlnWhenMisSixNearbyShort
\\
\label{eq-Chevalley-relators-m=6-distant-short}
\ChevRlnWhenMisSixDistantShort
\notag% since this relation spans two lines
\\
\label{eq-Chevalley-relators-m=6-distant-short-and-long}
\ChevRlnWhenMisSixDistantShortAndLong
\notag% since this relation spans two lines
\end{align}

\begin{remark}[Asymmetry in the $A_2$ relators]
The relators \eqref{eq-Chevalley-relators-m=3-close-roots}--\eqref{eq-Chevalley-relators-m=3-distant-roots}
are not symmetric in $i$ and $j$.  Since $m_{j i}=3$
whenever $m_{i j}=3$, we are using both these relators
and the ones got from them by exchanging $i$ and $j$.
\end{remark}

\begin{remark}[Demazure's form of the Chevalley relations]
\label{remark-Demazure-form-of-relations}
Our relators are written in a form due to Demazure
 (Props. 3.2.1, 3.3.1 and 3.4.1 in
\cite[Exp.\ XXIII]{Demazure-Grothendieck}).  
They
appear more complicated than the more usual one (for
example \cite[thm.\ 5.2.2]{Carter}), but have two important advantages.  First, there are
no implicit signs to worry about, and second, the presentation refers
only to the Dynkin diagram, rather than the full root system.  

One can convert \eqref{eq-Chevalley-relator-m=2}--\eqref{eq-Chevalley-relators-m=6-distant-short-and-long}
to a more standard form by working out which root groups contain the
terms on the ``right hand sides'' of the relators.  For example, the
term $S_l X_s(t u)S_l^{-1}$ of
\eqref{eq-Chevalley-relators-m=6-distant-short-and-long} lies in
$S_l\,\U_s S_l^{-1}=\U_{\alpha_s+\alpha_l}$ because reflection in $\alpha_l$ sends $\alpha_s$ to
$\alpha_s+\alpha_l$.  Applying the same reasoning to the other terms,
\eqref{eq-Chevalley-relators-m=6-distant-short-and-long} equals
$[X_s(t),X_l(u)]$ times a particular element of
$\U_{3\alpha_s+2\alpha_l}
\cdot
\U_{2\alpha_s+\alpha_l}
\cdot
\U_{\alpha_s+\alpha_l}
\cdot
\U_{2\alpha_s+\alpha_l}
$.
The advantages of Demazure's form of the relators come from the fact that no
identifications of these root groups with~$R$ is required.  
We simply use the already-fixed identifications of the simple root
groups with $R$, and transfer them to these other root groups by
conjugation by $S_s$ and~$S_l$.
\end{remark}

\begin{remark}[Diagram automorphisms in characteristics $2$ and~$3$]
Some of the relators can be written in simpler but less-symmetric
ways.  For example,
\eqref{eq-Chevalley-relators-m=4-adjacent-short-and-long} is the
Chevalley relator for the roots $s_s(\alpha_l)$ and $s_l(\alpha_s)$ of
$B_2$,
which make angle~$\pi/4$.  As we will see in the proof of
theorem~\ref{thm-G3-is-spherical-Tits-Steinberg-group-semidirect-What},
one could replace this pair of roots by any other pair of roots in the
span of $\alpha_s,\alpha_l$ that make this angle.  So for example one
could replace
\eqref{eq-Chevalley-relators-m=4-adjacent-short-and-long} by the
simpler relator $[S_s X_l(t)S_s^{-1},X_s(u)]$.  We prefer
\eqref{eq-Chevalley-relators-m=4-adjacent-short-and-long} because it
maps to itself under the exceptional diagram automorphism in
characteristic~$2$;  see section~\ref{sec-diagram-automorphisms} for
details.  Similar considerations informed our choice of relators
\eqref{eq-Chevalley-relators-m=6-adjacent-short-and-long}--\eqref{eq-Chevalley-relators-m=6-orthogonal-short-and-long},
and the ordering of the last four terms of \eqref{eq-Chevalley-relators-m=6-distant-short-and-long}.
\end{remark}

\begin{remark}[Redundant relations]
In practice most of the relators coming from absent and single bonds
in the Dynkin diagram, i.e., 
\eqref{eq-Chevalley-relator-m=2}--\eqref{eq-Chevalley-relators-m=3-distant-roots},
can be omitted.  Usually this reduces the size of the presentation
greatly.  See propositions \ref{prop-omitting-m-i-j=2-relations} and~\ref{prop-omitting-m-i-j=3-relations}.
\end{remark}

In section~\ref{sec-step-3} we prove the following
more conceptual description of~$\G_3$.  To be able to state it we
use the temporary notation $\PStTits$ for the group functor defined in
the same way as $\StTits$ (see section~\ref{sec-Steinberg-group}), but only using classically prenilpotent
pairs rather than all prenilpotent pairs.  So $\PStTits$ is related to~$\StTits$ in the same way that $\PSt$ is related to~$\St$.
$\What$ acts on $\PStTits$ for
the same reason it acts on $\StTits$.

\begin{theorem}
\label{thm-G3-is-spherical-Tits-Steinberg-group-semidirect-What}
The group functor $\PStTits\semidirect\What$ coincides with
$\G_3$.
More precisely, under the identification
$\G_2\iso\bigl(\freeproduct_{\alpha\in\Phi}\,\U_{\alpha}\big)\semidirect\What$ of
theorem~\ref{thm-G-2-is-free-product-semidirect-What}, the kernels of $\G_2\to\G_3$ and
$\bigl(\freeproduct_{\alpha\in\Phi}\,\U_{\alpha}\big)\semidirect\What\to\PStTits\semidirect\What$
coincide.  
\end{theorem}

\medskip
Finally, we define $\G_4$ as the quotient of $\G_3$ by the smallest
normal subgroup containing the relators
\begin{align}
\label{eq-torus-action-1}
\TorusActionOne
\\
\label{eq-torus-action-2}
\TorusActionTwo
\\
\label{eq-collapse-to-PSt}
\WhatCollapse
\end{align}
where $r$ varies over $\Runits$, $t$ over $R$ and $i,j$ over $I$.  
We are using the definitions
\begin{align}
\label{eq-definition-of-s-tilde}
\DefnOfstilde
\\
\label{eq-definition-of-h-tilde}
\DefnOfhtilde.
\end{align}
Note that this definition of $\stilde_i(r)$ is compatible with the one
in section~\ref{sec-Steinberg-group}, because $X_i(r)\in\G_3$ corresponds to
$\x_{e_i}(r)\in\PStTits$ under the isomorphism of lemma~\ref{lem-description-of-G-1}, while
$S_i X_i(1/r) S_i^{-1}$ corresponds to
$\sstar_i\bigl(\x_{e_i}(1/r)\bigr)=\x_{f_i}(1/r)$.  As before, we will
abbreviate $\stilde_i(1)$ to $\stilde_i$.

The following theorem is the main result of this section and a
restatement of theorem~\ref{thm-definition-of-PSt-as-presentation}
from the introduction.

\begin{theorem}[Presentation of the pre-Steinberg group $\PSt$]
\label{thm-G-4-is-PSt}  The group functor
$\PSt$ coincides with $\G_4$.  In particular, for any commutative ring $R$,
$\PSt(R)$ has a presentation with generators $S_i$ and $X_i(t)$ for
$i\in I$ and $t\in R$, and relators \eqref{eq-def-of-What-Artin-relators}--\eqref{eq-collapse-to-PSt}.
\end{theorem}

\begin{proof}
By definition, $\G_4$ is the quotient of $\G_3$ by the relations
\eqref{eq-torus-action-1}--\eqref{eq-collapse-to-PSt}.  Because $S_i$
acts on each $\U_\beta$ by $\sstar_i$ (theorem~\ref{thm-G-2-is-free-product-semidirect-What}), imposing
\eqref{eq-collapse-to-PSt} forces $\stilde_i$ to also act this way.
We consider the intermediate group $\G_{3.5}$, of fleeting interest,
got from $\G_3$ by imposing
\eqref{eq-torus-action-1}--\eqref{eq-torus-action-2} and the relations
that $\stilde_i$ acts on every $\U_\beta$ as $\sstar_i$ does.  In
other words, we are imposing on
$\PStTits\sset\PStTits\semidirect\What=\G_3$ the relations
\eqref{eq-alternative-Steinberg-relations-h-action-on-simple-root-groups}--\eqref{eq-alternative-Steinberg-relations-s-action-on-everything}.
Theorem~\ref{thm-alternate-relations-for-Steinberg-group} 
and remark~\ref{remark-do-not-need-St-Tits-relations} show
that this reduces $\G_3$ to $\PSt\semidirect\What$.

So $\G_4$ is the quotient of $\G_{3.5}=\PSt\semidirect\What$ by the relations
$S_i=\stilde_i$.  We use Tietze transformations to eliminate the $S_i$'s
from the presentation, in favor of the $\stilde_i$'s.  So $\G_4$ is
the quotient of $\PSt$ by the subgroup normally generated by the words
got by replacing $S_i$ by $\stilde_i$ in each of the relators
\eqref{eq-def-of-What-Artin-relators}--\eqref{eq-torus-action-2}.  All
of these relators are already trivial in $\PSt$, so $\G_4=\PSt$.

In more detail, \eqref{eq-def-of-What-Artin-relators} requires the
$\stilde_i$ to satisfy the Artin relations, which they do in $\PSt$ by
\cite[(d) on p.~551]{Tits}.  The remaining relations
\eqref{eq-def-of-What-commutators-of-squares-even-case}--\eqref{eq-torus-action-2}
involve the $S_i$'s only by their conjugacy action. For example
\eqref{eq-Chevalley-relators-m=6-nearby-long} says that $X_l(t)$
commutes with the conjugate of $X_l(u)$ by a certain word in $S_s$ and
$S_l$.  Since $S_i$ acts as $\sstar_i$ by
theorem~\ref{thm-G-2-is-free-product-semidirect-What} and $\stilde_i$
acts the same way by the definition of $\PSt$, these relations still
hold after replacing each $S_i$ by the corresponding $\stilde_i$.
(When defining $\What$ we were careful not to impose any relations on
the $S_i$'s except those which are also satisfied by the $\stilde_i$'s.)
\end{proof}

\begin{remark}[Redundant relators]
In most cases of interest, $A$ is $2$-spherical without $A_1$
components.  Then one can forget the relators \eqref{eq-torus-action-1}--\eqref{eq-torus-action-2} because
they follow from previous relations.  More specifically,
suppose $m_{i j}$ is $3$, $4$ or~$6$.  Then the relators
\eqref{eq-torus-action-1}--\eqref{eq-torus-action-2} are already trivial in $\G_3$.  The same holds if
$i=j$ and there exists some $k\in I$ with $m_{i k}\in\{3,4,6\}$.  See
\cite[p.~550, (a$_4$)]{Tits} for details.
\end{remark}

\begin{remark}[More redundant relators]
One need only impose the relators \eqref{eq-collapse-to-PSt} for a
single $i$ in each component $\Omega$ of the ``odd Dynkin diagram''
$\odddiagram$ considered in section~\ref{sec-W-star}.  This is because if
$m_{i j}=3$ then $S_i S_j$ conjugates $S_i$ to $S_j$ and $X_i(t)$ to
$X_j(t)$.  This uses relators \eqref{eq-def-of-What-Artin-relators} and \eqref{eq-S-i-X-j-relator-when-m-is-3}.
\end{remark}

\begin{remark}[Precautions against typographical errors]
\label{rem-precautions}
  We found explicit matrices for our generators, in standard
representations of the $A_1^2$, $A_2$, $B_2$ and $G_2$ Chevalley
groups over $\Z[r^{\pm1},t,u]$.  Then
we checked on the computer that they satisfy the defining relations
\eqref{eq-def-of-What-Artin-relators}--\eqref{eq-collapse-to-PSt}.
For \eqref{eq-def-of-What-commutators-of-squares-even-case}--\eqref{eq-collapse-to-PSt}
we only typed in the relations once, for both
typesetting and this check.
\end{remark}

\section{The isomorphism $\G_2\iso\bigl(\freeproduct_{\alpha\in\Phi}\U_{\alpha}\bigr)\semidirect\What$}
\label{sec-proof-of-free-by-semidirect-product-theorem}

\noindent
In this section we will suppress the dependence of group functors on
the base ring $R$, always meaning the group of points over $R$.  Our
goal is to prove
theorem~\ref{thm-G-2-is-free-product-semidirect-What}, namely that the
group $\G_2$ with generators $S_i$ and $X_i(t)$, $i\in I$ and $t\in
R$, 
modulo the  subgroup normally generated by the relators
\eqref{eq-def-of-What-Artin-relators}--\eqref{eq-S-i-X-j-relator-when-m-is-6},
is $\bigl(\freeproduct_{\alpha\in\Phi}\U_{\alpha}\bigr)\semidirect\What$.  The
genesis of the theorem is the following elementary principle.  It
seems unlikely to be new, but I have not seen it before.

\begin{lemma}
\label{lem-general-finite-presentation-principle}
Suppose $G=\bigl(\freeproduct_{\alpha\in\Phi}\,U_{\alpha}\bigr)\semidirect H$, where
$\Phi$ is some index set, the $U_{\alpha}$'s are groups isomorphic to each other, and $H$ is a group
whose action on the free product permutes the displayed
factors transitively.  
Then $G\iso
\bigl(U_\infty\semidirect H_\infty\bigr)*_{H_\infty}H$, where $\infty$
is some element of $\Phi$ and $H_\infty$ is its $H$-stabilizer.
\end{lemma}

\begin{proof}
The idea is that $U_\infty\semidirect H_\infty\mapsto(U_\infty\semidirect H_\infty)*_{H_\infty} H$ is a sort of
free-product analogue of inducing a representation from $H_\infty$ to
$H$.  We suppress the subscript $\infty$ from $U_\infty$.
Take a set $Z$ of left coset representatives for $H_\infty$ in $H$, and for
$u\in U$ and $z\in Z$ define $u_z:=z u z^{-1}\in G$.  The $u_z$ for
fixed $z$ form the free factor $z U z^{-1}=U_{z(\infty)}$ of $\bigl(\freeproduct_{\alpha\in\Phi}
U_{\alpha}\bigr)\sset G$.  Assuming $U\neq1$, every displayed free factor occurs exactly once
this way, since $H$'s action on $\Phi$ is the same as on $H_\infty$'s left
cosets.  So the maps $u_z\mapsto z u z^{-1}\in(U\semidirect H_\infty)*_{H_\infty}H$ define a
homomorphism $\bigl(\freeproduct_{\alpha\in\Phi}U_{\alpha}\bigr)\to(U\semidirect H_\infty)*_{H_\infty}H$.
This homomorphism is obviously $H$-equivariant, so it extends to a
homomorphism $G\to(U\semidirect H_\infty)*_{H_\infty}H$.  It is easy to see that this
is inverse to the obvious homomorphism $(U\semidirect H_\infty)*_{H_\infty}H\to
G$.
\end{proof}

Now we begin proving
theorem~\ref{thm-G-2-is-free-product-semidirect-What} by reducing it
to lemma~\ref{lem-free-product-model-for-one-component-of-odd-diagram}
below, which is an analogue of
theorem~\ref{thm-G-2-is-free-product-semidirect-What} for a single
component of the ``odd Dynkin diagram'' $\odddiagram$ introduced in
section~\ref{sec-W-star}.
It is well-known that two generators $s_i$, $s_j$ of
$W$ ($i,j\in I$) are conjugate in $W$ if and only if $i$ and $j$ lie
in the same component of $\odddiagram$.  (If $m_{ij}=3$ then $s_i s_j
s_i=s_j s_i s_j$ implies the conjugacy of $s_i$ and $s_j$, while
distinct components of $\odddiagram$ correspond to different elements
of the abelianization of $W$.)

Let $\Omega$ be one of these
components, and write $\Phi(\Omega)\sset\Phi$ for the roots whose
reflections are conjugate to some (hence any) $s_{i\in\Omega}$.
Because $\Phi(\Omega)$ is a $W$-invariant subset of $\Phi$, we may
form the group
$\bigl(\freeproduct_{\alpha\in\Phi(\Omega)}\,\U_{\alpha}\bigr)\semidirect\What$
just as we did
$\bigl(\freeproduct_{\alpha\in\Phi}\,\U_{\alpha}\bigr)\semidirect\What$.
We will write $\G_{2,\Omega}$ for the group having generators $S_i$,
with $i\in I$, and $X_i(t)$, with $i\in\Omega$ and $t\in R$, 
modulo the subgroup normally generated by the relators
\eqref{eq-def-of-What-Artin-relators}--\eqref{eq-def-of-What-commutators-of-squares-odd-case},
and those relators
\eqref{eq-additive-relators-in-root-groups}--\eqref{eq-S-i-X-j-relator-when-m-is-6}
with $i\in\Omega$.  Note that \eqref{eq-S-i-X-j-relator-when-m-is-3}
is relevant only if $m_{i j}=3$, in which case $i\in\Omega$ if and only
if
$j\in\Omega$, so the relator makes sense.  
{\bf Caution:} the
subscripts on $S$ vary over all of $I$ while those on $X$ vary only
over $\Omega\sset I$.

\begin{lemma}
\label{lem-free-product-model-for-one-component-of-odd-diagram}
For any component $\Omega$ of $\odddiagram$,
$$\G_{2,\Omega}\iso\Bigl(\,\freeproduct_{\alpha\in\Phi(\Omega)}\,\U_{\alpha}\Bigr)\semidirect\What.$$
\end{lemma}

\begin{proof}[Proof of theorem~\ref{thm-G-2-is-free-product-semidirect-What}, given
    lemma~\ref{lem-free-product-model-for-one-component-of-odd-diagram}]
An examination of the presentation of $\G_2$ reveals
that the $X$'s corresponding to different components of $\odddiagram$
don't interact.  Precisely: $\G_2$ is the amalgamated free product
of the $\G_{2,\Omega}$'s, where $\Omega$ varies over the components of
$\odddiagram$ and  the amalgamation is that the copies
of $\What$ in the $\G_{2,\Omega}$'s are identified in the obvious way.  
Lemma~\ref{lem-free-product-model-for-one-component-of-odd-diagram} shows that
$\G_{2,\Omega}=\bigl(\freeproduct_{\alpha\in\Phi(\Omega)}\,\U_{\alpha}\bigr)\semidirect\What$
for each $\Omega$.  Taking their free product, amalgamated along their copies of $\What$,
obviously yields 
$\bigl(\freeproduct_{\alpha\in\Phi}\,\U_{\alpha}\bigr)\semidirect\What$.
\end{proof}

The rest of the section is devoted to proving
lemma~\ref{lem-free-product-model-for-one-component-of-odd-diagram}.
So we fix a component $\Omega$ of $\odddiagram$ and phrase our problem
in terms of the free product
$F:=\bigl(\freeproduct_{j\in\Omega}\,\U_j\bigr)\freeproduct\What$.
This is the group with generators $S_{i\in I}$ and
$X_{j\in\Omega}(t)$, whose relations are
\eqref{eq-def-of-What-Artin-relators}--\eqref{eq-def-of-What-commutators-of-squares-odd-case}
and those cases of \eqref{eq-additive-relators-in-root-groups} with
$i\in\Omega$.  The heart of the proof of
lemma~\ref{lem-free-product-model-for-one-component-of-odd-diagram} is
to define normal subgroups $M,N$ of $F$ and show they are equal.  $M$
turns out to be normally generated by the relators from
\eqref{eq-S-i-squared-action-on-X-j}--\eqref{eq-S-i-X-j-relator-when-m-is-6}
for which $i\in\Omega$.  Given this, $\G_{2,\Omega}=F/M$ by definition.
The other group $F/N$ has a presentation like the one in
lemma~\ref{lem-general-finite-presentation-principle}.  But it
requires some preparation even to define, so we begin with an informal
overview.

Start with the presentation of $\G_{2,\Omega}$, and distinguish some
point $\infty$ of $\Omega$ and a spanning tree $T$ for $\Omega$.  We
will use the relators \eqref{eq-S-i-X-j-relator-when-m-is-3} coming
from the edges of $T$ to rewrite the $X_{j\in\Omega-\{\infty\}}(t)$ in
terms of $X_\infty(t)$, and then eliminate the
$X_{j\in\Omega-\{\infty\}}(t)$ from the presentation.  This ``uses up''
those relators and makes the other relators messier because each
$X_{j\neq\infty}(t)$ must be replaced by a word in $X_\infty(t)$ and
elements of $\What$.  We studied the $\Wstar$-stabilizer of $\alpha_\infty$
in theorem~\ref{thm-root-stabilizers-in-W-star}, and how it acts on
$\g_\infty$, hence on $\U_\infty$.  It turns out that the remaining
relations in $\G_{2,\Omega}$ are exactly the relations that the
$\What$-stabilizer $\What_\infty$ of $\alpha_\infty$ acts on $\U_\infty$
via $\What_\infty\to\What\to\Wstar\sset\Aut\g$.  That is, $\G_{2,\Omega}\iso\bigl(\U_\infty\semidirect\What_\infty\bigr)\freeproduct_{\What_\infty}\What$.
Then lemma~\ref{lem-general-finite-presentation-principle}  identifies
this with
$\bigl(\freeproduct_{\alpha\in\Phi(\Omega)}\,\U_{\alpha}\bigr)\semidirect\What$.

Now we proceed to the formal proof, beginning by defining some
elements of $F$.  
For $\gamma$ an edge-path in $\Omega$, with
$i_0,\dots,i_n$ the vertices along it, define $\alpha(\gamma)=i_0$ and
$\w(\gamma)=i_n$ as its initial and final endpoints, and define $P_\gamma$ by
\eqref{eq-def-of-p-gamma} with $S$'s in place of $s$'s.  For $k\in I$
evenly joined to the end of $\gamma$ (i.e., $m_{k\w(\gamma)}$ finite and even),
define
$$
R_{\gamma,k}=
P_\gamma^{-1}\cdot\left\{
\begin{matrix}
S_k\\
S_k S_{\w(\gamma)}S_k\\
S_k S_{\w(\gamma)}S_k S_{\w(\gamma)}S_k
\end{matrix}
\right\}\cdot
P_\gamma
$$
according to whether $m_{k\w(\gamma)}=2$, $4$ or $6$.  ($R_{\gamma,k}$ is
got from \eqref{eq-def-of-r-gamma-k} by replacing $s$'s and $p$'s by $S$'s and $P$'s, and
$j$ by $\w(\gamma)$.)  Next, for $t\in R$ we define
\begin{align*}
C_\gamma(t)
&{}:=
P_\gamma X_{\alpha(\gamma)}(t)\cdot\Bigl(X_{\w(\gamma)}(t)P_\gamma\Bigr)^{-1}
\\
\noalign{\noindent and for $k\in I$ evenly joined to $\w(\gamma)$ we define}
D_{\gamma,k}(t)
&{}:=[R_{\gamma,k},X_{\alpha(\gamma)}(t)].
\\
\noalign{\noindent For ease of reference we will also give the name}
B_{i j}(t)
&{}:=
S_i^2X_j(t)S_i^{-2}\cdot X_j\bigl((-1)^{A_{i j}}t\bigr)^{-1}
\end{align*}
to the word \eqref{eq-S-i-squared-action-on-X-j}, where $i\in I$ and $j\in\Omega$.  We will
suppress the dependence of the $X_j$, $B_{i j}$, $C_\gamma$ and $D_{\gamma,k}$ on $t$
except where it plays a role.

The following formally-meaningless intuition may help the reader;
lemma~\ref{lem-G-2-description-compositions-of-paths} below
gives it some support.
The
relation $C_\gamma=1$ declares that the path $\gamma$ conjugates the
$X$ ``at'' the beginning of $\gamma$ to the $X$ ``at'' the end.  And
the relation $D_{\gamma,k}=1$ declares that the $X$ ``at'' the
beginning of $\gamma$ commutes with a certain word that corresponds to
going along $\gamma$, going around some sort of ``loop based at the endpoint
of~$\gamma$'', and then retracing
$\gamma$.

Our first normal subgroup $M$ of $F$ is defined as the subgroup
normally generated by all the $B_{i j}$, the $C_\gamma$ for all
$\gamma$ of length~$1$, and the $D_{\gamma,k}$ for all $\gamma$ of
length~$0$.  Unwinding the definitions shows that these elements of
$F$ are exactly the ones we used in defining $\G_{2,\Omega}$.  For
example, if $\gamma$ is the length~$1$ path from one vertex $j$ of
$\Omega$ to an adjacent vertex $i$ then $P_\gamma=S_j S_i$ and
$C_\gamma$ is the word \eqref{eq-S-i-X-j-relator-when-m-is-3}.  And if
$i\in\Omega$ is evenly joined to $j\in I$ then we take $\gamma$ to be
the zero-length path at $i$, and $D_{\gamma,j}$ turns out to be the
relator \eqref{eq-S-i-X-j-relator-when-m-is-2},
\eqref{eq-S-i-X-j-relator-when-m-is-4} or
\eqref{eq-S-i-X-j-relator-when-m-is-6}.  Which one of these applies depends on
$m_{i j}\in\{2,4,6\}$.  So $F/M\iso\G_{2,\Omega}$.

Before defining the other normal subgroup
$N$ we explain how to work with the $C$'s and $D$'s by thinking in
terms of paths rather than complicated words.

\begin{lemma}
\label{lem-G-2-description-compositions-of-paths}
Suppose $\gamma_1$ and $\gamma_2$ are paths in $\Omega$ with
$\w(\gamma_1)=\alpha(\gamma_2)$, and let $\gamma$ be the path which traverses $\gamma_1$
and then $\gamma_2$.
\begin{enumerate}
\item
\label{item-G-2-description-composition-of-C-and-C}
Any normal subgroup of $F$ containing two of $C_{\gamma_1}$,
$C_{\gamma_2}$ and $C_\gamma$ contains the third.
\item
\label{item-G-2-description-composition-of-C-and-D}
Suppose $k\in I$ is evenly joined to $\w(\gamma_2)$.  Then any normal subgroup
of $F$ containing $C_{\gamma_1}$ and one of $D_{\gamma_2,k}$ and
$D_{\gamma,k}$ contains the other as well.
\end{enumerate}
\end{lemma}

\begin{proof}
Both identities 
\begin{align*}
C_\gamma
&{}=
\bigl(P_{\gamma_2}C_{\gamma_1}P_{\gamma_2}^{-1}\bigr)C_{\gamma_2}
\\
D_{\gamma,k}
&{}=
P_{\gamma_1}^{-1}\Bigl(\bigl(R_{\gamma_2,k}C_{\gamma_1}R_{\gamma_2,k}^{-1}\bigr)D_{\gamma_2,k}C_{\gamma_1}^{-1}\Bigr)P_{\gamma_1}
\end{align*}
unravel to tautologies, using $P_\gamma=P_{\gamma_2}P_{\gamma_1}$.
These imply \ref{item-G-2-description-composition-of-C-and-C} and \ref{item-G-2-description-composition-of-C-and-D} respectively.
\end{proof}

To define $N$ we refer to the base vertex $\infty$ and spanning tree
$T$ that we introduced above.  For each $j\in\Omega$ we take $\delta_j$ to
be the backtracking-free path in $T$ from $\infty$ to $j$.  For each
edge of $\Omega$ not in $T$, choose an orientation of it, and define
$\E$ as the corresponding set of paths of length~$1$.  For
$\gamma\in\E$ we write $z(\gamma)$ for the corresponding loop in $\Omega$
based at $\infty$.  That is, $z(\gamma)$ is $\delta_{\alpha(\gamma)}$ followed by $\gamma$
followed by the reverse of $\delta_{\w(\gamma)}$.  We define $Z$ as
$\set{z(\gamma)}{\gamma\in\E}$, which is a free basis for
the fundamental group $\pi_1(\Omega,\infty)$.  We define $N$ as the subgroup of $F$
normally generated by all $B_{i\infty}$ with $i\in I$, all $C_{z\in
  Z}$, the $C_{\delta_j}$ with $j\in\Omega$, and all $D_{\delta_j,k}$ where
$j\in\Omega$ and $k\in I$ are evenly joined.  We will show $M=N$; one
direction is easy:

\begin{lemma}
\label{lem-M-contains-N}
$M$ contains $N$.
\end{lemma}

\begin{proof}
Since $M$ contains $C_\gamma$ for every length~$1$ path $\gamma$, repeated
applications of lemma~\ref{lem-G-2-description-compositions-of-paths}\ref{item-G-2-description-composition-of-C-and-C} show that it contains the
$C_{\delta_i}$ and $C_{z\in Z}$.  
Since $M$ contains $D_{\gamma,k}$ for every $\gamma$ of length~$0$, part \ref{item-G-2-description-composition-of-C-and-D} of the
same lemma
shows that $M$ also contains the $D_{\delta_j,k}$.  Since $M$ contains all
the $B_{i j}$, not just the
$B_{i\infty}$, the proof is complete.
\end{proof}

Now we set about proving the reverse inclusion.  For convenience we
use $\equiv$ to mean ``equal modulo $N$''.  We must show that each
generator of $M$ is${}\equiv1$.

\begin{lemma}
\label{lem-N-contains-length-1-directed-paths-in-T}
$C_\gamma\equiv1$ for every length~$1$ subpath $\gamma$ of every $\delta_j$.
\end{lemma}

\begin{proof}
This follows from
lemma~\ref{lem-G-2-description-compositions-of-paths}\ref{item-G-2-description-composition-of-C-and-C}
because $\delta_{\alpha(\gamma)}$ followed by $\gamma$ is $\delta_{\w(\gamma)}$.
\end{proof}

\begin{lemma}
\label{lem-N-contains-the-B-i-j}
$B_{i k}\equiv1$ for all $i\in I$ and $k\in\Omega$.
\end{lemma}

\begin{proof}
We claim: if $\gamma$ is a length~$1$ path in $\Omega$, such that 
$C_{\gamma}\equiv1$ and $B_{i\alpha(\gamma)}\equiv1$ for every $i\in I$, then also
$B_{i\w(\gamma)}\equiv1$ for every $i\in I$.  Assuming this, we use the
fact that $B_{i\infty}\equiv1$ for all $i\in I$ and also $C_\gamma\equiv1$ for every
length~$1$ subpath $\gamma$ of every $\delta_k$
(lemma~\ref{lem-N-contains-length-1-directed-paths-in-T}).  Since
every $k\in\Omega$ is the end of chain of such $\gamma$'s starting
at $\infty$, the lemma follows by induction.

So now we prove the claim, writing $i$ for some element of $I$ and $j$
and $k$ for the initial and final endpoints of $\gamma$.  We use
$C_\gamma\equiv1$, i.e., $S_j S_k X_j(t)\equiv X_k(t)S_j S_k$, to get
\begin{align}
S_i^2&X_k(t)S_i^{-2}
\notag\\
&{}\equiv
S_i^2S_j S_k X_j(t)S_k^{-1}S_j^{-1}S_i^{-2}
\notag
\\
\label{eq-foo-2}
&{}=
S_j S_k\Bigl[(S_k^{-1}S_j^{-1}S_i^2S_j S_k)X_j(t)(S_k^{-1}S_j^{-1}S_i^{-2}S_j S_k)\Bigr]S_k^{-1}S_j^{-1}
\end{align}
We rewrite the relation from lemma~\ref{lem-basic-properties-of-What}\ref{item-conjugates-of-squares-in-What} as $S_j^{-1}S_i^2
S_j=S_j^{(-1)^{A_{i j}}-1}S_i^2$. Then we use it and its analogues with
subscripts permuted 
to simplify the first parenthesized term in \eqref{eq-foo-2}.  We also
use $A_{j k}=-1$, which holds since $j$ and $k$ are joined.  
The result is
$$
S_k^{-1} S_j^{-1} S_i^2 S_j S_k
=
S_k^{1-(-1)^{A_{i j}}} S_j^{(-1)^{A_{i j}}-1} S_k^{-1+(-1)^{A_{i k}}} S_i^2
$$
Note that each exponent is $0$ or~$\pm2$.

The bracketed term in \eqref{eq-foo-2} is the conjugate of $X_j(t)$ by
this.  We work this out  in four steps, using  our assumed relations $B_{i j}\equiv
B_{j j}\equiv B_{k j}\equiv1$. 
Conjugation by $S_i^2$ changes
$X_j(t)$ to $X_j\bigl((-1)^{A_{i j}}t\bigr)$.  
Because $A_{k j}=-1$, 
conjugating $X_j\bigl((-1)^{A_{i j}}t\bigr)$  by
$S_k^{(-1)^{A_{i k}}-1}$ sends it to
\begin{center}
\begin{tabular}{ll}
itself&if $A_{i k}$ is even, because $(-1)^{A_{i k}}-1=0$\\
$X_j\bigl(-(-1)^{A_{i j}}t\bigr)$&if $A_{i k}$ is \phantom{even}\llap{odd}, because $(-1)^{A_{i k}}-1=-2$.
\end{tabular}
\end{center}
We write this as
$X_j\bigl((-1)^{A_{i k}}(-1)^{A_{i j}}t\bigr)$.
In the third step we conjugate by an even power of
$S_j$, which does nothing.  
The fourth step is like the second, and introduces a second factor $(-1)^{A_{i j}}$.
The net result is that the bracketed term of \eqref{eq-foo-2}
equals $X_j\bigl((-1)^{A_{i k}}t\bigr)$ modulo~$N$.

Plugging this into \eqref{eq-foo-2} and then using the conjugacy relation $C_\gamma\equiv1$ between
$X_j$ and $X_k$ yields
$$
S_i^2 X_k(t) S_i^{-2}
\equiv
S_j S_k
X_j\bigl((-1)^{A_{i k}}t\bigr)
S_k^{-1}S_j^{-1}
\equiv
X_k\bigl((-1)^{A_{i k}}t\bigr).
$$
We have established the desired relation $B_{i k}\equiv1$.
\end{proof}

\begin{lemma}
\label{lem-N-contains-reversals-of-length-1-paths}
Suppose $\gamma$ is a length~$1$ path in $\Omega$ with
$C_\gamma\equiv1$.  Then  $C_{\reverse(\gamma)}\equiv1$ also.
\end{lemma}

\begin{proof}
Suppose $\gamma$ goes from $j$ to $k$.
We begin with our assumed relation $C_\gamma\equiv1$, i.e., $S_j S_k X_j(t)\equiv
X_k(t)S_j S_k$, rearrange and apply 
the 
relation
from lemma~\ref{lem-basic-properties-of-What}\ref{item-conjugates-of-squares-in-What} with $A_{j k}={}$odd.
\begin{align*}
X_k(t)
&{}\equiv
S_j S_k X_j(t)S_k^{-1}S_j^{-1}
\\
S_k S_j X_k(t)
&{}\equiv
\bigl(S_k S_j^2S_k^{-1}\bigr)S_k^2X_j(t)S_k^{-1}S_j^{-1}
\\
&{}=
\bigl(S_k^2S_j^2)S_k^2 X_j(t)S_k^{-1}S_j^{-1}.
\\
\noalign{\noindent
Then we use lemma~\ref{lem-N-contains-the-B-i-j}'s $B_{j j}\equiv B_{k j}\equiv1$ 
 with $A_{k j}={}$odd:}
&{}\equiv
S_k^2 S_j^2 X_j(-t) S_k^2 S_k^{-1}S_j^{-1}\\
&{}\equiv
S_k^2 X_j(-t)  S_j^2 \cdot S_k^2 S_k^{-1}S_j^{-1}\\
&{}\equiv
X_j(t) S_k^2 S_j^2\cdot S_k^2 S_k^{-1}S_j^{-1}\\
&{}=
X_j(t) S_k S_j^2 S_k^{-1}\cdot S_k^2 S_k^{-1}S_j^{-1}\\
&{}=
X_j(t) S_k S_j
\end{align*}
We have shown $C_{\reverse(\gamma)}\equiv1$, as desired.
\end{proof}

\begin{lemma}
\label{lem-N-equals-M}
$M=N$.  In particular, $\G_{2,\Omega}$ is the quotient of
$F=\bigl(\freeproduct_{j\in\Omega}\,\U_j\bigr)\freeproduct\What$ by $N$.
\end{lemma}

\begin{proof}
We showed $N\sset M$ in lemma~\ref{lem-M-contains-N}.  To show the
reverse inclusion, recall that $M$ is normally generated by 
all $B_{i j}$, the $C_\gamma$ for all $\gamma$ of length~$1$, and the
$D_{\gamma,k}$ for all $\gamma$ of length~$0$.  
We must show that each of these is${}\equiv1$.  
We showed $B_{i j}\equiv1$ in 
lemma~\ref{lem-N-contains-the-B-i-j}.

Next we show that $C_\gamma\equiv1$ for every length~$1$ path $\gamma$ in $T$.
If $\gamma$ is part of one of the paths $\delta_j$ in $T$ based at
$\infty$, then $C_\gamma\equiv1$ by lemma~\ref{lem-N-contains-length-1-directed-paths-in-T}, and then 
$C_{\reverse(\gamma)}\equiv1$ by
lemma~\ref{lem-N-contains-reversals-of-length-1-paths}.

Lemma~\ref{lem-G-2-description-compositions-of-paths}\ref{item-G-2-description-composition-of-C-and-C}
now shows $C_\gamma\equiv1$ for every path $\gamma$ in $T$.

Next we show $C_\gamma\equiv1$ for every
length~$1$ path $\gamma$ not in $T$.
Recall that we chose a set $\E$ of length~$1$ paths, one traversing
each edge of $\Omega$ not in $T$.  For $\gamma\in\E$ we wrote $z(\gamma)$ for
the corresponding loop in $\Omega$ based at $\infty$, 
namely $\delta_{\alpha(\gamma)}$ followed by $\gamma$ followed by
$\reverse(\delta_{\w(\gamma)})$.  Recall that  $N$ contains
$C_{z(\gamma)}$ by definition, and contains $C_{\delta_{\alpha(\gamma)}}$ and
$C_{\reverse(\delta_{\w(\gamma)})}$ by the
previous paragraph. So a double application of
lemma~\ref{lem-G-2-description-compositions-of-paths}\ref{item-G-2-description-composition-of-C-and-C}
proves $C_\gamma\in N$.  And another use of
lemma~\ref{lem-N-contains-reversals-of-length-1-paths} shows that $N$
also contains $C_{\reverse(\gamma)}$.  
This finishes the proof that $C_\gamma\equiv1$ for all length~$1$
paths $\gamma$ in~$\Omega$

It remains only to show $D_{\gamma,k}\equiv1$ for every length~$0$ path
$\gamma$ in $\Omega$ and each $k\in I$ joined evenly to the unique point
of $\gamma$, say $j$.  Since $N$ contains $C_{\delta_j}$ and $D_{\delta_j,k}$ by
definition, and $\delta_j$ followed by $\gamma$ is trivially equal to $\delta_j$,
lemma~\ref{lem-G-2-description-compositions-of-paths}\ref{item-G-2-description-composition-of-C-and-D}
shows that $N$ contains $D_{\gamma,k}$ also.
\end{proof}

We now review the general form of the description $F/N$ of
$\G_{2,\Omega}$ that we have just established.  
The
generators are the $S_{i\in I}$ and the $X_{j\in\Omega}(t)$ with $t\in R$.  The
relations are the addition rules defining the $\U_j$, the
relations on the $S_i$'s defining $\What$, and the $B_{i\infty}$, 
$C_{z\in Z}$,  $C_{\delta_j}$ and $D_{\delta_j,k}$ where $i$ varies over
$I$, $j$ over $\Omega$, and $k\in I$
is evenly joined to $j$.  The relations $B_{i\infty}\equiv1$
say that $S_i^2$
centralizes or inverts every $X_\infty(t)$.  Each relation $C_{z}\equiv1$ says
that a certain word in $\What$ conjugates every $X_\infty(t)$ to
itself.  The relations $D_{\delta_j,k}\equiv1$ say that certain other words
in $\What$ also commute with every $X_\infty(t)$.  Finally, for each
$j$, the relations $C_{\delta_j}\equiv1$ express the $X_j(t)$ as
conjugates of the $X_\infty(t)$ by still more words in $\What$.
The obvious way to simplify the presentation is to use this last batch of
relations to eliminate the $X_{j\neq\infty}(t)$ from the
presentation.  We make this precise in the following lemma.

\begin{lemma}
\label{lem-structure-of-F-infty-mod-N-infty}
Define $F_\infty=\U_\infty\freeproduct\What$ and let $N_\infty$ be the
subgroup normally generated by the $B_{i\infty}$ ($i\in I$), the
$C_z$ ($z\in Z$), and the $D_{\delta_j,k}$ ($j\in\Omega$ and
$k\in I$ evenly joined).  Then the natural map $F_\infty/N_\infty\to
F/N$ is an isomorphism.
\end{lemma}

\begin{proof}
We begin with the presentation $F/N$ from the previous paragraph and
apply Tietze transformations.  The relation $C_{\delta_j}(t)\equiv1$
reads:
$$
X_j(t)\equiv P_{\delta_j}\,X_\infty(t)\,P_{\delta_j}^{-1}
$$
For $j=\infty$ this is the trivial relation $X_\infty(t)=X_\infty(t)$,
which we may discard.  For $j\neq\infty$ we use it to
replace $X_j(t)$ by $P_{\delta_j}\,X_\infty(t)\,P_{\delta_j}^{-1}$ everywhere else in the
presentation, and then discard $X_j(t)$ from the generators and
$C_{\delta_j}(t)$ from the relators. 

The only other occurrences of $X_{j\neq\infty}(t)$ in the presentation
are in the relators defining $\U_j$.  After the replacement of the
previous paragraph, these relations read
$$
P_{\delta_j}X_\infty(t)P_{\delta_j}^{-1}\cdot P_{\delta_j}
X_\infty(u)P_{\delta_j}^{-1}
\equiv
P_{\delta_j}X_\infty(t+u)P_{\delta_j}^{-1}.
$$
These relations can be discarded because they are the $P_{\delta_j}$-conjugates  
of the relations 
$X_\infty(t)
X_\infty(u)
\equiv
X_\infty(t+u)$.  What remains is the presentation $F_\infty/N_\infty$.
\end{proof}

\begin{proof}[Proof of lemma~\ref{lem-free-product-model-for-one-component-of-odd-diagram}]
The previous 
lemma shows $\G_{2,\Omega}\iso F_\infty/N_\infty$.  
So $\G_{2,\Omega}$ is the quotient of
$\U_\infty*\What$ by relations asserting that certain elements of
$\What$ act on $\U_\infty$ by certain automorphisms.  The relations
$B_{i\infty}=1$ make $S_i^2$ act on $\U_\infty$ by
$(-1)^{A_{i\infty}}$.  The relations $C_z=D_{\delta_j,k}=1$
make the words $P_z$ and $R_{\delta_j,k}$ centralize $\U_\infty$.  

By
lemma~\ref{lem-basic-properties-of-What}\ref{item-squares-generate-kernel-of-What-to-W},
the $S_i^2$ generate the kernel of $\What\to W$.  By
theorem~\ref{thm-root-stabilizers-in-W-star}, the images of the $P_z$
and $R_{\delta_j,k}$ in $W$ generate the $W$-stabilizer of the simple
root $\infty\in I$.  Therefore the $S_i^2$, $P_z$ and $R_{\delta_j,k}$
generate the $\What$-stabilizer $\What_\infty$ of $\infty$.  Their
actions on $\U_\infty$ are the same as the ones given by the
homomorphism $\What\to\Wstar$, by
theorem~\ref{thm-root-stabilizers-in-W-star}. Therefore
$\G_{2,\Omega}=\bigl(\U_\infty\semidirect\What_\infty)\freeproduct_{\What_\infty}\What$.
And lemma~\ref{lem-general-finite-presentation-principle} identifies
this with
$\bigl(\freeproduct_{\alpha\in\Phi(\Omega)}\U_{\alpha}\bigr)\semidirect\What$,
as desired.
\end{proof}

\section{The isomorphism $\G_3\iso\PStTits\semidirect\What$}
\label{sec-step-3}

\noindent
We have two goals in this section.  The first is to 
start from
theorem~\ref{thm-G-2-is-free-product-semidirect-What},
that
$\G_2\iso\bigl(\freeproduct_{\alpha\in\Phi}\,\U_\alpha\bigr)\semidirect\What$,
and 
prove
theorem~\ref{thm-G3-is-spherical-Tits-Steinberg-group-semidirect-What},
that $\G_3\iso\PStTits\semidirect\What$.
The second is to explain how one may discard many of the Chevalley
relations, for example for $E_{n\geq6}$ one can get away with imposing
the relations for a single unjoined pair of nodes of the Dynkin
diagram, and for a single joined pair.  The latter material is not
necessary for our main results.

\begin{proof}[Proof of theorem~\ref{thm-G3-is-spherical-Tits-Steinberg-group-semidirect-What}]
First we show that the relators \eqref{eq-Chevalley-relator-m=2}--\eqref{eq-Chevalley-relators-m=6-distant-short-and-long}, regarded as
elements of
$\G_2\iso\bigl(\freeproduct_{\alpha\in\Phi}\,\U_\alpha\bigr)\semidirect\What$,
become trivial in $\StTits\semidirect\What$.  Then we will show that
they normally generate the whole kernel of
$\G_2\to\StTits\semidirect\What$.

If $\alpha,\beta$ are a prenilpotent pair of roots with $\theta(\alpha,\beta)=\{\alpha,\beta\}$,
then the Chevalley relation for $\alpha$ and $\beta$ is
$[\U_\alpha,\U_\beta]=1$.  This shows that relators \eqref{eq-Chevalley-relator-m=2},
\eqref{eq-Chevalley-relators-m=3-close-roots}, \eqref{eq-Chevalley-relators-m=4-adjacent-short-and-long}, \eqref{eq-Chevalley-relators-m=4-orthogonal-long}, \eqref{eq-Chevalley-relators-m=6-nearby-long}, \eqref{eq-Chevalley-relators-m=6-adjacent-short-and-long} and \eqref{eq-Chevalley-relators-m=6-orthogonal-short-and-long} become
trivial in $\StTits\semidirect\What$.  Careful calculation verifies
that the remaining relators are equivalent to those given by Demazure
in \cite[ch.~XXIII]{Demazure-Grothendieck}.  Here are some remarks on the
correspondence between his notations and ours.  In the $A_2$ case (his
proposition 3.2.1), his $\alpha$ and $\beta$ correspond to our
$\alpha_j$ and $\alpha_i$, his $X_\alpha$ and $X_\beta$ to our $e_j$
and $e_i$, his $X_{-\alpha}$ and $X_{-\beta}$ to our $-f_j$ and
$-f_i$, and his $p_\alpha(t)$ and $p_\beta(t)$ to our $X_j(t)$ and
$X_i(t)$.  His $w_\alpha$ and $w_\beta$ are not the same as our $S_j$
and $S_i$ (which are not even elements of
$\freeproduct_{\gamma\in\Phi}\,\U_\gamma$), but their actions on 
the $\U_\gamma$'s
are the same, so his
$p_{\alpha+\beta}(t):=w_\beta\,p_\alpha(t)\,w_\beta^{-1}$ corresponds
to our $S_i X_j(t) S_i^{-1}$.  One can
now check that our \eqref{eq-Chevalley-relators-m=3-distant-roots} is equivalent to  his
3.2.1(iii).

In the $B_2$ case (his proposition 3.3.1), his $\alpha$ and $\beta$
correspond to our $\alpha_s$ and $\alpha_l$, his $X_\alpha$ and
$X_\beta$ to our $e_s$ and $e_l$, his $X_{-\alpha}$ and $X_{-\beta}$
to our $-f_s$ and $-f_l$, and his $p_\alpha(t)$ and $p_\beta(t)$ to
our $X_s(t)$ and $X_l(t)$.  His $w_\alpha$ and $w_\beta$ correspond to
our $S_s$ and $S_l$ in the same sense as above.  It follows that his
$p_{\alpha+\beta}(t)$ and $p_{2\alpha+\beta}(t)$ correspond to our
$S_l X_s(t) S_l^{-1}$ and $S_s X_l(t) S_s^{-1}$.  Then our \eqref{eq-Chevalley-relators-m=4-orthogonal-short}
and \eqref{eq-Chevalley-relators-m=4-distant-short-and-long} are equivalent to his 3.3.1.  The $G_2$ case  is the same
(his
proposition 3.4.1),
except that his $p_{\alpha+\beta}(t)$, $p_{2\alpha+\beta}(t)$,
$p_{3\alpha+\beta}(t)$ and $p_{3\alpha+2\beta}(t)$ correspond to our
$S_l X_s(t) S_l^{-1}$, $S_s S_l X_s(t) S_l^{-1} S_s^{-1}$, $S_s
X_l(-t) S_s^{-1}$ and $S_l S_s X_l(-t) S_s^{-1} S_l^{-1}$.  Then our
\eqref{eq-Chevalley-relators-m=6-distant-long}--\eqref{eq-Chevalley-relators-m=6-distant-short-and-long}
are among the relations in his 3.4.1(iii).  
%(He also gives the commutator of $p_{\alpha+\beta}(t)$ and
%$p_{2\alpha+\beta}(u)$.)

As a check (indeed a second proof that our relations are the Chevalley
relations) we constructed our elements of the various root groups in
explicit representations of the Chevalley groups $\SL_2\times\SL_2$,
$\SL_3$, $\Sp_4$ and $G_2$ over $R=\Z[t,u]$, faithful on the unipotent
subgroups of their Borel subgroups.  As mentioned in
remark~\ref{rem-precautions}, we used a computer to check
that our relators map to the identity.  By functoriality, the same
holds with $R$ replaced by any ring.  In addition to our relations,
the root groups satisfy the Chevalley relations, by construction.  By
the isomorphism $
\U_{\theta(\alpha,\beta)}\iso\prod_{\gamma\in\theta(\alpha,\beta)}\U_\gamma
$ of underlying schemes
(lemma~\ref{lem-existence-of-unipotent-groups}), the only relations
having the form of the Chevalley relations that can hold are the
Chevalley relations themselves.  So our relations are among them.

It remains to prove that the Chevalley relators of any classically
prenilpotent pair $\alpha',\beta'\in\Phi$ become trivial in $\G_3$.
By classical prenilpotency, $\Phi_0':=(\Q \alpha'+\Q
\beta')\cap\Phi$ is an $A_1$, $A_1^2$, $A_2$, $B_2$ or $G_2$ root
system.
In the $A_1$ case we have $\alpha'=\beta'$ and the Chevalley relations
amount to the commutativity of $\U_{\alpha'}$.  This follows from
$\U_{\alpha'}\iso R$.  So we consider the other cases.
There exists $w\in W$ sending $\Phi_0'$ to the root system
$\Phi_0\sset\Phi$ generated by some pair of simple roots.  (Choose
simple roots for $\Phi_0'$. Then choose a chamber in the  Tits
cone which has two of its facets lying in the mirrors of those roots,
and which lies on the positive sides of these mirrors.  Choose $w$ to
send this chamber to the standard one.)

We choose a pair of roots $\alpha,\beta\in\Phi_0$ as follows.  First,
they should have the same relative configuration as $\alpha',\beta'$
have. (That is, they should have the same short/long root status, and make the
same angle.)  And second, their Chevalley relators should appear among
\eqref{eq-Chevalley-relator-m=2}--\eqref{eq-Chevalley-relators-m=6-distant-short-and-long}.
Such $\alpha,\beta$ can always be chosen.  For example, in the $G_2$
case,
\eqref{eq-Chevalley-relators-m=6-nearby-long}--\eqref{eq-Chevalley-relators-m=6-distant-short-and-long}
are respectively the Chevalley relations for two long roots with angle
$\pi/3$, a short and a long root with angle $\pi/6$, two orthogonal
roots, two long roots with angle $2\pi/3$, two short roots with angle
$\pi/3$, two short roots with angle $2\pi/3$, and a short and a long
root with angle $5\pi/6$.  The other cases are similarly exhaustive.
By refining the choice of $w$, we may suppose that it sends
$\{\alpha',\beta'\}$ to $\{\alpha,\beta\}$.  Now choose
$\what\in\What$ lying over $w$.  The Chevalley relators for
$\alpha',\beta'$ are the $\what^{-1}$-conjugates of the Chevalley
relators for $\alpha,\beta$.  Since the latter become trivial in
$\G_3$, so do the former.
\end{proof}

The proof of
theorem~\ref{thm-G3-is-spherical-Tits-Steinberg-group-semidirect-What}
exploited the $\What$-action on
$\freeproduct_{\alpha\in\Phi}\U_{\alpha}$ to obtain the
Chevalley relators for all classically prenilpotent pairs from those listed explicitly in
\eqref{eq-Chevalley-relator-m=2}--\eqref{eq-Chevalley-relators-m=6-distant-short-and-long}.
One can further exploit this idea to omit many of the relators coming
from the cases $m_{i j}=2$ or~$3$.  Our method derives from the notion
of an ordered pair of simple roots being associate to another pair,
due to Brink-Howlett \cite{BH-normalizers} and Borcherds
\cite{Borcherds-normalizers}.  But we need very little of their
machinery, so we will argue directly.  There does not seem to be any
similar simplification possible if $m_{i j}=4$ or~$6$.

\begin{proposition}
\label{prop-omitting-m-i-j=2-relations}
Suppose $i,j,k\in I$ form an $A_1A_2$ diagram, with $j$ and $k$
joined.  Then imposing the relation $[\U_i,\U_j]=1$ on 
$\G_2\iso\bigl(\freeproduct_{\alpha\in\Phi}\U_\alpha\bigr)\semidirect\What$
also
  imposes $[\U_i,\U_k]=1$.  More formally, the normal closure of the
    relators \eqref{eq-Chevalley-relator-m=2} in
$\G_2$ 
    contains the relators got from them by replacing $j$ by $k$.
\end{proposition}

\begin{proof}
Some element of the copy of $W(A_2)$ generated by $s_j$ and $s_k$
sends $\alpha_j$ to $\alpha_k$, and of course it fixes $\alpha_i$.
Choose any lift of it to $\What$.  Conjugation by it in $\G_2$ fixes
$\U_i$ and sends $\U_j$ to $\U_k$.  So it sends the relators \eqref{eq-Chevalley-relator-m=2}
to the relators got from them by replacing $j$ by~$k$.
\end{proof}

The lemma shows that imposing on $\G_2$ the relations
\eqref{eq-Chevalley-relator-m=2} for a few well-chosen unordered pairs
$\{i,j\}$ in $I$ with $m_{i j}=2$ automatically imposes the
corresponding relations for all such pairs.  As examples, for
spherical Dynkin diagrams it suffices to impose these relations for
$$
\begin{aligned}
&\hbox{$3$ such pairs (that is, all of them) for $D_4$;}\\
&\hbox{$2$ such pairs for $B_{n\geq4}$, $C_{n\geq4}$ or $D_{n\geq5}$;}\\
&\hbox{$1$ such pair for $A_{n\geq3}$, $B_3$, $C_3$, $E_n$ or $F_4$.}
\end{aligned}
$$

\begin{proposition}
\label{prop-omitting-m-i-j=3-relations}
Suppose $i,j,k\in I$ form an $A_3$ diagram, with $i$ and~$k$
unjoined.  Then the normal closure of the relators \eqref{eq-Chevalley-relators-m=3-close-roots}--\eqref{eq-Chevalley-relators-m=3-distant-roots}
in
$\G_2\iso\bigl(\freeproduct_{\alpha\in\Phi}\U_\alpha\bigr)\semidirect\What$
contains the relators got from them by replacing $i$ and~$j$ by $j$
and~$k$ respectively.
\end{proposition}

\begin{proof}
The argument is the same as for proposition~\ref{prop-omitting-m-i-j=2-relations}, using an element of
$W(A_3)$ that sends $\alpha_i$ and $\alpha_j$ to $\alpha_j$ and
$\alpha_k$.  An example of such an element is the ``fundamental
element'' (or ``long word'') of
$\gend{s_i,s_j}$, followed by the fundamental element of
$\gend{s_i,s_j,s_k}$.  The first transformation sends $\alpha_i$ and
$\alpha_j$ to $-\alpha_j$ and $-\alpha_i$.  The second sends
$\alpha_i$, $\alpha_j$ and $\alpha_k$ to $-\alpha_k$, $-\alpha_j$ and
$-\alpha_i$. 
\end{proof}

Similarly to the $m_{i j}=2$ case,
imposing on $\G_2$ the
relations \eqref{eq-Chevalley-relators-m=3-close-roots}--\eqref{eq-Chevalley-relators-m=3-distant-roots} for some well-chosen ordered pairs $(i,j)$ in $I$
with $m_{i j}=3$ automatically imposes the corresponding relations for
all such pairs.  For spherical diagrams, it suffices to impose these relations for
$$
\begin{aligned}
&\hbox{$4$ such pairs (that is, all of them) for $F_4$;}\\
&\hbox{$2$ such pairs for $A_{n\geq2}$, $B_{n\geq3}$ or $C_{n\geq3}$;}\\
&\hbox{$1$ such pair for $D_{n\geq4}$ or $E_n$.}
\end{aligned}
$$

\section{The adjoint representation}
\label{sec-adjoint-representation}

\noindent
A priori, it is conceivable that for some commutative ring $R\neq0$
and some generalized Cartan matrix $A$, the Steinberg group $\St_A(R)$
might collapse to the trivial group.  That this doesn't happen follows
from work of Tits \cite[\S4]{Tits} and R\'emy \cite[Ch.~9]{Remy} on
the ``adjoint representation'' of $\St_A$.  We will improve their
results slightly by proving that the unipotent group scheme $\U_\Psi$ embeds in the Steinberg
group functor $\St_A$, for any nilpotent set of roots $\Psi$.  We need
this result in the next section, in our proof that
$\PSt_A(R)\to\St_A(R)$ is often an isomorphism.

Recall that lemma~\ref{lem-existence-of-unipotent-groups} associates to $\Psi$ a
unipotent group scheme $\U_\Psi$ over $\Z$.  Furthermore, there are natural
homomorphisms $\U_{\gamma}\to\U_\Psi$ for all $\gamma\in\Psi$, and the
product map $\prod_{\gamma\in\Psi}\U_\gamma\to\U_\Psi$ is an
isomorphism of the underlying schemes, for any ordering of the factors.

Also in section~\ref{sec-Steinberg-group}, we defined Tits' Steinberg functor $\StTits_A$ as the direct limit
of the group schemes $\U_\gamma$ and $\U_\Psi$, where $\gamma$ varies
over $\Phi$, and $\Psi$ varies over the nilpotent subsets of $\Phi$ of
the form $\Psi=\theta(\alpha,\beta)$, with $\alpha,\beta$ a
prenilpotent pair of roots.  
Composing with $\StTits_A\to\St_A$, we have natural maps
$\U_\Psi\to\St_A$ for such $\Psi$.
A special case of the following theorem
is that these maps are embeddings.
We would like to say that the same holds for
$\Psi$ an arbitrary nilpotent set of roots.  But ``the same holds''
doesn't quite have meaning, because the definition of $\St_A$ doesn't provide
a natural map $\U_\Psi\to\St_A$ for general~$\Psi$.  So we phrase the result as follows.

\begin{theorem}[Injection of unipotent subgroups into $\St_A$]
\label{thm-injection-of-unipotent-subgroups}
Suppose $A$ is a generalized Cartan matrix and $\Psi$ is a nilpotent
set of roots.  Then there is a unique homomorphism 
$\U_\Psi\to\St_A$ whose restriction to each $\U_{\alpha\in\Psi}$ is
the natural map to $\St_A$, and it is an embedding.
\end{theorem}

Uniqueness is trivial, by the isomorphism of
underlying schemes $\U_\Psi\iso\prod_{\alpha\in\Psi}\U_\alpha$.
Existence is easy: every pair of roots in $\Psi$ is prenilpotent,
their Chevalley relations hold in $\St$, and these relations suffice
to define $\U_\Psi$ as a quotient of
$\freeproduct_{\beta\in\Psi}\,\U_\beta$.  So we must show that
that this homomorphism is an embedding.  Our proof below relies on a linear representation
of $\St_A$, functorial in $R$, called the adjoint representation.  Its
essential properties are developed in \cite[Ch.~9]{Remy}, relying on a
$\Z$-form of the universal enveloping algebra of $\g$ introduced by
Tits \cite[\S4]{Tits}.

\medskip
Following Tits and R\'emy we will indicate all ground rings other than
$\Z$ explicitly, in particular writing $\g_\C$ for the Kac--Moody
algebra~$\g$.  We write $\curlyU_\C$ for its universal enveloping
algebra.  Recall from section~\ref{sec-Steinberg-group} that for each root
$\alpha\in\Phi$ we distinguished a subgroup $\g_{\alpha,\Z}\iso\Z$ of
$\g_{\alpha,\C}$ and the set $E_\alpha$ consisting of the two
generators for $\g_{\alpha,\Z}$.

Generalizing work of Kostant \cite{Kostant} and Garland \cite{Garland}, Tits defined
an integral form of $\curlyU_\C$, meaning a subring $\curlyU$ with the
property that the natural map $\curlyU\tensor\C\to\curlyU_\C$ is an
isomorphism.  It is the subring generated by the divided powers
$e_i^{n}/n!$ and $f_i^n/n!$, as $i$ varies over $I$, together with the
``binomial coefficients''
$\binom{h}{n}:=h(h-1)\cdots(h-n+1)/n!$ where $h$ varies over the
$\Z$-submodule of $\g_{0,\C}$ with basis $\barh_i$.

\begin{remark}[The role of the root datum]
Although it isn't strictly  necessary, we mention that lurking
behind the scenes is a choice of root datum.  It is the one which
R\'emy calls simply connected \cite[\S7.1.2]{Remy} and Tits calls ``simply
connected in the strong sense'' \cite[remark\ 3.7(c)]{Tits}.  A choice
of root datum is necessary to define $\curlyU$, hence the adjoint
representation, and the choice does matter.  For example, $\SL_2$ and
$\PGL_2$ have the same Cartan matrix, but different root data.  Their
adjoint representations are distinct in characteristic~$2$, when we
compare them by regarding both as representations of $\SL_2$ via the
central isogeny $\SL_2\to\PGL_2$.  Similarly, they provide distinct
representations of $\St_{A_1}$.
%In this case, the
%lattice they call $\Lambda^\vee$ is the subgroup of $\g_{0,\C}$
%generated by the $\barh_i$.  
For us the essential fact is that each  $\barh_i$
generates a $\Z$-module summand of $\curlyU$, as explained in the next
paragraph.  As an example of what could go wrong, using the root
datum for $\PGL_2$ would lead to $\barh_i/2\in\curlyU$ and spoil the
proof of theorem~\ref{thm-injection-of-unipotent-subgroups} in characteristic~$2$.
\end{remark}

In the sense Tits used, an integral form of a $\C$-algebra need not be
free as a $\Z$-module.  For example, $\Q$ is a $\Z$-form of $\C$ since
$\Q\tensor_\Z\C\to\C$ is an isomorphism.  But $\curlyU$ is free as a $\Z$-module.
To see this, one uses the following ingredients from
\cite[sec.~4.4]{Tits}. First, the $\Z^I$-grading makes it easy to see that
\begin{align*}
\curlyU_+&{}:=\biggend{\set{e_i^n/n!}{\hbox{$i\in I$ and
      $n\geq0$}}}\\
\llap{and\quad}
\curlyU_-&{}:=\biggend{\set{f_i^n/n!}{\hbox{$i\in I$ and
      $n\geq0$}}}
\end{align*} 
are free as $\Z$-modules, and that
$\{e_{i\in I}\}$ and $\{f_{i\in I}\}$ extend to bases of them.
Second, the universal enveloping algebra
$\curlyU_{0,\C}$ of the Cartan algebra $\g_{0,\C}$ 
is a polynomial ring.  This makes it easy to see
that
$$
\textstyle
\curlyU_0:=\biggend{\set{\binom{h}{n}}{\hbox{$h\in\oplus_{i}\,\Z\barh_i$
      and $n\geq0$}}}
$$ is free as a $\Z$-module.  Indeed, Prop.~2 of
\cite[VIII.12.4]{Bourbaki-7-and-8} extends $\{\barh_{i\in I}\}$ to a
$\Z$-basis for $\curlyU_0$.  Finally,
$\curlyU_-\tensor\curlyU_0\tensor\curlyU_+\to\curlyU$ is an
isomorphism by \cite[Prop.~2]{Tits}.  One can obtain a
$\Z$-basis for $\curlyU$ by tensoring together members of bases for
$\curlyU_-$, $\curlyU_0$ and $\curlyU_+$.

A key property of $\curlyU$ is  its stability under $(\ad
e_i)^n/n!$ and $(\ad f_i)^n/n!$ for all~$n\geq0$ (see
\cite[eqn.\ (12)]{Tits}).  The local nilpotence of $\ad e_i$ and $\ad
f_i$ on $\g_\C$ implies their local nilpotence on $\curlyU_\C$. As
exponentials of locally nilpotent derivations, $\exp\ad e_i$ and
$\exp\ad f_i$ are automorphisms of $\curlyU_\C$.  Since they preserve
its subring $\curlyU$,
they are automorphisms of it.   
Since the generators $\sstar_i$ for
$W^*$ are defined in terms of them by \eqref{eq-definition-of-s-star}, $\Wstar$  also
acts on $\curlyU$.  
%It follows that $\curlyU$ contains the
%divided powers $e^n/n!$, where $e\in E_\alpha$  for any root $\alpha$.

Because $\curlyU$ is free as a $\Z$-module, $\curlyU_R:=\curlyU\tensor
R$ is free as an $R$-mod\-ule.  It is the $R$-module underlying the
adjoint representation of $\St_A(R)$ in theorem~\ref{thm-adjoint-representation} below, which we
will now develop.  For each root $\alpha$ we define an exponential map
$\exp:\U_\alpha(R)\to\Aut(\curlyU_R)$ as follows.  Recall that
$\U_\alpha(R)$ was defined as $\g_{\alpha,\Z}\tensor R$.  If $x$ is an
element of this, then we choose  $e\in E_\alpha$
and define $t\in R$ by $x=t e$.  Then we define $\exp(x)$ to be the
$R$-module endomorphism of $\curlyU_R$ given by $\sum_{n=0}^\infty
t^n(\ad e)^n/n!$.
The apparent dependence on the choice of $e$ is no
dependence at all, because if one makes the other choice $-e$ then
one must also replace $t$ by $-t$.
As shown in \cite[\S9.4]{Remy}, $\exp(x)$ is an $R$-algebra automorphism of
$\curlyU_R$, not merely an $R$-module endomorphism.  
%(It also preserves the
%filtration coming from the one $\curlyU_\C$ has as the universal
%enveloping algebra of $\g_\C$, but we don't need this.)

\begin{theorem}[Adjoint representation]
\label{thm-adjoint-representation}
For any commutative ring~$R$, there exists a homomorphism
$\Ad:\St_A(R)\to\Aut \curlyU_R$, functorial in~$R$ and 
characterized by the following
property.  For every root $\alpha$
the exponential map $\exp:\U_\alpha(R)\to\Aut\curlyU_R$ factors as the
natural map $\U_\alpha(R)\to\St_A(R)$ followed by $\Ad$.
\end{theorem}

\begin{proof}
This is from 9.5.2--9.5.3 of R\'emy \cite{Remy}.  We remark that he
used Tits' version of the Steinberg functor (what we call $\StTits_A$)
rather than the Morita--Rehmann version (what we call $\St_A$).  But
his theorem 9.5.2 states that $\Ad$ is a representation of Tits'
Kac--Moody group $\tilde{\mathcal{G}}_D(R)$.  Since the extra
relations in the Morita--Rehmann version of the Steinberg group are
among those defining $\tilde{\mathcal{G}}_D(R)$, we may regard $\Ad$
as a representation of $\St_A(R)$.  

A few comments are required to identify our relations with (some of)
his.  $\tilde{\mathcal{G}}_D(R)$ is defined in \cite[8.3.3]{Remy} as a
quotient of the free product of $\PStTits_A(R)$ with a certain torus
$\T$.  R\'emy's third relation identifies our $\htilde_i(r)$ from
\eqref{eq-defn-of-h-tilde-e} with the element of $\T$ that R\'emy calls $r^{h_i}$.
R\'emy's first relation says how $\T$ acts on each $\U_j$, and amounts
to our \eqref{eq-alternative-Steinberg-relations-h-action-on-simple-root-groups}.  R\'emy's fourth relation is our \eqref{eq-alternative-Steinberg-relations-s-action-on-everything}, saying that
each $\stilde_i$ acts as $\sstar_i$ on every $\U_\beta$.  R\'emy's
second relation says how each $\stilde_i$ acts on $\T$, and in
particular describes $\stilde_j\,r^{h_i}\,\stilde_j^{-1}$.  Together
with the known action of $\htilde_i(r)$ on $\U_j$ and fact that
$\stilde_j$ exchanges $\U_{\pm j}$, this describes how
$\htilde_i(r)$ acts on $\U_{-j}$, and recovers our
relation \eqref{eq-alternative-Steinberg-relations-h-action-on-negative-simple-root-groups}.  By theorem~\ref{thm-alternate-relations-for-Steinberg-group}, this shows that all the
relations in our $\St(R)$ hold in $\tilde{\mathcal{G}}_D(R)$.
\end{proof}

\begin{proof}[Proof of theorem~\ref{thm-injection-of-unipotent-subgroups}]
By induction on $|\Psi|$.  The base case, with $\Psi=\emptyset$, is
trivial.  So suppose $|\Psi|>0$.  Since $\Psi$ is nilpotent, there is
some chamber pairing positively with every member of $\Psi$ and
another one pairing negatively with every member.  It follows that
there is a chamber pairing positively with one member and negatively
with all the others.  In other words, after applying an element of
$\Wstar$ we may suppose that $\Psi$ contains exactly one positive
root.  We may even suppose that this root is simple, say~$\alpha_i$.
Write $\Psi_0$ for $\Psi-\{\alpha_i\}$.

Consider the adjoint representation $\U_\Psi(R)\to\St(R)\to\Aut\curlyU_R$,
in particular the action of $x\in\U_\Psi(R)$ on $f_i\in\curlyU_R$.  If
$x\in\U_{\Psi_0}(R)$ then the component of $x(f_i)$ in the subspace of
$\curlyU_R$ graded by $0\in\ZI$ is trivial, since $f_i$ and the
$\beta\in\Psi_0$ are all negative roots.  On the other hand, we can
work out the action of $\x_i(t)$ as follows.    A computation in $\curlyU$ shows
$$
(\ad e_i)(f_i)=-\barh_i,
\quad
{\textstyle\frac{1}{2}}(\ad e_i)^2(f_i)=e_i,
\quad
\hbox{and}
\quad
\textstyle{\frac{1}{n!}}(\ad e_i)^n(f_i)=0
$$
for $n>2$.  Therefore we have
$$
\Ad(\x_i(t))(f_i)
=
\sum_{n=0}^\infty t^n\frac{(\ad e_i)^n}{n!}(f_i)
=
f_i-t\barh_i+t^2 e_i.
$$ 
Recall that
$f_i$, $\barh_i$ and
$e_i$ are three members of a $\Z$-basis for
$\curlyU$.  So their images in $\curlyU_R$ are members of an $R$-basis.
If $t\neq0$ then
the component of $\Ad(\x_i(t))(f_i)$ graded by
$0\in\ZI$ is the nonzero element $-t\barh_i$ 
of $\curlyU_R$.  

Therefore only the trivial element of $\U_i(R)$ maps into the image of
$\U_{\Psi_0}(R)$ in $\Aut\curlyU_R$.
So the same is true with $\St(R)$ in place of $\Aut\curlyU_R$.
From induction and the bijectivity of the product map
$\U_i(R)\times\U_{\Psi_0}(R)\to\U_{\Psi}(R)$ it follows that
$\U_\Psi(R)$ embeds in $\St(R)$.
\end{proof}

\section{$\PSt\to\St$ is often an isomorphism}
\label{sec-PSt-to-St}

\noindent
The purpose of this section is to prove parts
\ref{item-isomorphism-for-3-spherical}--\ref{item-isomorphism-for-2-spherical} of
theorem~\ref{thm-examples-of-PSt-to-St-being-an-isomorphism},
showing that the natural map $\PSt_A(R)\to\St_A(R)$ is an isomorphism
for many choices of generalized Cartan matrix $A$ and commutative ring
$R$.  These cases includes most of part
\ref{item-isomorphism-for-irreducible-affine} of the same theorem;
see \cite{Allcock-affine-Kac-Moody-groups} for the complete result.
And part \ref{item-isomorphism-for-finite-dimensional-type} of the
theorem is the case that $A$ is spherical.  As remarked in
section~\ref{sec-pre-Steinberg-group}, in this case $\PSt_A$ and
$\St_A$ are the same group by definition.  

In the case that $R$ is a field, Abramenko and M\"uhlherr \cite{Abramenko-Muhlherr} proved
our \ref{item-isomorphism-for-2-spherical} with Kac--Moody groups in place
of Steinberg groups.  Our proof of
\ref{item-isomorphism-for-2-spherical} derives from the proof of their
theorem~A; with the following preparatory lemma the argument goes
through in our setting.
For 
\ref{item-isomorphism-for-3-spherical} we use a more elaborate form of
the idea, with lemma~\ref{lem-unipotent-subgroup-generation-rank-3} as preparation.

\begin{lemma}[Generators for unipotent groups in rank~$2$]
\label{lem-unipotent-subgroup-generation-rank-2}
Let $R$ be a commutative ring, $\Phi$ be a rank~$2$
spherical root system equipped with a choice of simple roots, and
$\Phi^+$ be the set of positive roots.  If $\Phi$ has type $A_1^2$ or
$A_2$ then $\U_{\Phi^+}(R)$ is generated by the root groups of the
simple roots.

If $\Phi$ has type $B_2$ then write $\alpha_s$ and $\alpha_l$ for the
short and long simple roots, and $\alpha_{s'}$ (resp.\ $\alpha_{l'}$)
for the image of $\alpha_s$ (resp.\ $\alpha_l$) under reflection in
$\alpha_l$ (resp.\ $\alpha_s$).  Then $\U_{\Phi^+}(R)$ is generated by
$\U_s(R)$, $\U_l(R)$ and either one of $\U_{s'}(R)$ and $\U_{l'}(R)$.  If $R$ has
no quotient $\F_2$ then $\U_s(R)$ and $\U_l(R)$ suffice.

If $\Phi$ has type $G_2$ then, using notation as for $B_2$,
$\U_{\Phi^+}(R)$ is
generated by $\U_s(R)$, $\U_l(R)$ and $\U_{s'}(R)$.  If $R$ has no
quotient $\F_2$ or $\F_3$ then $\U_s(R)$ and $\U_l(R)$ suffice.
\end{lemma}

\begin{proof}
We will suppress the dependence of group functors on $R$, always
meaning groups of points over $R$.  The $A_1^2$ case is trivial
because the simple roots are the only positive roots.

In the $A_2$ case we write $\alpha_i$ and $\alpha_j$ for the simple
roots.  The only other positive root is $\alpha_i+\alpha_j$.  As in
section~\ref{sec-Steinberg-group}, we choose $e_i\in E_i$ and $e_j\in
E_j$.  Then we can use the notation $X_i(t)$, $X_j(t)$ for the
elements of $\U_i$ and $\U_j$, where $t$ varies over $R$.  The
Chevalley relation \eqref{eq-Chevalley-relators-m=3-distant-roots} is
$[X_i(t),X_j(u)]=S_i X_j(t u) S_i^{-1}$.  
Therefore every element of
$S_i \U_j(R) S_i^{-1}$ lies in $\biggend{\U_i(R),\U_j(R)}$.
Since $S_i\U_j S_i^{-1}=\U_{\alpha_i+\alpha_j}$, the proof is
complete.

In the $B_2$ and $G_2$ cases we choose $e_s\in E_s$ and $e_l\in E_l$,
so we may speak of $X_s(t)\in\U_s$ and $X_l(u)\in\U_l$.  We write
$X_{s'}(t)$ for $S_l X_s(t) S_l^{-1}$ and $X_{l'}(t)$ for $S_s X_l(t)
S_s^{-1}$.  In the $G_2$ case we also define $X_{s''}(t)=S_s S_l
X_s(t) S_l^{-1} S_s^{-1}$ and $X_{l''}(t)=S_l S_s X_l(t) S_s^{-1}
S_l^{-1}$.

Rather than mimicking the direct computation of the $A_2$ case, we
use the well-known fact that a subset of a nilpotent group
generates that group if and only if its image in the abelianization
generates the abelianization.  We will apply this to the subgroup of
$\U_{\Phi^+}$ generated by $\U_s\cup\U_l$.  Namely, we write $Q$ for
the quotient of the abelianization of $\U_{\Phi^+}$ by the image of
$\gend{\U_s,\U_l}$.  Under the hypotheses about $R$ having no tiny
fields as quotients, we will prove $Q=0$.  In this case it follows that
$\gend{\U_s,\U_l}$ maps onto the abelianization and is therefore all
of $\U_{\Phi^+}$.  We must also prove, this time with no hypotheses on
$R$, that $\U_{\Phi^+}=\gend{\U_s,\U_l,\U_{s'}}$ and (in the $B_2$
case) that $\U_{\Phi^+}=\gend{\U_s,\U_l,\U_{l'}}$.  This uses the same
argument, with calculations so much simpler that we omit them.  

First consider the $B_2$ case. Among the Chevalley relators defining
$\U_{\Phi^+}$ are \eqref{eq-Chevalley-relators-m=4-orthogonal-short} and \eqref{eq-Chevalley-relators-m=4-distant-short-and-long}, namely
\begin{align*}
[X_s(t),X_{s'}(u)]&{}\cdot X_{l'}(2t u)\\
[X_s(t),X_l(u)]&{}\cdot X_{l'}(-t^2u) X_{s'}(t u) 
\end{align*}
for all $t,u\in R$.  The remaining Chevalley relations say that
various root groups commute with various other root groups.  
Therefore the abelianization of $\U_{\Phi^+}$ is the quotient of the
abelian group $\U_s\times\U_l\times\U_{s'}\times\U_{l'}\iso R^4$ by
the images of the displayed relators.  
We obtain $Q$ by killing the image of $\U_s\times\U_l$.  

So, changing to
additive notation, $Q$ is the quotient of $\U_{s'}\oplus\U_{l'}\iso
R^2$ by the subgroup generated by $0\oplus 2R$ and all $(t u,-t^2u)$,
where $t,u$ vary over $R$.  Taking $t=1$ in the latter shows that
$2R\oplus0$ also dies in $Q$.  So $Q$ is the quotient of $(R/2R)^2$ by
the subgroup generated by all $(t u,-t^2 u)$.  That is, $Q$ is 
(the abelian group underlying) the quotient of $(R/2R)^2$ by the
submodule(!) generated by all $(t,-t^2)$.  This submodule contains
$(1,-1)$, so it is equally well generated by it and all
$(t,-t^2)-t(1,-1)=(0,t-t^2)$.  We may discard the first summand $R/2R$
from the generators and $(1,-1)$ from the relators.  So $Q$ is the
(abelian group underlying) the quotient of $R/2R$ by the ideal $I$
generated by all $t-t^2$.  To prove $Q=0$ we will suppose $Q\neq0$ and
derive a contradiction.  As a nonzero ring with identity,
$R/I$ has some field as a quotient, in which  $t=t^2$ holds identically.  The only field with this property is $\F_2$, which is a
contradiction since we supposed that $R$ has no such quotient.

For the $G_2$ case
the Chevalley relators include
\begin{align*}
[X_l(t),X_{l'}(u)]&{}\cdot X_{l''}(-t u)\\
[X_s(t),X_{s''}(u)]&{}\cdot X_{l'}(-3t u)\\
[X_s(t),X_{s'}(u)]&{}\cdot X_{l''}(3t u^2) X_{l'}(3t^2 u) X_{s''}(2t
u)\\
[X_s(t),X_l(u)]&{}\cdot
X_{l''}(t^3u^2) X_{l'}(-t^3u) X_{s'}(t u) X_{s''}(-t^2 u) \\
% extra one
[X_{s'}(t),X_{s''}(u)]&{}\cdot X_{l''}(-3t u)
\end{align*}
for all $t,u\in R$.  The first four relations are from \eqref{eq-Chevalley-relators-m=6-distant-long}--\eqref{eq-Chevalley-relators-m=6-distant-short-and-long}.  The fifth is the conjugate of
\eqref{eq-Chevalley-relators-m=6-nearby-short} by $S_l$, which commutes with $\U_{s''}$ and sends 
$X_s(t)$ to $X_{s'}(t)$ and $X_{l'}(-3t u)$ to $X_{l''}(-3t u)$, by
their definitions.
All the remaining Chevalley relations say
that various root groups commute with each other.

Proceeding as in the $B_2$ case, $Q$ is the quotient of the abelian
group $\U_{l'}\oplus\U_{s''}\oplus\U_{s'}\oplus\U_{l''}$ by the
subgroup generated by the relators $(0,0,0,\discretionary{}{}{}-t u)$, $(-3t u,0,0,0)$,
$(3t^2u,\discretionary{}{}{}2t
u,\discretionary{}{}{}0,\discretionary{}{}{}3t u^2)$, $(-t^3 u,-t^2 u,
t u, t^3u^2)$ and $(0,\discretionary{}{}{}0,\discretionary{}{}{}0,\discretionary{}{}{}-3t u)$ where $t,u$ vary over $R$.  Because
of the first relator, we may discard the $\U_{l''}$ summand.  This
leads to the following description of $Q$: the quotient of $R^3$ by
the $R$-submodule spanned by the relators $(-3t,0,0)$, $(3t^2 ,2t
,0)$ and $(-t^3,-t^2,t)$, where $t$ varies over~$R$.  Using
$(-1,-1,1)$ in the same way we used $(1,-1)$ in the $B_2$ case shows
that $Q$ is the quotient of $R^2$ by the submodule generated by all
$(-3t,0)$, $(3t^2,2t)$ and $(t^3-t,t^2-t)$.  This is the same as the
quotient of $R/3R\oplus R/2R$ by the submodule generated by all
$(t^3-t,t^2-t)$.  Now, $R/3R\oplus R/2R$ is isomorphic to $R/6R$ by
$(a,b)\leftrightarrow 2a+3b$.  So $Q$ is the quotient of $R/6R$ by the ideal $I$ generated by
$2(t^3-t)+3(t^2-t)$ for all $t$.  As in the $B_2$ case, if $Q\neq0$
then it has a further quotient that is a field $F$, obviously of
characteristic $2$ or~$3$.  In $F$,  either $t^2=t$ holds identically or
$t^3=t$ holds identically, according to these two possibilities.  So
$F=\F_2$ or $\F_3$, a contradiction.
\end{proof}

\begin{lemma}[Generators for unipotent groups in rank~$3$]
\label{lem-unipotent-subgroup-generation-rank-3}
Let $R$ be a commutative ring,  $\Phi$ be a spherical root system of rank~$3$,
$\{\beta_{i\in I}\}$ be simple roots for it, and $\Phi^+$ be the
corresponding set of positive roots.  
Write $s_i$ for the reflection in $\beta_i$, and
for each ordered pair $(i,j)$ of
distinct elements of $I$ write $\gamma_{i,j}$ for $s_i(\beta_j)$. 
Then $\U_{\Phi^+}(R)$ is generated by the $\U_{\beta_i}(R)$ and the
$\U_{\gamma_{i,j}}(R)$. 
\end{lemma}

\begin{proof}
As in the previous proof, we suppress the dependence of group functors
on $R$.  If $\Phi$ is reducible then we apply  the
previous lemma.  So it suffices to treat the cases $\Phi=A_3$, $B_3$
and~$C_3$.  We write $U$ for the subgroup of $\U_{\Phi^+}$ generated
by the $\U_{\beta_i}$ and $\U_{\gamma_{i,j}}$.  We must show that it
is all of $\U_{\Phi^+}$.

For type $A_3$ we describe $\Phi$ by using four coordinates summing to
zero, and take the simple roots $\beta_i$ to be $({+}{-}00)$, $(0{+}{-}0)$ and
$(00{+}{-})$, where $\pm$ are short for $\pm1$.  The  $\gamma_{i,j}$
are the roots $({+}0{-}0)$ and $(0{+}0{-})$.  The only
remaining positive root is $({+}00{-})$.  This is the sum of
$({+}0{-}0)$ and $(00{+}{-})$.  So the $A_2$ case of lemma~\ref{lem-unipotent-subgroup-generation-rank-2} shows
that its root group lies in the $U$.

For type $B_3$ we take the simple roots $\beta_i$ to be $({+}{-}0)$, $(0{+}{-})$
and $(00{+})$.  The  $\gamma_{i,j}$ are
$({+}0{-})$ and $(0{+}0)$.  
The remaining positive roots are
$({+}00)$, $({+}{+}0)$, $({+}0{+})$ and $(0{+}{+})$.  First, 
$(00{+})$, $(0{+}0)$ and $(0{+}{-})$
are three of the four positive roots of a $B_2$ root system
in $\Phi$, including a pair of simple roots for it.  Since  $U$
contains $\U_{00{+}}$, $\U_{0{+}0}$ and $\U_{0{+}{-}}$, lemma~\ref{lem-unipotent-subgroup-generation-rank-2}
shows that $U$ also
contains 
root group corresponding to the fourth positive root, namely $(0{+}{+})$.  
Second, applying the $A_2$ case of that
lemma to $\U_{0{+}{+}},\U_{{+}{-}0}\sset U$ shows that $U$ also
contains $\U_{{+}0{+}}$.  Third, repeating this using
$\U_{{+}0{+}},\U_{0{+}{-}}\sset U$ shows that $U$ contains
$\U_{{+}{+}0}$.  Finally, using the $B_2$ case again, the fact that $U$
contains $\U_{00{+}}$, $\U_{{+}0{-}}$ and $\U_{{+}0{+}}$ shows that
$U$ contains $\U_{{+}00}$.  We have shown that $U$ contains all the
positive root groups, so $U=\U_{\Phi^+}$ as desired.

The $C_3$ case is the same: replacing the short roots $(1,0,0)$, $(0,1,0)$
and $(0,0,1)$ by $(2,0,0)$, $(0,2,0)$ and $(0,0,2)$ does not affect
the proof.
\end{proof}

The next proof uses the geometric language of the Tits cone (or
Coxeter complex), its subdivision into chambers, and the combinatorial
distance between chambers.  Here is minimal background; see
\cite[ch.\ 5]{Remy} for more.  The root system $\Phi$ lies in
$\Z^I\sset\R^I$.  The fundamental (open) chamber is the set of
elements in $\Hom(\R^I,\R)$ having positive pairing with all simple
roots.  We defined an action of the Weyl group $W$ on $\Z^I$ in
section~\ref{sec-Kac-Moody-algebra}, so $W$ also acts on this dual
space.  A chamber means a $W$-translate of the fundamental chamber,
and the Tits cone means the union of the closures of the chambers.
It is tiled by them. 
$W$'s action is properly discontinuous on the interior of this cone.
A gallery of length~$n$ means a sequence of chambers
$C_0,\dots,C_n$, each $C_i$ sharing a facet with $C_{i-1}$ for
$i=1,\dots,n$.  The gallery is called minimal if there is no shorter
gallery from $C_0$ to $C_n$.

To each root $\alpha\in\Phi$ corresponds a halfspace in the Tits cone,
namely those points in it having positive pairing with $\alpha$.  We
write the boundary of this halfspace as $\alpha^\perp$.  
We will identify each root with its halfspace, so we may speak of
roots containing chambers.  In this language, a set of roots is
prenilpotent if there is some chamber lying in all of them, and some
chamber lying in none of them.  

\begin{proof}[Proof of theorem~\ref{thm-examples-of-PSt-to-St-being-an-isomorphism}\ref{item-isomorphism-for-3-spherical}--\ref{item-isomorphism-for-2-spherical}]
We suppress the dependence of group functors on $R$, always meaning
groups of points over~$R$.
Recall that $\St$ is obtained from $\PSt$ by adjoining the Chevalley
relations for the prenilpotent pairs of roots that are not classically
prenilpotent.  So we must show that these relations already hold in
$\PSt$.  For $\Psi$ any nilpotent set of roots we will write $G_\Psi$
for the subgroup of $\PSt$ generated by the $\U_{\alpha\in\Psi}$.
Theorem~\ref{thm-injection-of-unipotent-subgroups} shows that the subgroup of $\St$ generated by these
$\U_\alpha$ is a copy of $\U_\Psi$, so we will just write $\U_\Psi$ for it.

We will prove by induction the following assertion $({\bf N}_{n\geq1})$: Suppose
$C_0,\dots,C_n$ is a minimal gallery, for each $k=1,\dots,n$ let $\alpha_k$ be
the root which contains $C_k$ but not $C_{k-1}$, and define
$\Psi=\{\alpha_1,\dots,\alpha_n\}$
and $\Psi_0=\Psi-\{\alpha_n\}$.  Then $\U_{\alpha_n}$ normalizes
$G_{\Psi_0}$ in $G_\Psi$.  (The $\bf N$ stands for ``normalizes''.  Also,
it is
easy to see that $\Psi$ is the set of all roots containing $C_n$ but
not $C_0$, so it is nilpotent, and similarly for $\Psi_0$.  So
$G_\Psi$ and $G_{\Psi_0}$ are defined.)  

Assuming $({\bf N}_n)$ for all $n\geq1$, it follows that for $\Psi$ of this form,
the multiplication map $\U_{\alpha_1}\times\dots\times\U_{\alpha_n}\to
G_\Psi$ in $\PSt$ is surjective.  We know from lemma~\ref{lem-existence-of-unipotent-groups} and
theorem~\ref{thm-injection-of-unipotent-subgroups} that the corresponding multiplication map in $\St$,
namely $\U_{\alpha_1}\times\dots\times\U_{\alpha_n}\to \U_\Psi$, is
bijective.  Since $G_\Psi\to\U_{\Psi}$ is surjective, it must also be
bijective, hence an isomorphism.  Now, if $\alpha$ and $\beta$ are a
prenilpotent pair of roots then we may choose a chamber in neither of
them and a chamber in both of them.  We join these chambers by a
minimal gallery $(C_0,\dots,C_n)$.  As mentioned above, the
corresponding nilpotent set $\Psi$ of roots consists of all roots
which contain $C_n$ but not $C_0$.  In particular, $\Psi$ contains
$\alpha$ and $\beta$.  We have shown that $G_\Psi\to\U_\Psi$ is an
isomorphism. Since the Chevalley relation of $\alpha$ and $\beta$
holds in $\U_\Psi$ (by the definition of $\U_{\Psi}$), it holds in
$G_\Psi$ too.  This shows that the Chevalley relations of all
prenilpotent pairs hold in $\PSt$, so $\PSt\to\St$ is an isomorphism,
finishing the proof.

It remains to prove $({\bf N}_n)$.  First we treat a special case that does
not require induction.  By hypothesis, $A$ is $S$-spherical, where $S$
is $2$ resp.\ $3$ for part \ref{item-isomorphism-for-2-spherical} resp.\ \ref{item-isomorphism-for-3-spherical} of the theorem.
To avoid degeneracies we suppose $\rk A>S$; the case $\rk A\leq S$
is trivial because then $A$ is spherical and the isomorphism
$\PSt\to\St$ is tautological.
Suppose that all the chambers in some minimal gallery $(C_0,\dots,C_n)$ have
a codimension${}\leq S$ face $F$ in common.  
By $S$-sphericity, the mirrors $\alpha^\perp$ of only finitely many
$\alpha\in\Phi$ contain $F$.  Therefore any pair from
$\alpha_1,\dots,\alpha_n$ is classically prenilpotent.  Their
Chevalley relations hold in $\PSt$ by definition.  The fact that
$\U_{\alpha_n}$  normalizes $G_{\Psi_0}$ in $G_\Psi$ follows from these
relations.  

Now, for any minimal gallery of length $n\leq S$, its chambers have a
face of codimension $n\leq S$ in common.  (It is a subset of
$\alpha_1^\perp\cap\dots\cap\alpha_n^\perp$.)  So the previous
paragraph applies.  This proves $({\bf N}_n)$ for $n\leq S$, which we
take as the base case of our induction.  For the inductive step we
take $n>S$, assume $({\bf N}_1),\dots,({\bf N}_{n-1})$, and suppose
$(C_0,\dots,C_n)$ is a minimal gallery. 
For $1\leq k\leq l\leq n$ we write $G_{k,l}$ for
$\biggend{\U_{\alpha_k},\dots,\U_{\alpha_l}}\sset\PSt$.  
We must show that
$\U_{\alpha_n}$ normalizes $G_{1,n-1}$.

Consider the subgallery $(C_{n-S},\dots,C_n)$ of length~$S$.   These chambers have a codimension-$S$ face $F$ in common. Write
$W_F$ for its $W$-stabilizer, which is finite by $S$-sphericity.  Among
all chambers having $F$ as a face, let $D$ be the one closest to
$C_0$.  By \cite[Prop.\ 5.34]{Abramenko-Brown} it is unique and there
is a minimal gallery from $C_0$ to $C_{n-1}$ having $D$ as one of its
terms, such that every chamber from $D$ to $C_{n-1}$ contains $F$.  By
replacing the subgallery $(C_0,\dots,C_{n-1})$ of our original minimal
gallery with this one, we may suppose without loss of generality that
$D=C_m$ for some $0\leq m\leq n-S$ and that $C_m,\dots,C_n$ all
contain~$F$.  (This replacement may change the ordering on
$\Psi_0=\{\alpha_1,\dots,\alpha_{n-1}\}$, which is harmless.)  The
special case shows that $\U_{\alpha_n}$ normalizes $G_{m+1,n-1}$.  So
it suffices to show that $\U_{\alpha_n}$ also normalizes $G_{1,m}$.

At this point we specialize to proving part \ref{item-isomorphism-for-2-spherical} of the theorem.  In
this case
$F$ has codimension~$2$.  There are two chambers adjacent to $C_m$
that contain $F$.  One is $C_{m+1}$ and we call the other one
$C_{m+1}'$.  We write $\alpha_{m+1}'$ for the root that contains
$C_{m+1}'$ but not $C_m$.  Recall that $C_m$ was the unique chamber
closest to $C_0$, of all those containing $F$.  It follows that
$(C_0,\dots,C_m,C_{m+1}')$ is a minimal gallery.  By a double
application of $({\bf N}_{m+1})$, which we may use because $m\leq
n-S=n-2$, both $\U_{\alpha_{m+1}}$ and $\U_{\alpha_{m+1}'}$ normalize
$G_{1,m}$.  Since
$\alpha_{m+1}$ and $\alpha_{m+1}'$ are simple roots for $W_F$, and
$\alpha_n$ is positive with respect to them, 
lemma~\ref{lem-unipotent-subgroup-generation-rank-2} shows
that $\U_{\alpha_n}$ lies in
$\gend{\U_{\alpha_{m+1}},\U_{\alpha_{m+1}'}}$.  This uses the
hypotheses on $R$ to deal with the possibility that $W_F$ has type
$B_2$ or $G_2$.  Therefore $\U_{\alpha_n}$ normalizes $G_{1,m}$,
completing the proof of part \ref{item-isomorphism-for-2-spherical}.

Now we prove part \ref{item-isomorphism-for-3-spherical}.  $F$ has codimension~$3$.  So there are
three chambers adjacent to $C_m$ that contain $F$.  Write $C_{m+1}'$
for any one of them (possibly $C_{m+1}$) and define
$\beta$ as the root containing $C_{m+1}'$ but not $C_m$.  The three
possibilities for $\beta$ form a system $\Sigma$ of simple roots for $W_F$.
With respect to $\Sigma$, the positive roots of $W_F$ are exactly the ones
that do not contain $C_m$.  For example, $\alpha_n$.

There are two chambers adjacent to $C_{m+1}'$ that contain $F$,
besides $C_m$.  Write $C_{m+2}'$ for either of them
and $\gamma$
for the root containing $C_{m+2}'$ but not $C_{m+1}'$.    Because
$C_m$ is the unique chamber containing $F$ that is closest to $C_0$,
$(C_0,\dots,C_m,C_{m+1}',C_{m+2}')$ is a minimal gallery.  In
particular, $\gamma$ is a positive root with respect to $\Sigma$.

We claim that $\U_\beta$ and $\U_\gamma$ normalize $G_{1,m}$.  For
$\beta$ this is just induction using $({\bf N}_{m+1})$.  For $\gamma$,
we appeal to $({\bf N}_{m+2})$, but all this tells us is that
$\U_\gamma$ normalizes $\gend{\U_\beta,G_{1,m}}$.  In particular, it
conjugates $G_{1,m}$ into this larger group.  To show that $\U_\gamma$
normalizes $G_{1,m}$ it suffices to show for every
$k=1,\dots,m$ that the Chevalley relation
for $\gamma$ and $\alpha_k$ has no $\U_\beta$ term.
That is, it suffices to show that
$\beta\notin\theta(\alpha_k,\gamma)$.  Suppose to the contrary.  Then
$\beta$ is an $\N$-linear combination of $\alpha_k$ and $\gamma$.  So
$\alpha_k$ is a $\Q$-linear combination of $\beta$ and $\gamma$, and
in particular its mirror contains $F$.  Of the Weyl chambers for
$W_F$, the one containing $C_0$ is the same as the one containing
$C_m$, since $C_m$ is closest possible to $C_0$.  Since $\alpha_k$
does not contain $C_0$, it does not contain $C_m$ either.  So, as a
root of $W_F$, it is positive with respect to $\Sigma$.  Now we have
the contradiction that the simple root $\beta$ of $W_F$ is an
$\N$-linear combination of the positive roots $\alpha_k$ and $\gamma$.  This proves $\beta\notin\theta(\alpha_k,\gamma)$, so
$\U_\gamma$ normalizes $G_{1,m}$.

We have proven that $\U_\beta$ and $\U_\gamma$ normalize $G_{1,m}$.
Letting $\beta$ and $\gamma$ vary over all possibilities gives all the
roots called $\beta_i$ and $\gamma_{i,j}$ in lemma~\ref{lem-unipotent-subgroup-generation-rank-3}.  By that lemma, the group
generated by these root
groups contains the root groups of all positive roots of $W_F$.  In
particular, $\U_{\alpha_n}$ normalizes $G_{1,m}$, as desired.  This
completes the proof of \ref{item-isomorphism-for-3-spherical}.
\end{proof}

\section{Finite presentations}
\label{sec-finite-presentations}

\noindent
In this section we prove theorems \ref{thm-finite-presentation-of-pre-Steinberg-groups} and~\ref{thm-finite-presentation-of-Kac-Moody-groups}:
pre-Steinberg groups, Steinberg groups and Kac--Moody groups are
finitely presented under various hypotheses.  Our strategy is to
first prove parts \ref{item-P-S-t-is-f-p-rank-2}--\ref{item-P-S-t-is-f-p-rank-3} of theorem~\ref{thm-finite-presentation-of-pre-Steinberg-groups}, and then prove
part \ref{item-finitely-generated-abelian-group-makes-St-finitely-presented} together with theorem~\ref{thm-finite-presentation-of-Kac-Moody-groups}.

For use in the proof of
theorem~\ref{thm-finite-presentation-of-pre-Steinberg-groups}\ref{item-P-S-t-is-f-p-rank-2}--\ref{item-P-S-t-is-f-p-rank-3},
we recall the following result of Splitthoff, which
grew from earlier work of Rehmann-Soul\'e \cite{Rehmann-Soule}.  Then
we prove theorem~\ref{thm-finite-generation-of-Steinberg-groups}, addressing finite generation rather than
finite presentation, using his methods.  Then we will prove
theorem~\ref{thm-finite-presentation-of-pre-Steinberg-groups}\ref{item-P-S-t-is-f-p-rank-2}--\ref{item-P-S-t-is-f-p-rank-3}. 

\begin{theorem}[{\cite[Theorem I]{Splitthoff}}]
\label{thm-Splitthoff}
Suppose $R$ is a commutative ring and $A$ is one of the ABCDEFG
Dynkin diagrams.  If either
\begin{enumerate}
\item
\label{item-fg-if-finitely-generated}
$\rk A\geq3$ and $R$ is finitely generated as a ring, or
\item
\label{item-fg-if-module-finite-over-finitely-many-units}
$\rk A\geq2$ and $R$ is finitely generated as a module over a subring
generated by finitely many units,
\end{enumerate}
then $\St_A(R)$ is finitely presented.
\qed
\end{theorem}

\begin{theorem}
\label{thm-finite-generation-of-Steinberg-groups}
Suppose $R$ is a commutative ring and $A$ is one of the ABCDEFG
Dynkin diagrams.  If either
\begin{enumerate}
\item
\label{item-Splitthoff-finitely-generated}
$\rk A\geq2$ and $R$ is finitely generated as a ring, or
\item
\label{item-Splitthoff-module-finite-over-finitely-many-units}
$\rk A\geq1$ and $R$ is finitely generated as a module over a subring
generated by finitely many units,
\end{enumerate}
then $\St_A(R)$ is finitely generated.
\end{theorem}

\begin{proof}
In light of Splitthoff's theorem, it suffices to treat the cases
$A=A_2,B_2,G_2$ in \ref{item-fg-if-finitely-generated} and the case
$A=A_1$ in \ref{item-fg-if-module-finite-over-finitely-many-units}.
For \ref{item-fg-if-finitely-generated} it suffices to treat the case
$R=\Z[z_1,\dots,z_n]$, since $\St_A(R)\to\St_A(R/I)$ is surjective for
any ideal~$I$.  In the rest of the proof we abbreviate $\St_A(R)$ to
$\St$.  Keeping our standard notation, $\Phi$ is the root system, and
$\St$ is generated by groups $\U_\alpha\iso R$ with $\alpha$ varying
over $\Phi$.  As discussed in section~\ref{sec-Steinberg-group}, writing down elements
$X_\alpha(t)$ of $\U_\alpha$ requires choosing one of the two
elements of $E_\alpha$, but the sign issues coming from
this choice will not affect the proof.
For each $p\geq1$ we write
$\U_{\alpha,p}$ for the subgroup of $\U_\alpha$ consisting of all $X_\alpha(t)$
where $t\in R$ is a polynomial of degree${}\leq p$.

$A_2$ case: if $\alpha,\beta\in\Phi$ make angle $2\pi/3$ then their
Chevalley relation reads
$[X_\alpha(t),X_\beta(u)]=X_{\alpha+\beta}(\pm t u)$, where the
unimportant sign depends on the choices of elements of $E_\alpha$,
$E_\beta$ and $E_{\alpha+\beta}$.  It
follows that $[\U_{\alpha,p},\U_{\beta,q}]$ contains
$\U_{\alpha+\beta,p+q}$.  An easy induction shows that $\St$ is
generated by the $\U_{\alpha,1}\iso\Z^{n+1}$, with $\alpha$ varying
over $\Phi$.

$B_2$ case: we write $\U_{S,p}$ resp.\ $\U_{L,p}$ for the subgroup of
$\St$ generated by all $\U_{\alpha,p}$ with $\alpha$ a short
resp.\ long root.  
If $\sigma,\lambda$ are short and long roots with
angle $3\pi/4$, then we recall their Chevalley relation 
from \eqref{eq-Chevalley-relators-m=4-distant-short-and-long} as
\begin{equation}
\label{eq-B-2-Chevalley-relation-for-finite-generation-lemma}
[X_\sigma(t),X_\lambda(u)]=
X_{\lambda+\sigma}(-t u) X_{\lambda+2\sigma}(t^2u)
\end{equation}
Here we have implicitly chosen some elements of $E_\sigma$,
$E_\lambda$, $E_{\lambda+\sigma}$ and $E_{\lambda+2\sigma}$ so that
one can write down the relation explicitly. 
Note that the first term on the right lies in a short root group and
the second lies in a long root group.  Recall that $n$ is the number
of variables in the polynomial ring~$R$.  We claim that $\St$ equals
$\langle\U_{S,n},\U_{L,n+2}\rangle$ and is therefore finitely
generated.  The case $n=0$ is trivial, so suppose $n>0$.  Our claim
follows from induction using the following two ingredients.

First, for any $p\geq 1$,  $\langle\U_{S,p},\U_{L,p+2}\rangle$ contains $\U_{S,p+1}$.  To
see this let $g\in R$ be any monomial of degree $p+1$ and write it as
$t u$ for monomials $t,u\in R$ of degrees $1$ and~$p$.  Then 
\eqref{eq-B-2-Chevalley-relation-for-finite-generation-lemma}
yields
$$
X_{\lambda+\sigma}(g)=X_{\lambda+2\sigma}(t^2u)[X_\lambda(u),X_\sigma(t)]
\in\U_{L,p+2}\cdot[\U_{L,p},\U_{S,1}].
$$
Letting $g$ vary shows that
$\U_{\lambda+\sigma,p+1}\sset\langle\U_{S,p},\U_{L,p+2}\rangle$.  Then
letting $\sigma,\lambda$ vary over all pairs of roots making angle
$3\pi/4$, so that $\lambda+\sigma$ varies over all short roots, shows that
$\U_{S,p+1}\sset\langle\U_{S,p},\U_{L,p+2}\rangle$, as desired.

Second, for any $p\geq n$, 
$\langle\U_{S,p+1},\U_{L,p+2}\rangle$ contains $\U_{L,p+3}$.   To see
this let $g\in R$ be any monomial of degree $p+3$ and write it as
$t^2u$ for monomials $t,u\in R$ of degrees $2$ and $p-1$.  This is
possible because $p+3$ is at least~$3$ more than the number of
variables in the polynomial ring~$R$.
Then
\eqref{eq-B-2-Chevalley-relation-for-finite-generation-lemma} can be
written
$$
X_{\lambda+2\sigma}(g)=
X_{\lambda+\sigma}(t u) [X_\sigma(t),X_\lambda(u)]
\in
\U_{S,p+1}\cdot[\U_{S,2},\U_{L,p-1}].
$$ 
Varying $g$ and the pair $(\sigma,\lambda)$ as in the previous
paragraph  establishes
$\U_{L,p+3}\sset\langle\U_{S,p+1},\U_{L,p+2}\rangle$.

$G_2$ case: defining $\U_{S,p}$ and $\U_{L,p}$ as in the $B_2$ case,
it suffices to show that $\St$ equals
$\langle\U_{L,1},\U_{S,n}\rangle$.   The $A_2$ case shows that
$\U_{L,1}$ equals the union $\U_{L,\infty}$ of all the $\U_{L,p}$.  
So it suffices to prove: if $p\geq n$ then $\langle
\U_{L,\infty},\U_{S,p}\rangle$ contains $\U_{S,p+1}$.  
If $\sigma,\lambda\in\Phi$ are short and long simple roots
then their Chevalley relation 
\eqref{eq-Chevalley-relators-m=6-distant-short-and-long} can be written
\begin{equation*}
[X_\sigma(t),X_\lambda(u)]
=
X_{\sigma''}(t^2u)
X_{\sigma'}(-t u)
\cdot(\hbox{long-root-group elements})
\end{equation*}
where $\sigma',\sigma''$ are the short roots $\sigma+\lambda$ and
$2\sigma+\lambda$.  As before, we have implicitly chosen elements of $E_\sigma$,
$E_\lambda$, $E_{\sigma'}$ and $E_{\sigma''}$.
Given any monomial $g\in R$ of degree~$p+1$, by using $p+1>n$ we may
write it as $t^2u$ where $t$ has degree~$1$ and $u$ has degree~$p-1$.
So every term in the Chevalley relation except $X_{\sigma''}(t^2u)$ lies
in $\U_{S,p}$ or $\U_{L,\infty}$.  Therefore
$\langle\U_{S,p},\U_{L,\infty}\rangle$ contains
$X_{\sigma''}(g)$, 
hence $\U_{\sigma''\!,p+1}$ (by varying~$g$), hence
$\U_{S,p+1}$ (by varying $\sigma$ and $\lambda$ so that $\sigma''$ varies
over the short roots).

$A_1$ case: in this case we are assuming there exist units
$x_1,\dots,x_n$ of $R$ and a finite set $Y$ of generators for $R$ as a
module over $\Z[x_1^{\pm1},\dots,\discretionary{}{}{}x_n^{\pm}]$.  We
suppose without loss that $Y$ contains~$1$.  We use the description of
$\St_{A_1}$ from section~\ref{sec-examples},  and write $G$
for the subgroup generated by $S$ and the
$X\bigl(x_1^{m_1}\cdots x_n^{m_n}y\bigr)$ with
$m_1,\dots,m_n\in\{0,\pm1\}$ and $y\in Y$.  
By construction, $G$ contains the $\stilde(x_k^{\pm1})$, and it
contains $\stilde(-1)$ 
since $Y$ contains~$1$.  Therefore  $G$
contains every
$\htilde(x_k^{\pm1})$.   
Relation \eqref{eq-action-of-h-on-X-A1-example} shows
that if $G$ contains $X(u)$ for some $u$, then it also contains
every $X(x_k^{\pm2}u)$.  It follows that $G$ contains every 
$X\bigl(x_1^{m_1}\cdots x_n^{m_n}y\bigr)$ with
$m_1,\dots,m_n\in\Z$.  Therefore $G=\St$.
\end{proof}

\begin{proof}[Proof of theorem~\ref{thm-finite-presentation-of-pre-Steinberg-groups}\ref{item-P-S-t-is-f-p-rank-2}--\ref{item-P-S-t-is-f-p-rank-3}]
We abbreviate $\PSt_A(R)$ to $\PSt_A$.  We begin with \ref{item-P-S-t-is-f-p-rank-2}, so $A$
is assumed $2$-spherical without $A_1$ components, and $R$ is finitely
generated as a module over a subring generated by finitely many units.
We must show that $\PSt_A$ is finitely presented.
Let $G$ be the direct limit of the groups $\PSt_B$ with $B$ varying
over the singletons and {\it irreducible\/} rank~$2$ subdiagrams.  
By $2$-sphericity,  each $\PSt_B$ is isomorphic to the corresponding
$\St_B$.  $G$ is generated by the images of the $\St_B$'s with $|B|=2$,
because every singleton lies in some irreducible rank~$2$ diagram.  
By Splitthoff's theorem, each of these $\St_B$'s is finitely
presented.   And theorem~\ref{thm-finite-generation-of-Steinberg-groups} shows that each $\St_B$ with $|B|=1$
is finitely generated.  Therefore the direct limit $G$ is finitely
presented.  

Now we consider all $A_1A_1$ subdiagrams $\{i,j\}$ of $A$.  For each
of them we impose on $G$ the relations that (the images in $G$ of)
$\St_{\{i\}}$ and $\St_{\{j\}}$ commute.  Because these two groups are
finitely generated (theorem~\ref{thm-finite-generation-of-Steinberg-groups} again), this can be done with
finitely many relations.  This finitely presented quotient of $G$ is
then the direct limit of the groups $\St_B$ with $B$ varying over {\it
  all\/} subdiagrams of $A$ of rank${}\leq2$.  Again using
$2$-sphericity, we can replace the $\St_B$'s by $\PSt_B$'s.  Then
corollary~\ref{cor-pre-Steinberg-as-direct-limit} says that the direct limit is $\PSt_A$.  This finishes
the proof of \ref{item-P-S-t-is-f-p-rank-2}.

Now we prove \ref{item-P-S-t-is-f-p-rank-3}, in which we are assuming
$R$ is a finitely generated ring.  Consider the direct limit of the
groups $\PSt_B$ with $B$ varying over the irreducible rank${}\geq2$
spherical subdiagrams.  Because every node and every
pair of nodes lies in such a subdiagram, this direct limit is the same
as $\PSt_A$.  Because every $B$ is spherical, we may
replace the groups $\PSt_B$ by $\St_B$.  By hypothesis on $A$, $G$ is
generated by the $\St_B$ with $|B|>2$, which are finitely presented by
Splitthoff's theorem.  And
theorem~\ref{thm-finite-generation-of-Steinberg-groups} shows that
those with $|B|=2$ are finitely generated.  So the direct limit is
finitely presented.
\end{proof}

Now we turn to Kac--Moody groups.  For our purposes, Tits' Kac--Moody
group $\G_{\!A}(R)$ may be defined as the quotient of $\St_A(R)$ by
the subgroup normally generated by the relators
\begin{equation}
\label{eq-torus-relators-that-gives-Kac-Moody-group}
\MultInTorus
\end{equation} 
with $i\in I$ and $u,v\in\Runits$.  See \cite[8.3.3]{Remy} or
\cite[\S3.6]{Tits} for the more general construction of $\G_D(R)$ from a root
datum $D$.  In the rest of this section, $\Runits$ will be finitely
generated, and under this hypothesis the choice of root datum has no
effect on whether $\G_{\!D}(R)$ is finitely presented.  (We are using
the root datum which R\'emy calls simply connected \cite[\S7.1.2]{Remy} and
Tits calls ``simply-connected in the strong sense'' \cite[remark\ 3.7(c)]{Tits}.)

The following technical lemma shows that when $\Runits$ is finitely
generated, killing a  finite set of relators \eqref{eq-torus-relators-that-gives-Kac-Moody-group} kills all
the rest too.  The reason it assumes only some of the relations present in
$\PSt_A(R)$ is so we can use it in the proof of
theorem~\ref{thm-finite-presentation-of-pre-Steinberg-groups}\ref{item-finitely-generated-abelian-group-makes-St-finitely-presented}.
There, the goal is to deduce the full presentation of $\PSt_A(R)$ 
from just some of its relations.

\begin{lemma}
\label{lem-only-finitely-many-relations-in-maximal-torus}
Suppose $R$ is a commutative ring and  $r_1,\dots,r_m$ are generators
for $\Runits$, closed under inversion.  Suppose $G$ is the group with generators $S$ and
$X(t)$ with $t\in R$, subject to the relations
\begin{align}
\label{eq-torus-action-on-simple-root-groups-for-finitely-many-units}
\htilde(r)\,X(t)\,\htilde(r)^{-1}
&{}=X(r^2 t)
\\
\label{eq-torus-action-on-negative-simple-root-groups-for-finitely-many-units}
\htilde(r)\,S X(t)S^{-1}\,\htilde(r)^{-1}
&{}=S X\bigl(t/r^2\bigr) S^{-1}
\end{align}
for all $r=r_1,\dots,r_m$ and all $t\in R$, where
$\htilde(r):=\stilde(r)\stilde(1)^{-1}$ and $\stilde(r):=X(r)S
X(1/r)S^{-1}X(r)$.  Then all 
$\calP_{u,v}:=\htilde(u v)\htilde(u)^{-1}\htilde(v)^{-1}$ with
$u,v\in\Runits$ lie in the subgroup of $G$ normally generated by some finite
set of them.
\end{lemma}

\begin{proof}
Define $N$ as the  subgroup of $G$ normally generated by 
the following finite set of $\calP_{u,v}$'s:
\begin{equation*}
%\label{eq-the-finitely-many-torus-relations}
\htilde\bigl(r_k r_1^{p_1}\cdots r_m^{p_m}\bigr)
\cdot
\htilde\bigl(r_1^{p_1}\cdots r_m^{p_m}\bigr)^{-1}
\,
\htilde(r_k)^{-1}
\end{equation*} 
with $k=1,\dots,m$ and $p_1,\dots,p_m\in\{0,1\}$.  We write $\equiv$
to indicate equality modulo~$N$.  As special cases we have
$[\htilde(r_k),\htilde(r_l)]\equiv1$,
$\htilde(r_k^2)\equiv\htilde(r_k)^2$, and that if
$p_1,\dots,p_m\in\{0,1\}$ then $\htilde(r_1^{p_1}\cdots r_m^{p_m})$
lies in the abelian subgroup $Y$ of $G/N$ generated by
$\htilde(r_1),\dots,\htilde(r_m)$.  

We claim that every $\calP_{u,v}$ lies in $Y$.  Since $Y$ is finitely
generated abelian, we may therefore kill all the $\calP_{u,v}$'s by killing some
finite set of them, proving the theorem.  To prove the claim it
suffices to show that every $\htilde(u)$ lies in $Y$, which we do by
induction.  That is, supposing $\htilde(u)\in Y$ we will prove
$\htilde(r_k^2u)\in Y$ for each $k=1,\dots,m$.   
The following calculations in $G$  mimic the proof of \eqref{eq-commutators-of-h-tildes}, paying close attention
to which relations are used.  First, \eqref{eq-torus-action-on-simple-root-groups-for-finitely-many-units}--\eqref{eq-torus-action-on-negative-simple-root-groups-for-finitely-many-units} imply
$\htilde(r_k)\stilde(u)\htilde(r_k)^{-1}=\stilde(r_k^2u)$.  From the
definition of $\htilde(u)$ we get
$\htilde(r_k)\htilde(u)\htilde(r_k)^{-1}=\htilde(r_k^2u)\htilde(r_k^2)^{-1}$.
Right multiplying by $\htilde(u)^{-1}$ yields
$[\htilde(r_k),\htilde(u)]=\calP_{r_k^2,u}$.  Now, $\htilde(u)\in Y$
implies $[\htilde(r_k),\htilde(u)]\equiv1$, so
$\calP_{r_k^2,u}\equiv1$, so
$\htilde(r_k^2u)\equiv\htilde(u)\htilde(r_k^2)\in Y$ as desired.
\end{proof}

\begin{corollary}
\label{cor-finitely-many-relations-in-maximal-torus}
Suppose $R$ is a commutative ring with finitely generated unit group
$\Runits$, and $A$ is any generalized Cartan matrix.  Then the
subgroup of $\PSt_A(R)$ normally generated by all relators
\eqref{eq-torus-relators-that-gives-Kac-Moody-group}
is normally generated by finitely many of them.
\qed
\end{corollary}

\begin{proof}[Proof of theorem~\ref{thm-finite-presentation-of-Kac-Moody-groups}]
We must show that $\G_D(R)$ is finitely
presented, assuming that $\St_A(R)$ is and  that $\Runits$ is finitely generated. 
For $\G_{\!A}(R)$ this is immediate from corollary~\ref{cor-finitely-many-relations-in-maximal-torus}.  Also, its
subgroup $H$ generated by the images of the $\htilde_i(r)$ with $i\in
I$ and $r\in\Runits$ is finitely generated abelian.  For a general
root datum $D$, one obtains $\G_D(R)$ by the following construction.
First one quotients $\G_{\!A}(R)$ by a subgroup of $H$.  Then one takes the
semidirect product of this by a torus $\T$ (a copy of $(\Runits)^n$).  Then
one identifies the generators of $H$ with certain elements of $\T$.  Since
$\Runits$ is finitely generated, none of these steps affects finite presentability.
\end{proof}

\begin{proof}[Proof of theorem~\ref{thm-finite-presentation-of-pre-Steinberg-groups}\ref{item-finitely-generated-abelian-group-makes-St-finitely-presented}]
We must show that if $R$ is finitely generated as an abelian group,
then $\PSt_A(R)$ is finitely presented for any generalized Cartan
matrix~$A$.  Suppose $R$ is generated as an abelian group by
$t_1,\dots,t_n$.  Then $\PSt_A(R)$ is generated by the $S_i$ and
$X_i(t_k)$, so it is finitely generated.  
Because $R$ is finitely generated as an abelian group, its
multiplicative group $\Runits$ is also.  At its heart, this is the
Dirichlet unit theorem.  See \cite[Cor.~7.5]{Lang} for the full
result.  Let $r_1,\dots,r_m$ be a set of generators for $\Runits$,
closed under inversion.

Let $N$ be the central subgroup of $\PSt_A(R)$ normally generated by all relators
\eqref{eq-torus-relators-that-gives-Kac-Moody-group}.  It is elementary and well-known that if a
group is finitely generated, and a central quotient of it is finitely
presented, then it is itself finitely presented.
(See \cite[\S10.2]{Johnson} for the required background.)
Therefore the finite
presentability of $\PSt_A(R)$ will follow from that of $\PSt_A(R)/N$.
The relators defining the latter group are
\eqref{eq-def-of-What-Artin-relators}--\eqref{eq-collapse-to-PSt} and
\eqref{eq-torus-relators-that-gives-Kac-Moody-group}.  We will show that finitely many of them
imply all the others.

In the definition of $\What$, there are only finitely many relations
\eqref{eq-def-of-What-Artin-relators}--\eqref{eq-def-of-What-commutators-of-squares-odd-case}.
The addition rules \eqref{eq-additive-relators-in-root-groups} in
$\U_i\iso R$ can be got by imposing finitely many relations on the
$X_i(t_k)$.  Relations
\eqref{eq-S-i-squared-action-on-X-j}--\eqref{eq-S-i-X-j-relator-when-m-is-6}
describe how certain words in the $S_i$ conjugate arbitrary
$X_j(t)$. By the additivity of $X_j(t)$ in $t$, it suffices to impose
only those with $t$ among $t_1,\dots,t_n$.  The Chevalley relations
\eqref{eq-Chevalley-relator-m=2}--\eqref{eq-Chevalley-relators-m=6-distant-short-and-long}
may be imposed using only finitely many relations, because the Borel
subgroup of any rank~$2$ Chevalley group over $R$ is polycyclic (since
$R$ is).  

Now for the  tricky step:  we impose relations
\eqref{eq-torus-action-1}--\eqref{eq-torus-action-2} for
$r=r_1,\dots,r_m$ and $t=t_1,\dots,t_n$.  The additivity of $X_j(t)$
in $t$ implies these relations for $r=r_1,\dots,r_m$ and arbitrary
$t\in R$.  These are exactly the relations \eqref{eq-torus-action-on-simple-root-groups-for-finitely-many-units}--\eqref{eq-torus-action-on-negative-simple-root-groups-for-finitely-many-units} assumed
in the statement of lemma~\ref{lem-only-finitely-many-relations-in-maximal-torus}.  That lemma  shows that we may
impose all the relations \eqref{eq-torus-relators-that-gives-Kac-Moody-group} by imposing some finite number of
them.  
Working modulo these, $\htilde_i(r)$ is multiplicative in $r$, for
each $i$.  Therefore our relations
\eqref{eq-torus-action-1}--\eqref{eq-torus-action-2} for
$r=r_1,\dots,r_m$ imply the same relations for arbitrary~$r$.  

Starting with the generators $S_i$, $X_i(t)$ with $i\in I$
and $t=t_1,\dots,t_n$, we have found finitely many relations
from \eqref{eq-def-of-What-Artin-relators}--\eqref{eq-collapse-to-PSt} and \eqref{eq-torus-relators-that-gives-Kac-Moody-group} that imply all
the others.  Therefore $\PSt_A(R)/N$ is finitely presented, so the
same holds for $\PSt_A(R)$ itself.
\end{proof}

\end{document}